\documentclass{article}
\usepackage[utf8]{inputenc}
\usepackage[margin=1in]{geometry}
\usepackage{amsmath, amsfonts, amsthm, amssymb}
\usepackage{comment}
\usepackage{color}
\usepackage{enumerate}
\usepackage{mathtools}
\usepackage{tikz-cd}

\usepackage[style=alphabetic, sortcites]{biblatex}
\newtheorem{theorem}{Theorem}
\newtheorem{proposition}{Proposition}

\newtheorem{remark}{Remark}
\newtheorem{lemma}{Lemma}
\theoremstyle{definition}
\newtheorem{definition}{Definition}
\newcommand{\C}{\mathbb{C}}
\newcommand{\F}{\mathbb{F}}
\newcommand{\R}{\mathbb{R}}
\newcommand{\Z}{\mathbb{Z}}

\newcommand{\M}{\mathcal{M}}
\newcommand{\cM}{{\overline{\M}}}

\newcommand{\Q}{\mathbb{Q}}

\DeclareMathOperator{\Bl}{Bl}
\newcommand{\vdim}{\text{vdim }}
\newcommand{\tensor}{\otimes}
\newcommand{\CP}{{\mathbb{C}\mathbb{P}}}
\renewcommand{\P}{{\mathbb{P}}}
\newcommand{\Grad}{\nabla}
\newcommand{\oo}[1]{\mathring{#1}}
\title{Integral Arnol'd Conjecture}
\author{Semon Rezchikov}
\date{February 2022}

\begin{document}
	
	\maketitle
	
	\section{Introduction}
	Let $(M^{2n}, \omega)$ be a compact symplectic manifold. The letter $J$ will always denote some compatible almost-complex structure. We will fix a time-dependent Hamiltonian $H = H_t$, $t \in [0,1]$, which generates a nondegenerate Hamiltonian diffeomorphism $\phi_H$. 
	In this document, we prove the following 
	\begin{theorem}[Integral Arnold Conjecture]
		\label{thm:integral-arnold}
		The number of fixed points $|Fix(H)|$ of a nondegenerate Hamiltonian diffeomorphism $\phi_H$ is bounded below by the minimal rank of a cochain complex of free $\Z$ modules with cohomology isomorphic to $H^*(M, \Z)$. 
	\end{theorem}
	The method is to set up a compatible system of global charts for Hamiltonian Floer Homology via an adaptation of \cite{AMS}, and then to apply the recent results involving FOP perturbations \cite{bai_xu} to extract differentials for the Floer complex with integer values. This is a stronger result than the result established in \cite{abouzaid-blumberg}; the latter only proves the Arnol'd conjecture for all primes $p$, and so in particular fails to capture the interaction of $p$-torsion information for distinct primes simultaneously. 
	
	The general idea of using perturbations similar to the form used in this paper is due to McLean \cite{mclean_secret}. The technical contributions of this paper are a convenient two-form which makes it easy to make sense of McLean stabilization in the context of Hamiltonian Floer Homology (Section \ref{sec:integral-two-form}), as well as a language in which to express how to compatibly choose perturbation data such that the various global charts can be  \emph{compatibly smoothed} and such that the resulting Kuranishi sections can be \emph{inductively perturbed}.  The author intends to use these charts for a somewhat more sophisticated result in an upcomping paper; this paper is written first because it is an easier result, in order to keep documents short, and to keep separate ideas into separate papers. 

    Much of the discussion is centered around around about the notion of a \emph{$G-\langle k \rangle$-manifold} and a \emph{derived $\langle k \rangle$-orbifold}. Regular Morse-theoretic moduli spaces (including those involved in more complex operations beyond the differential) are naturally smooth $\langle k \rangle$-manifolds,  and in the general case, we expect that the  (compatible) global charts constructed in this paper should naturally give the moduli spaces underlying Floer-theoretic operations with multiple inputs and outputs the structure of derived $\langle k \rangle$-orbifolds. A $\langle k \rangle$-manifold is a type of manifold with corners to which one can apply the \emph{doubling} trick -- the natural generalization to manifolds with corners of the standard realization of a manifold-with-boundary as a fundamental domain for a $\Z/2$ action on a closed manifold. Passing to the double allows us to avoid much of the general complexity of working with manifolds or orbifolds with corners, and in particular to smooth the moduli spaces using a standard application of equivariant smoothing theory on the doubled spaces. Compatibility is achieved by a generalization of uniqueness of stable equivariant smoothing up to concordance, together with taking collars, which is also a particularly nice operation on $\langle k \rangle$-manifolds.  We thus hope that the combinatorial  language and group-theoretic constructions of this paper will be helpful to others working in pseudoholomorphic curve theory.
    
     \begin{remark}The author was informed shorly before making this paper public that the forthcoming independent paper of Bai-Xu \cite{} has also achieved a proof of this conjecture, with 
     analogous but distinct techniques. We hope that the geometric and combinatorial constructions of our paper can provide a shortcut to using these techniques for those interested in flexibly applying these tools to various problems in Floer theory. \end{remark}
    
    \paragraph{Acknowledgements.} We thank Egor Shelukhin for his encouragement and commentary when the author was trying to understand \cite{AMS}. We also thank Mohammed Abouzaid for answering the author's questions about \cite{AMS}. Finally, we thank Shayoun Bai and 
    Guangbo Xu for communication regarding their paper.

	\section{Basic Notions}
	\subsection{Floer homology conventions}
	We write 
	\begin{equation}
		S^1 = [0,1]/(0 \sim 1);  Z = \R_s \times S^1_{t}; 
		D^2 = \{ z \in \C; |z|\leq 1\}; \partial D^2 \simeq S^1 \text{ via }  z= e^{2 \pi i t}. 
	\end{equation}
	We have $i \partial/\partial_s = \partial/\partial_t$. 
	
	The Hamiltonian vector field $X_H$ associated to a Hamiltonian function $H: M \to \R$ satisfies $\omega(X_H, v) = -dH(v)$. This convention reproduces the flow $X_H = p (\partial/\partial q)$ of the standard free-particle Hamiltonian if we choose $H: \R^2_{p, q} \to \R$ to be $\frac{1}{2}p^2$ with $\omega = dp dq$. If $H$ depends on an auxiliary time variable then we continue to write $X_H$ for the associated time-dependent Hamiltonian vector field. 
	
	A capping of a free loop $x: S^1 \to M$ is a map $\tilde{x}: D^2 \to M$ which restricts to $x$ on $S^1 = \partial D^2$. Two cappings are equivalent if the functional \[\bar{x} \mapsto \int_{\bar{x}} \omega\] agrees on the cappings. 
	\begin{remark}
	There is another variant of this equivalence relation, where one also requires that the functional 
	\[ \bar{x} \mapsto \int_{\bar{x}} \tilde{c}_1(M)\]where $\tilde{c}_1(M)$ is any differential form representing the real first Chern class $c_1^{\R}(M) \in H^2(M, \R)$ agrees on the cappings. This allows one to define integrally graded Floer groups, which we will avoid in this paper. 
	\end{remark}
	
	Let $LM$ be the free loop space of $M$, and let $
	\widetilde{LM}$ be its Novikov cover. Elements of the latter are elements $x \in LM$ equipped with equivalence classes $\bar{x}$ of cappings of the loops $x$.  The fiber over a loop $x$ is naturally a torsor over the group 
	\begin{equation}
	    \label{eq:novikov-group}
	    \Pi = \frac{\pi_2(M)}{\ker [\omega]}
	\end{equation} 
	via which acts on the fiber via connect sum $(A, \bar{x}) \mapsto \bar{x} \# A.$ 
	
	We write $N$ for minimal Chern number of $M$, that is, the positive generator of the subgroup $c_1(\Pi) \subset \Z$. 
	
    We define $\Lambda^\Pi_\Z$ to be completion of the integral group ring $\mathbb{Z}[\Pi]$ with respect to the valuation defined by $\omega$. 
    
    We define the \emph{universal Novikov ring} $\Lambda^{univ}_\Z$ to be the set 
    \begin{equation}
        \Lambda^{univ}_\Z = \{\sum_i c_i T^{r_i} | c_i \in \Z, r_i \in \R, \#\{c_i \neq 0 | r_i < R\} < \infty \text{ for all } R \in \R\}. 
    \end{equation}
    We write $\Lambda^{univ}_{\Z,0}$ to consist of sums $\sum_i c_i T^{r_i}$ with $r_i \geq 0$ for all $i$, and $\Lambda^{univ}_{\Z,+}$ to consist of such sums with $r_i > 0$ for all $i$. 
    
	
	We will write $Fix(H) = \{x: S^1 \to M; x'= X_{H}, [x] = [pt] \in \pi_0(LM)\}$, and will write $\widetilde{Fix(H)}$ for the set of equivalence classes of cappings of elements of $Fix(H)$.
	
	Given $\bar{x}_\pm \in \widetilde{Fix}(H)$, we write $\M(\bar{x}_-, \bar{x}_+)$ for the set of Floer trajectories 
	\begin{equation}
		\label{eq:floer-moduli-space}
		\{u: Z \to M :  \partial_su + J (\partial_tu - X_{H}) = 0, \lim_{s \to \pm \infty} u(s, \cdot) = x_\pm(\cdot), \bar{x}_- = \tilde{x}_+ \# u . \}
	\end{equation}
	Here the last line denotes the concatenation action of cylinders on cappings. The Floer compactifications of these spaces are denoted by $\cM(\tilde{x}_-,\tilde{x}_+)$; elements of $\cM(\tilde{x}_-,\tilde{x}_+)$ are sequences of Floer trajectories with pseudoholomorphic bubble trees attached.

	For any given Floer trajectory, as $s$ increases, the family of loops $u(s, \cdot)$ traces out \cite{salamon-zehnder} a negative gradient trajectory of the action functional $\mathcal{A}: \widetilde{LM} \to \R$  given by  
	\begin{equation}\label{eq:action-functional}\mathcal{A}_{H}(x) = \int_0^1 H(t, x(t))\,dt - \int_{\bar{x}} \omega.
	\end{equation}
	
	The Floer cochain groups $CF^k(H, J;\Z)$ consist, as abelian groups, of infinite sums $\sum_i a_i \bar{x_i}$ with $a_i \in \F$, all $\bar{x}_i$ of Conley-Zehnder index $CZ(\tilde{x}_i) = k \;(\mod 2N)$, and such that for any $C >0$,  the set $\{i| a_i \neq 0, \omega(\bar{x}_i) < C\}$ is finite for all $C \in \R$. Our conventions on Conley-Zehnder indices are such that if we take a constant capping of a fixed point the Hamiltonian flow of a $C^2$-small morse function, its Conley-Zehnder index is equal to its index as critical point (modulo $2N)$. To compare with the conventins of \cite{SalamonZehnder}, we have that $CZ(\tilde{x}) = \mu_\tau(\tilde{x}) + n$ (for $\mu_\tau$ as in \cite{SalamonZehnder}). Whenever $\# Fix(\phi_H) < \infty$, these are finitely generated free modules over $\Lambda^{\Pi}_\Z$ of dimension equal to the number of elements of $Fix(\phi_H)$ admitting a cap of Conley-Zehnder index equal to $k$ modulo $2N$; a basis can be obtained by fixing an equivalence class of such a cap for each such periodic orbit. The task of this paper is to define a differential on this complex which is equivariant with respect to the operation of $\Pi$ on these groups, and show that its quasi-isomorphism type agrees with the $2N$-periodized Morse cochain complex of $M$ with coefficients in $\Lambda^\Pi_\Z$.

	\subsection{Constructions on manifolds with corners}
	In this section we will discuss the notion of a $\langle k \rangle$-manifold, which is a particularly convenient type of manifold with corners. The global Kuranishi charts we will produce for Floer trajectories will be $\langle k \rangle$-manifolds, and the extra structure will record the decomposition of the underlying Floer domain into a sequence of $1$-level Floer domains. 
	
	\subsubsection{Manifolds with corners.}
	A stratified topological space $X$ is a filtration on a space $X$ by closed subsets 
	$X = X(d) \supset X(d-1) \supset \ldots \supset X(0) \supset X(-1) = \emptyset$. The \emph{open stratum} associated to $k = 0, \ldots, d$ is $X(k) \setminus X(k-1) = \oo{X}(k)$.
	
	More generally, we can consider a topological space $X$ stratified by a poset $\mathcal{P}$. The data of such a stratification is an assignment, for each $p \in \mathcal{P}$, of a closed subset $X(p) \subset X$, such that if $p \leq q$ then $X(p) \subset X(q)$. If $\mathcal{P}$ is a finite poset, then we can define the open stratum associated to $p$ to be \[\oo{X}(p) = X(p) \setminus \left(\bigcup_{q < p} X(q)\right).\] This notion specializes to the previous one when we take $\mathcal{P}= \{ 0 < 1 < \ldots < d\}.$ Given $S \in \mathcal{P}$, we will refer to the maximal length of a chain of strictly increasing elements of $\mathcal{P}$ starting with $S$ as the \emph{codimension of the stratum} $X(S)$.
    
	Write $\R^d_+$ for the closed positive orthant in $\R^d$. This is space stratified by $0< \ldots < d$ with $\R^d_+(i)$ the subset with at least $d-i$ coordinates equal to zero. A topological manifold of dimension $d + \ell$ with codimension-$d$ corners is a second-countable Hausdorff space $X$ stratified by $0 < \ldots < d$ such that for every $x \in X$, there is an open neighborhood $X \supset V \ni x$ and a homeomorphism called a \emph{chart}
	\begin{equation}
	    \psi: \R^d_+ \times \R^\ell \supset U \to V \subset X
	\end{equation}
	such that $\psi(U \cap \R^d_+(i) \times \R^\ell) \subset X(i)$, and such that $\oo{X}(d) \neq \emptyset$.  A topological manifold with codimension-$d$ corners is countable disjoint union of the above objects where $\ell$ is allowed to vary. We will drop ``codimension-$d$'' if we do not need to specify $d$, although $d$ is always finite. We note that topological manifolds are topological manifolds with codimension $d$ corners for all $d \geq 0$. 

    A smooth manifold with codimension $d$ corners is a topological manifold with corners equipped with a maximal atlas of charts to $\R^d_+ \times \R^\ell$ with smooth transition functions between such charts.

    \subsubsection{$\langle k \rangle$-manifolds}
    \label{sec:k-manifolds}
    Let $[k] = \{1, \ldots, k\}$ be the set with $k$ elements. (We set $[0] = \emptyset$.) We can think of elements of $2^{[k]}$ (note that $2^{\emptyset}$ is a set of a single element) as subsets $S \subset [k]$; alternatively, we can think of them corresponding to partitions 
    $P \in \mathcal{P}_{k+1}$ of the ordered sequence $1, \ldots, k+1$. The bijection is as follows: if $[n] \setminus S = \{s'_1 < \ldots < s'_j\} \subset [n]$, then the corresponding partition $P(S)$ is $\{\{1,\ldots, s'_1\}, \{s'_1 +1 , \ldots, s'_2\}, \ldots, \{s'_j+1, \ldots, k+1\}\}$. We write $h(S) = (s'_1, s'_2-s_1, \ldots, n+1-s'_j)$ for the sequence of lengths of elements of $P(S)$, which of course determines the partition $P(S)$ itself. To clarify any ambiguity in the notation, we have that if $S = [k]$ then $P(S) = \{\{1, \ldots, n+1\}\}$, and $h(S) = (k+1)$. We will write $S(h)$ for the subset $S \subset [k]$ associated to the sequence of lengths $h = (h_1, \ldots, h_{j+1})$ of a partition $P \in \mathcal{P}_{k+1}$. We will also write $|S|$ for the cardinality of a subset $S \subset [k]$, and $|h| = |S(h)|+1$ for the length of a sequence $h =h(S)$.

    Now $2^{[k]}$ is a poset; we will say that $S \leq T$ if the partition $P(S)$ refines the partition $P(T)$, or equivalently that $S \subset T$.
    
    For each $S \in 2^{[k]}$, we have a closed subset $\R^k_+(S) \subset \R^k_+$ consisting of those vectors for which the coordinates with indices in $[k]\setminus S$ are forced to be zero. Thus, the partition $\{\{1, \ldots, k+1\}\} = P([k])$ corresponds to $\R^k_+$, while the partition $\{\{1\}, \ldots, \{n+1\}\} = P(\emptyset)$ corresponds to $\{0\} \subset \R^n_+$. 
    
    More generally, given a finite set $R$, we have the space $\R^{R}_+$, stratified by $2^{R}$ with $\R^{R}_+(T)$, for $T \subset R$, given by the subset with coordinates in $R \setminus T$ required to be zero. 
    
    Let $X$ be a space stratified by $2^{[k]}$. For $x \in X$ we write $S(x) = \min \{S \in 2^{[k]} \, | x in X(S)\}$, and we define $\text{sdim } X = |S(x)|$. Given $S \in 2^{[k]}$ write $[k] \setminus S = \{r'_1, \ldots, r'_\ell\}$. Then there is a unique isomorphism of posets 
    \begin{equation}
    	r_S: \{T \subset [k]\; | \: T \supset S\} \to 2^{[k \setminus |S|]} \text{ such that } r_S(S \cup \{s'_j\}) = \{j\}.
    \end{equation}

    A \emph{$n$-dimensional $\langle k \rangle$-manifold chart} (centered at $x \in X$) is a homeomorphism 
    \begin{equation}
    	\psi: \R^{[k] \setminus S}_+  \times \R^{n - |k \setminus S|} \supset U \to V \subset X
    	\end{equation}
    from an open subset $U \ni (0, 0)$ to an open subset $V$ such that
    \begin{itemize} 	
    	\item For $y \in V$, we have that $S(y) \geq S(x)$, and moreover
    	\item For $[k] \supset T \supset S$, have that $\psi^{-1}(X(T)) \subset \R^{[k] \setminus S}_+(r_S(T)) \times \R^{n - |k \setminus S|}$; and finally
    	\item we say that the chart $\psi$ is centered at $x$ if $\psi(0, 0) = x$. 
    \end{itemize}
    A  \emph{topological $\langle k \rangle$-manifold} of dimension $n$ is a second-countable Hausdorff space $X$ stratified by $2^{[k]}$ such that for every $x \in X$, there exists an $n$-dimensional $\langle k \rangle$-manifold chart centered at $x$, and such that $\oo{X}([k]) \neq \emptyset$. A \emph{smooth} $\langle k \rangle$-manifold is a topological $\langle k \rangle$-manifold equipped with a maximal subset (an \emph{atlas})of $n$-dimensional $\langle k \rangle$-charts such that transition maps between charts are smooth. 
    
    We will drop the dimension $n$ when it is implied, and declare a countable disjoint union of topological $\langle k \rangle$-manifolds of varying dimension to be a topological $\langle k \rangle$-manifold. The quantity $k$ can sometimes be a generic parameter; thus we may say that $\R_+^1, \R_+^2 \ldots$  are all $\langle k \rangle$-manifolds. When $k$ is a specific quantity, this should be clear from the context.

    \begin{remark}
    A $\langle k \rangle$-manifold is, with this definition, naturally a $\langle k +1\rangle$-manifold in many diferent ways, corresponding to the inclusions of posets $2^{[k]} \subset 2^{[k+1]}$
    induced from the $k+1$ distinct inclusions of ordered sets $[k] \subset [k+1]$. This construction forces certain strata, e.g. the deepest stratum $X(\emptyset)$, to be empty, which is allowed. A $\langle k+1 \rangle$-manifold, however, is never a $\langle k \rangle$-manifold. 
    \end{remark}
    
    Let $G$ be a group. A topological $G$-$\langle k \rangle$-manifold is a $\langle k \rangle$-manifold with an action of $G$ on the underlying space such that the action of $G$ preserves each stratum. A smooth $G$-$\langle k \rangle$-manifold is a topological $G$-$\langle k \rangle$-manifold with underlying $\langle k \rangle$-manifold equipped with a smooth structure, such that the action of $G$ on the corresponding manifold with corners is smooth. In particular, taking $k=0$, we recover the notions of topological and smooth $G$-manifolds. 
    
    We will write $G_x$ for the stabilizer of $x \in X$ of the action of $G$ on a $G$-space $X$. We may also write $Stab(x)$ when the group $G$ is clear from the context.
    
    \paragraph{Transversality on $\langle k \rangle$-manifolds.}
    Let $V$ be a vector bundle over a smooth $\langle k \rangle$-manifold $X$, which just means a smooth vector over the underlying manifold with corners, and let $s$ be a section of $V$. We say that $s$ is \emph{transverse} (along $K \subset X$) if, writing $s_V:  U \to \R^{\dim V}$ for the expression defining $s$ in some chart $X \supset V \xrightarrow{\psi^{-1}} U \subset \R^{[k] \setminus S}_+  \times \R^{n-|[k] \setminus S|}$ and some $S \in 2^{[k]}$, we have that $ds_{U_i}$ is surjective on the image of $K \cap s^{-1}(0) \cap V_i$ in $U_i$ for a collection of charts $\psi: U_i \simeq V_i$ with codomains $V_i$ covering $K$. Now, every substratum $X(S)$ is naturally a $\langle k'\rangle$-manifold for some $k' \leq k$, and we can ask for $s|_{X(S)}$ to be transverse along $K \cap X(S)$. It is possible for $s|_{X(S)}$ not to be transverse along $K \cap X(S)$, even when $s$ is transverse on $s^{-1}(0)$ in $X$. We say that $s$ is \emph{strongly transverse} (along $K$) when $s|_{X(S)}$ is strongly transverse (along $K \cap \cap X(S)$) for all $S \in 2^{[k]}$. Note that if $s|_{X(S)}$ is strongly transverse, then $s$ is strongly transverse on an open neighborhood of $X(S)$ in $X$. 
    
    \paragraph{Submersions and immersions of $\langle k \rangle$-manifolds}
    A submersion of $\langle k \rangle$-manifolds is a stratum-preserving smooth map $f: M \to N$ between $\langle k \rangle$-manifolds $M$ and $N$ for which the differential is surjective everywhere. More generally:
    \begin{definition}
        \label{def:general-submersion-of-k-manifolds}
    Given a $\langle k \rangle$-manifold $M$ of dimension $m$ and a $\langle k'\rangle$-manifold $N$ of dimension $n$ with $k' \leq k$, a submersion from $M$ to $N$ is: 
    \begin{itemize}
        \item Let $\mathcal{P}\subset 2^{[k]}$, $\mathcal{P}' \subset 2^{[k']}$ be the sub-posets of nonempty strata of $M$ and $N$, respecitvely. We fix a subset $S_f \subset [k]$ and a map of sets $\tilde{f}: [k] \setminus S_f \to [k']$ such the formula $S \mapsto [k'] \setminus \tilde{f}([k] \setminus (S \cup S_f))$ defines a map of posets $\bar{f}: \mathcal{P} \to \mathcal{P}'$ with $|[k] \setminus (S \cup S_f)| = |[k'] \setminus \bar{f}(S)|$ for all $S \in \mathcal{P}$; and furthermore,
        \item we have a smooth map $f: M \to N$ which is locally modelled on the map
        \begin{equation}
        \R^{[k]\setminus S}_+ \times \R^{n-k+|S|} =  \R^{([k] \setminus S) \cap S_f}_+ \times \R^{[k] \setminus (S \cup S_f)}_+ \times \R^{n-k+|S|} \to \R^{[k'] \setminus \bar{f}(S)}_+ \times \R^{m-k'+\bar{f}(S)}
        \end{equation}
        which is a restriction of the linear map $\R^n \to \R^m$ 
        which is zero on $\R^{([k] \setminus S) \cap S_f}_+$ (the ``fiber corner directions''), and is otherwise the sum of a surjective map $\R^{n-k+|S|} \to \R^{m-k'+\bar{f}(S)}$ and an isomorphism $\R^{[k] \setminus (S \cup S_f)}_+ \to \R^{[k'] \setminus \bar{f}(S)}_+$ induced by applying $f$ to the coordinate labels. The fibers of such a map are manifestly $\langle S_f\rangle$-manifolds. 
        \end{itemize}
    \end{definition}
    \begin{remark}
    We will only use the case where $S_f=\emptyset$ and the case where $S_f = \{2, \ldots, k\}$ and $k'=1$ in this paper, for which the definition simplifies drastically. 
    \end{remark}

    An immersion of $\langle k \rangle$-manifolds from an $m$-dimension $\langle k \rangle$-manifold $M$ to an $n$-dimensional $\langle k \rangle$-manifold $N$ is a stratum-preserving smooth map from $M$ to $N$ locally modelled on the inclusions $\R^{[k] \setminus S}_+ \times \R^m \to \R^{[k] \setminus S}_+ \times \R^n$ for $m < n$.
    
    We now describe two natural operations on $\langle k \rangle$-manifolds.

    \subsubsection{The double of a $\langle k \rangle$-manifold}
    \label{sec:doubling}
    Just for this paragraph, let $M$ be a topological $\langle 1 \rangle$-manifold, i.e. a topological manifold with boundary. The \emph{double} $\mathcal{D}(M)$ is the space $M \cup_{\partial M} M$; this is a topological manifold equipped with an inclusion $M \to \mathcal{D}$ which is a homeomorpism onto its image. 
    
    Now let $M$ be a $\langle k \rangle$-manifold. The group $(\Z/2)^k$ acts on $\R^k$ by flipping signs, and each $S \in 2^{[k]}$ defines a subgroup $(\Z/2)^k_S= (\Z/2)^{[k]\setminus S}$ which preserves the subset $\R^k_+(S) \subset \R^k$ under this action, and does not preserve any $\R^k_+(T)$ for $T \supsetneq S$. The double $\mathcal{D}(M)$ of $M$ is the union 
    \begin{equation}
    \label{eq:double}
        \mathcal{D}(M) := \bigsqcup_{g \in (\Z/2)^k} g M/ \sim
    \end{equation}
    with the equivalence relation being that for each $S \in 2^{[k]}$, we identify 
    $gx \sim g'x$ for $x \in M(S)$ and $g, g' \in (\Z/2)^k_S$. We have a canonical inclusion $\iota_{\mathcal{D}}: M \to \mathcal{D}$ and a canonical action of $(\Z/2)^k$ on $\mathcal{D}(M)$. The operation of doubling is a functor: given a map $f: M \to N$ of topological $G-\langle k \rangle$-manifolds, there is a unique extension of $f$ to a map of topological $G\times (\Z/2)^k$-spaces $\mathcal{D}f: \mathcal{D}(M) \to \mathcal{D}(N)$. By applying this functor to charts, one immediately concludes
    
    \begin{lemma}
        The space $\mathcal{D}(M)$ is a topological $(\Z/2)^k$-manifold. If $M$ was a $G-\langle k \rangle$-manifold then there is an action of $G$ on $\mathcal{D}(M)$ induced by acting separately on each piece of the union in (\ref{eq:double}) which makes $\mathcal{D}(M)$ into a topological $G \times (\Z/2)^k$-manifold. $\qed$
    \end{lemma}

    \subsubsection{Collaring of $\langle k \rangle$-manifolds}
    Let us again assume for just this paragraph that $M$ is a topological manifold with boundary. Then $Coll(M) = \partial M \times [0,1] \cup M /((x, 1) \sim x)$ is again a topological manifold with boundary, which is homeomorphic to $M$. Moreover, if $M$ is smooth then $Coll(M)$ is also canonically smooth and diffeomorphic to $M$.
    
    A generalization of this construction exists for $\langle k \rangle$-manifolds $M$. As a space, we have
    \begin{equation}
    \label{eq:collar}
        Coll(M) = \bigcup_{S \in 2^{[k]}} M(S) \times[0,1]^{[k] \setminus S}/\sim
    \end{equation}
    with the equivalence relation is generated by the relations that if $(x, t) \in M(S) \times [0,1]^{[k]\setminus S}$ and $(y, u) \in M(T) \times [0,1]^{[k \setminus T]}$, then $(x, t) ~(y, u)$ if $S <T$ and $x=y$, and if $t|_{([n] \setminus S) \setminus ([n] \setminus T)} = 1$. One can use the charts on a $\langle k\rangle$-manifold to immediately check that 
    \begin{lemma}
        The space $Coll(M)$ is a topological $\langle k\rangle$-manifold with $Coll(M)(S) = M(S) \times 0$. If $M$ was a topological $G$-$\langle k \rangle$-manifold then the $G$ action on $Coll(M)$ descending from the $G$-action on each component of the union in \eqref{eq:collar} via $g\cdot(x,t) = (gx, t)$, is continuous, making $Coll(M)$ into a topological $G-\langle k \rangle$-manifold. $\qed$
    \end{lemma}
    
    Given a map $f: M \to N$ of $\langle k \rangle$-manifolds, we get an associated map $Coll(f): Coll(M) \to Coll(N)$ of $\langle k \rangle$-manifolds, defined via 
\begin{equation}
    Coll(f)|_{M(S) \times [0,1]^{[k] \setminus S}} = f|_{M(S)} \times \text{id}. 
\end{equation}
    
    We have the following useful
    \begin{lemma}
    \label{lemma:collars-and-doubles-are-smooth}
        If $M$ is a smooth $G-\langle k \rangle$-manifold, then $Coll(M) \setminus M$ admits a canonical smooth structure and a smooth embedding of $G$-$\langle k \rangle$-manifolds $Coll(M) \setminus M \to M$ that is canonical up to isotopy of embeddings. In particular, $\mathcal{D}(M)$ is a smooth $G\times(\Z/2)^k$-manifold. 
    \end{lemma}
    \begin{proof}
        This follows immediately from Proposition 4.1 and Theorem 4.2 of \cite{AlbinMelrose}. Indeed, the boundary faces of a $G$-$\langle k \rangle$ manifold have a natural partition into $k$ $G$-invariant collective boundary hypersurfaces satisfying \cite[Equation~4.2]{AlbinMelrose}. The doubling construction in that paper agrees with our description. The smooth embedding of $Coll(M) \setminus M$ into $M$ is exactly the data of the $G$-invariant product structure of Proposition 4.1 of \cite{AlbinMelrose}.
    \end{proof}
    
    As a corollary of the above result we see that if $f: M \to N$ is a submersion of smooth $\langle k \rangle$-manifolds then $\mathcal{D}(f): \mathcal{D}(M) \to \mathcal{D}(N)$ is also a smooth submersion. Note however that in general, $Coll(f)$ is not necessarily a smooth map when $f$ was smooth.

    \subsubsection{Smoothing of $G$-$\langle k \rangle$-manifolds}
    Much of the technical content of this paper will involve the use of equivariant $G$-smoothing theory \cite{Lashof1979} in order to compatibly smooth topological $G-\langle k \rangle$-manifolds underlying the Kuranishi charts we produce. Before delving into smoothing theory, we note the following lemma, which will allow us to apply smoothing theory to the setting of $\langle k \rangle$-manifolds. 
    
    \begin{lemma}
    \label{lemma:smoothing-double-smooths-original}
        Let $M$ be a topological $G-\langle k \rangle$-manifold. Suppose that $\mathcal{D}(M)$ is equipped with a smooth structure as a $(G \times \Z/2^k)$-manifold. There is a unique smooth structure on $M$ such that the inclusion $M \to \mathcal{D}(M)$ when viewed through smooth $\langle k \rangle$-manifold charts on $M$ and smooth charts on $\mathcal{D}(M)$.
    \end{lemma}
    \begin{proof}
        We must specify which topological $\langle k \rangle$-manifold charts on $M$ are smooth.
        
        For any $x \in M \subset \mathcal{D}(M)$ we can choose a chart $\psi^{-1}: V \to U \subset \R^n$, where $n = \dim M$, and such that $V$ is fixed by the action of $(G \times \Z/2)^k_x$,  $\R^n$ is a $(G \times \Z/2^k)_x$-representation, the map $\psi$ is equivariant,  $\psi^{-1}(x) = 0$, and $U$ is bounded. For any $S \in 2^{[k]}$, we have that $x \in M(S)$ if and only if $(G \times \Z/2^k)_x = G_x \times (\Z/2)^k_S \subset G \times (\Z/2)^k$ for some subgroup $G_x \subset G$.  Each of the $\Z/2$ factors $(\Z/2)^k_{S, j} \subset (\Z/2)^k_S$ fixes a codimension $1$ submanifold which we denote $\mathcal{D}(M)([k] \setminus \{j\})$, and the fixed loci all intersect transversely. From this we conclude that as a $(\Z/2)^k_S$-representation, the codomain of $\psi$ can be written (non-canonically) as  $\R^{[k]\setminus S} \times \R^{n-k+|S|}$ where the second factor is a trivial $(\Z/2)^k_S$ representation and the first factor is the standard representation of $(\Z/2)^{[k] \setminus S}$ by flipping signs.  Moreover, by composing with the action of $(\Z/2)^k$ on $\mathcal{D}(M)$ we can assume that $\psi^{-1}(V \cap M)\subset \R_+^{[k]\setminus S} \times \R^{n-k+|S|}$. 
        We declare that the restrictions of the inverses of such charts to $U \cap \R_+^{[k]\setminus S} \times \R^{n-k+|S|}$, which have image $M \cap V$, are smooth charts on $M$. Now, the transition maps between any pair of such charts on $M$ are smooth, since they are the compositions of coordinate permutations and smooth maps arising from transition functions of charts on $\mathcal{D}(M)$. This gives $M$ the structure of a smooth $\langle k \rangle$-manifold. To see that $G$ action with respect to this smooth structure is smooth, we use the same construction to build smooth charts on the $\langle k \rangle$-manifold $G \times M$ from the smooth charts on $G \times \mathcal{D}(M)$ to in order to verify that the map $G \times M \to M$ is smooth.
    \end{proof}
    In particular, by Lemma \ref{lemma:collars-and-doubles-are-smooth} and Lemma \ref{lemma:smoothing-double-smooths-original} we see that an equivariant smoothing of a $G-\langle k \rangle$-manifold $X$ is \emph{equivalent} to an equivariant smoothing of its double. 
    
    \subsubsection{Products, Induction, Restriction}
    Finally, we describe three natural operations on $G-\langle k \rangle$-manifolds which will be used to relate different such manifolds thorougout this work. 
    
    \paragraph{Product.} Given a $G_1- \langle k_1 \rangle$-manifold $M_1$ and a $G_2-\langle k_2 \rangle$-manifold $M_2$, the space $M_1 \times M_2$ admits the structure of a $(G_1 \times G_2) -\langle k_1 + k_2 \rangle$-manifold by using the product action and the product charts. If $M_1$ and $M_2$ were smooth, then $M_1 \times M_2$ is manifestly smooth as well.
    
    \paragraph{Induction.}
    We will assume now that $G$ is a compact Lie group. Given a closed Lie subgroup $H \subset G$ and an $H-\langle k \rangle$-manifold $M$, we have an associated $G-\langle k \rangle$-manifold with underlying stratified space
    \begin{equation}
        Ind_H^G M = G \times_H M, (Ind_H^G M)(S) = G \times_H M(S) \text{ for } S \in 2^{[k]}. 
    \end{equation}
    This is a smooth $G-\langle k \rangle$ manifold if $M$ was smooth. We say that any $G-\langle k \rangle$-manifold that is homeomorphic/diffeomorphic to $Ind_H^G M$ as a $G-\langle k \rangle$-manifold is an \emph{induction} (to $G$ along $H \subset G$) of $M$. 
    
    \paragraph{Restriction.} Given a $G-\langle k \rangle$-manifold $N$, we will say that an $H -\langle k \rangle$-manifold $M$ is a \emph{restriction} of $M$ if $M$ is an induction of $N$ (along an inclusion $H \subset G$. A criterion for checking that $N$ is a restriction of $M$ is 
    \begin{lemma}
    \label{lemma:check-restriction}
        With notation above, if there exists an immersion of $\langle k \rangle$-manifolds $N \to M$ such that any point of $M$ can be translated to a point in the image of $N$ via $G$, and such that the subgroup $H_n \subset G$ sending $n \in N$ into $N$ is independent of $n$, then $N$ is a restriction of $M$.
    \end{lemma}
    \begin{proof}
        Write $H_n = H$ for this subgroup of $G$, which must also be the subgroup of $G$ preserving the image of $N$; $H$ is thus a closed Lie subgroup. We have a smooth $G$-equivariant map $G \times_H N \to M$ given by $(g, n) \mapsto gn$. The condition in the lemma makes this map surjective. This map is injective since if $g_1n_1 = g_2n_2$, then $n_1 = g_1^{-1}g_2 n_2$, and thus $g_1^{-1}g_2 \in H$ and so $(g_1, n_1) \sim (g_1 g_1^{-1}g_2, n_2) = (g_2, n_2)$.
    \end{proof}
    
    Moreover, induction is manifestly functional: given an $H$-equivariant map between $H$-spaces $f: M_1 \to M_2$, we get an associated map between $G$-spaces $Ind_H^G: Ind_H^G M_1 \to Ind_H^G M_2$. In particular, an equivariant section of a vector bundle induces to an equivariant section on the induced vector bundle. 
    
    \begin{lemma}
    \label{lemma:induction-of-vector-bundles}
    Let $V$ be an $G$-representation, $H$ a closed subgroup of $G$, and $N$ an $H-\langle k \rangle$-manifold. Then, viewing $V$ as a $G$-space via the restriction, we have that  $Ind_H^G (N \times V)$ is canonically isomorphic as a $G-\langle k \rangle$-manifold to $(Ind_H^G) \times V$.
    \end{lemma}
    \begin{proof}
    One checks that the canonical map $G \times_H (N \times V) \to (G\times_H N) \times V$, which on equivalence classes is given by $(g, n, v) \mapsto (g, n, v)$, is a bijection. 
    \end{proof}
    
    \paragraph{Generalized induction.}
    Given a submanifold $V \subset G$ containing a subgroup $H$ and preserved by $H$, we have a functorial operation from $H-\langle k \rangle$-manifolds to $H-\langle k \rangle$-manifolds given by $X \mapsto Ind_H^{G, V}X := V \times_H X$. We call this procedure \emph{generalized induction}.

    \paragraph{Smooth structures.}
    
    \begin{lemma}
        \label{lemma:induction-bijects-smooth-structures}
        Let $M$ be an induction of $N$. There is a bijection between smooth structures on $M$ and smooth structures on $N$.
    \end{lemma}
    \begin{proof}
        By Lemma \ref{lemma:collars-and-doubles-are-smooth} and Lemma \ref{lemma:smoothing-double-smooths-original} we can assume that $M$ and $N$ are manifolds. 
        We can assume without loss of generality that $M = Ind_H^G N$. Given a smooth structure on $N$, it is clear that $Ind_H^G N$ inherits a canonical smooth structure. We argue that this map on smooth structures has an inverse. Note that for $x \in N$, the stabilizer of the image of $x$ in $Ind_H^G N$ is just the stabilizer of $x$ in $N$, which we denote by $H_x$. The slice theorem for smooth group actions says that an open neighborhood of $G \cdot x$ is difeomorphic to the total space of a $H_x$-equivariant disk bundle over $G \cdot x \simeq G/H_x$; the disk is a disk inside the normal in $T_x Ind_H^G N$ to the submanifold $G \cdot x$. But an open neighborhood of $H \cdot x \subset N \subset Ind_H^G N$ is then just the restriction of this principal bundle to $H /H_x \subset G / H_x$, and so in particular $N \subset Ind_H^G N$ is automatically a smooth submanifold. 
    \end{proof}
    
    \begin{lemma}
    \label{lemma:unique-smoothing-of-vector-bundle}
    Let $E$ be a topological $G$-vector bundle over a topological $G-\langle k \rangle$-manifold $M$. A choice of smooth structure on $M$ induces a choice of smooth structure on $E$ that is canonical up to contractible choice.
    \end{lemma}
    \begin{proof}
        This follows by applying the doubling construction to $E$ and $M$, using the corresponding bijection between isomorphism classes of topological and smooth equivariant bundles over a smooth $G \times (\Z/2)^k$-manifold \cite{Wasserman1969}, and restricting the resulting transition functions.
    \end{proof}
    
    \subsection{Kuranishi charts and derived orbifolds with corners.}
    \begin{definition}
    let $X$ be a space stratified by $2^{[k]}$. A \emph{topological Kuranishi chart} is a quadruple $\mathcal{K} = (G, T, V, \sigma)$, where $G$ is a compact Lie group, $T$ is a topological $G-\langle k \rangle$-manifold where $G$ has finite stabilizers, $V$ is a $G$-vector bundle over $T$, and $\sigma$ is a $G$-equivariant section of $V$. We say that $\mathcal{K}$ is a Kuranishi chart \emph{on X} when a \emph{footprint homeomorphism} $\phi: X \simeq \sigma^{-1}(0)/G$ is specified, where $\phi$ should be a map of stratified spaces.
    \end{definition}
    We say that a topological Kuranishi chart $\mathcal{K}$ is smooth when the thickening $T$ is a smooth $G-\langle k \rangle$-manifold and $V$ is a smooth vector bundle over $T$. A smooth Kuranishi chart $\mathcal{K}$ is \emph{regular} among a subset $C \subset \sigma^{-1}(0)/G$ when $\sigma$ is strongly transverse along the preimage of $C$ in $T$. Such a chart is \emph{regular} whenever it is regular on $\sigma^{-1}(0)/G$. 
    
    Given a smooth Kuranishi chart $\mathcal{K} = (G, T, V, \sigma)$, if the thickening $T$ is a $\langle 0 \rangle$-manifold then we get an associated \emph{derived orbifold chart} $([T/G], [V/G], [\sigma/G])$ \cite[Definition~1.3]{bai_xu}. In the setting of Floer homology we must introduce an appropriate generalization of the language of derived orbifolds to include the case where the thickening can be a convenient kind of manifold with corners. 
    
    Let $X$ be a Hausdorff and second-countable topological space stratified by $\{0 < \ldots<  k\}$. 
    \begin{definition}
    \label{def:orbifold-chart-with-corners}
    An (effective) \emph{$n$-dimensional orbifold chart with codimension $\ell$-corners} is a triple $C = (U, \psi, \Gamma)$ where
    \begin{itemize}
        \item $U \subset \R^\ell_+ \times \R^{n-\ell}$ is an open neighborhood of $0$,
        \item $\Gamma$ acts effectively and linearly on $\R^n \supset \R^\ell_+ \times \R^{n-\ell}$,  with the action a direct sum of a nontrivial action on $\R^{n-\ell}$ with the trivial action on $\R^\ell$;
        \item $U$ is $\Gamma$-invariant with respect to this action;
        \item  the \emph{footprint} map $\psi: U \to X$ is a $\Gamma$-invariant continuous map such that the induced map $\bar{\psi}: U/\Gamma \to X$ is a homeomorphism onto an open subset $V \subset X$, and 
        \begin{equation}
        \bar{\psi}^{-1}(X(i)) = U \cap ((\R^\ell_+)(i) \times \R^{n-\ell}). 
    \end{equation}
    \end{itemize} 
    \end{definition}
    Given $C=(U, \psi, \Gamma)$ and $C'=(U', \psi', \Gamma')$ a pair of $n$-dimensional orbifold charts with codimension $\ell$ and $\ell'$ corners, respectively, a chart embedding is a smooth open embedding $\iota: U \to U'$  such that $\psi' \circ \iota = \psi$. We will drop ``with codimension $\ell$'' when some $\ell$ is implicit. (To avoid linguistic confusion, note that that the inclusion $\R_+ \subset \R$ is \emph{not} an open embedding, but the inclusion  $(1,2) \times \R_+ \subset \R_+^2$ \emph{is} an open embedding. Also note that smooth open embeddings will send points in codimension $i$ strata to points in codimension $i$ strata necessarily.) 
    This automatically defines an equivariant injection of groups $\iota: \Gamma \to \Gamma'$. 
    
    
    WIth these definitions, we can proceed as usual. A chart is centered at $x \in X$ if $0 \in U$ and $x = \bar{\psi}(0)$; two charts are compatible if there are charts centered at every point in the intersections of their footprints which each admit a chart embedding into both charts. An orbifold atlas with codimension $k$ corners is a collection of orbifold charts on $X$ with corners of codimension at most $k$ which are mutually compatible, and a pair of orbifold atlases with codimension $k$ corners are equivalent if they admit a common refinement. Thus $X$ as above with an equivalence class of such atlases is an effective orbifold with codimension $k$-corners.
    
    Continuous functions $f: X \to \R$ are considered smooth when their pullbacks to all charts as in Definition \ref{def:orbifold-chart-with-corners} are smooth; orbifold vector bundles (orbibundles) and their sections are defined identically to \cite[Definition~3.3]{bai_xu}. 
    
   There is a natural refinement of an effective orbifold with codimension $k$ corners to the notion of an effective $\langle k \rangle$-orbifold. We use notation as in Section \ref{sec:k-manifolds}.
    \begin{definition}
    Let $X$ be a space stratified by $2^{[k]}$. We can view $X$ as a space stratified by $\{0 < \ldots < k\}$. A \emph{$\langle k \rangle$-chart} $(U, \Gamma, \psi)$ where $\psi$ is a map 
    \begin{equation}
        \psi: \R^{[k] \setminus S}_+  \times \R^{n - |k \setminus S|} \supset U \to V \subset X
    \end{equation}
    such that, 
    v, and we say that an orbifold chart with codimension $\ell$ corners $C = $ on $X$ is a  if, writing $\bar{\psi}(U) = V$, we have 
    \begin{itemize} 
        \item Viewing the domain of $\psi$ as $\R^\ell_+ \times \R^{n - \ell}$ (via, say, the order-preserving bijection $[k] \setminus S \to [\ell]$) we have that $(U,\Gamma,\psi)$ is an $n$-dimensional orbifold chart with codimension $\ell$ corners. Writing $\bar{\psi}: U/\Gamma \to X$ as in Definition \ref{def:orbifold-chart-with-corners}, it also holds that
    	\item for $y \in V$, we have that $S(y) \geq S(x)$, and moreover
    	\item For $[k] \supset T \supset S$, have that $\bar{\psi}^{-1}(X(T)) \subset \R^{[k] \setminus S}_+(r_S(T)) \times \R^{n - |k \setminus S|}$.
    \end{itemize}
   \end{definition}
   An embedding of $\langle k \rangle$-charts is just an embedding of the underlying orbifold charts with corners, and thus the notion of compatiblity of $\langle k \rangle$-charts continues to make sense. We have
   \begin{definition}
       An \emph{(effective) $\langle k \rangle$-orbifold} is a space $X$ stratified by $2^{[k]}$ which is equipped with an orbifold $\langle k \rangle$-atlas, i.e. a maximal collection of compatible $\langle k \rangle$-charts. We will denote a $\langle k \rangle$-orbifold with underlying space $X$ by $X$ when the orbifold $\langle k \rangle$-atlas is clear from the context.
   \end{definition}
   A smooth function, orbibundle, smooth section, etc. on a $\langle k \rangle$-orbifold are the same as those on the underlying orbifold with corners. 
   
   The formula for doubling a $\langle k \rangle$-manifold makes sense for \emph{any} space stratified by $2^{[k]}$, and by functoriality of doubling, the double $\mathcal{D}(X)$ of the underlying space of an effective $\langle k \rangle$-orbifold is covered by charts $\mathcal{D}(\psi): \mathcal{D}(U) \to \mathcal{D}(X)$:  if $C = (U, \Gamma, \psi)$ was a $\langle k \rangle$-chart with codimension $\ell$-corners, then $U$ is an $\langle \ell \rangle$-manifold, and $\mathcal{D}(\psi)$ is $\Gamma \times (\Z/2)^\ell$-equivariant. We call any such orbifold chart on $\mathcal{D}(X)$ produced by this procedure \emph{doubled chart}. Since the double of an open embedding of $\langle k \rangle$-manifolds is a smooth open embedding (by Lemma \ref{lemma:collars-and-doubles-are-smooth} 
   the doubled charts are mutually compatible, and so $\mathcal{D}(X)$ the underlying space of an effective orbifold, also denoted $\mathcal{D}(X)$, by declaring the doubled charts to be smooth. Moreover, by composing the doubled charts with the quotient map $\mathcal{D}(X) \to X$, we get new charts $(\mathcal{D}(U), \Gamma \times (\Z/2)^k_{S(\psi(0))}, \mathcal{D}(\psi))$ on $X$ (where $\mathcal{D}(U) \subset \R^{[k] \setminus S}  \times \R^{n - |k \setminus S|}$) which are also mutually compatible. These make make $X$ into the underlying space of another effective orbifold, which we will denote by $\mathcal{D}(X)/(\Z/2)^k$. 
   
   \paragraph{Producing $\langle k \rangle$-orbifolds from global quotients.}
    Let $T$ be a smooth $G-\langle k \rangle$-manifold with corners where the $G$-action has finite stabilizers with the stabilizer of a generic point being empty. 
    \begin{lemma}
        The space $T/G$ is equipped with a canonical structure of an effective $\langle k \rangle$-orbifold, which we denote $[T/G]$.
    \end{lemma}
    \begin{proof}
        We only explain the construction of the charts. Choose $x \in T \subset \mathcal{D}(T)$, and use the slice theorem for the latter manifold acted on by $G \times (\Z/2)^k$. Write $S = S(x)$. We have that $Stab(x) := (G \times (\Z/2)^k)_x = G_x \times (\Z/2)_S^k \subset G \times (\Z/2)^k$. The slice theorem then shows that $x$ has a $G$-invariant neighborhood $V$ which is $G \times (\Z/2)_S^k$-equivariantly isomorphic to $G \times_{G_x} D$, where the latter is a disk in the normal to $V = T_xG\cdot x$; here $(\Z/2)_S^k$ acts on $D$ via its action on the normal space. In particular there is a $Stab(x)$-invariant map $\psi^{-1}$ from a neighborhood $V$ of $x$ to  $U \subset Lie(G) \times \mathcal{R}$, where $Lie(G)$ is the lie algebra of $G$ thought of as a trivial $G_x \times (\Z/2)_S^k$-representation, while $\mathcal{R}$ is a nontrivial such representation. As in Lemma \ref{lemma:smoothing-double-smooths-original} we can write $\mathcal{R}$ as $\R^{\ell}\times\R^{n-\ell} $ where the second factor is a trivial $(\Z/2)_S^k \simeq (\Z/2)^\ell$ representation but a nontrivial $G_x$-representation, while in the first factor the $\Z/2$ factors of $(\Z/2)_S^k$ act by flipping sings while $G_x$ acts trivially. Composing such $\psi$ with an appropriate permutation of the last $\ell$ coordinates to produce $\psi'$, we see that $(G_x, (\psi)^{-1}(D) \cap \R^{\ell}_+ \times \R^{n-\ell}, \psi')$ is a $\langle k \rangle$-chart on $T/G$. We declare such charts to be smooth.
    \end{proof}
    
    From the above construction of charts, we conclude 
    \begin{lemma}
        \label{lemma:induced-orbifolds-are-the-same}
        Let $T'$ be the induction of $T$ along $G \subset G'$. Then $[T'/G']$ and $[T/G]$ define canonically isomorphic orbifolds. $\qed$
    \end{lemma}
    
      If $V$ is a $G$-vector bundle over $T$, then there is an associated orbibundle denoted $[V/G]$. Moreover, by applying the induction functor to charts, we see that the the induction $V' = Ind_G^{G'} V$ to a vector bundle over $T'$ defines the an isomorphic vector bundle $[V'/G']$ over $[T'/G] \simeq [T/G]$. Finally, the construction of charts above shows (by restricting $s$ to $1 \times D \cap \R^{\ell}_+ \times \R^{n-\ell}  \subset G \times_{G_x} D$) that given a $G$-equivariant section $s: T \to V$, there is an induced section $s: [T/G] \to [V/G]$ with underlying map $s/G: T/G \to V/G$ induced by the quotient construction on spaces, which is smooth if $s$ was smooth.

    \paragraph{Derived $\langle k \rangle$-orbifolds and global Kuranishi charts. }
        Given a smooth Kuranishi chart $\mathcal{K} = (G, T, V, \sigma)$, we can form a triple $\mathcal{K}//G = ([T/G], [V/G], [\sigma /G])$, where $[T/G]$ is an effective $\langle k \rangle$-orbifold, $[V/G]$ is the orbibundle over $[T/G]$ corresponding to $V$, and $[\sigma/G]$ is the continuous section $[T/G] \to [V/G]$ induced from $\sigma$. We call such a triple a \emph{derived orbifold chart}, and we say that $\mathcal{K}//G$ is the derived orbifold chart \emph{associated} to $\mathcal{K}$.  We write $\vdim (T, V, \sigma) = \dim V - \dim T$ for the expected dimension of $\sigma^{-1}(0)$. We say that a pair of derived orbifold charts $(\mathcal{T}_1, \mathcal{V}_1, \sigma_1)$ and $(\mathcal{T}_2, \mathcal{V}_2, \sigma_2)$ are \emph{equivalent} when there is a diffeomorphism $\Phi: U_1 \to U_2$ between open neighborhoods $U_i \supset \sigma_i^{-1}(0)$ along with an isomorphism of orbibundles $\bar{\Phi}: \mathcal{V}_1 \simeq \Phi^*\mathcal{V}_2 $ such that 
        $\sigma_2 \circ \Phi = \bar{\Phi} \circ \sigma_1$. There is an obvious notion of the product of a pair of derived orbifold charts. 
        
        \begin{lemma}
            Let $\mathcal{K}_1 = (H, T, V, \sigma)$ and let $\mathcal{K}_2 = Ind_H^G \mathcal{K}_1$. Then $\mathcal{K}_2//G$ is equivalent to $\mathcal{K}_1//H$. Similarly, given $\mathcal{K}' = (H', T', V', \sigma')$, we have that 
            $\mathcal{K} \times \mathcal{K}'//(H \times H')$ is equivalent to $\mathcal{K}_1//H \times \mathcal{K}_2//H'$. 
        \end{lemma}
        \begin{proof}
            This is an immediate consequence of Lemma \ref{lemma:induced-orbifolds-are-the-same} and the subsequent discussion about induction of vector bundles and smooth sections of orbibundles induced from equivariant sections.  
        \end{proof}
        
    A \emph{diffeomorphism} of derived $\langle k \rangle$-orbifolds $F: (\mathcal{T}_1, \mathcal{V}_1, \sigma_1) \to (\mathcal{T}_2, \mathcal{V}_2, \sigma_2)$ is a pair of diffeomorphisms $(f,g)$, with $f: \mathcal{T}_1 \to \mathcal{T}_2$ and $g: \mathcal{V}_1 \to \mathcal{V}_2$, such that $\pi_2 g = f \pi_1$, $g\iota_1 = \iota_2 f$, and $g\sigma_1= \sigma_2 g$, where $\pi_i$ is the projection from $\mathcal{V}_i$ to $\mathcal{T}_i$, while $\iota_i$ is the inclusion of the zero section.

    \paragraph{Normally complex structures} Fix a $\langle k \rangle$-orbifold $X$ equipped an orbibundle $V$ and a distinguished compact subset $K \subset X$. We can view $X$ as the as a smooth orbifold $\mathcal{D}(X)/(\Z/2)^k$ via the doubling construction; $V$ then becomes an orbibundle $\mathcal{D}(V)/(\Z/2)^k$ over this orbifold.  
    
    We say that a straightened metric on $X$ is a straightened metric on $\mathcal{D}(X)/(\Z/2)^k$; given a straightened metric on $X$, a straightened connection on $\mathcal{V}$ is a straightened connection on $\mathcal{D}{V}/(\Z/2)^k$ with respect to this metric; and a straightened structure on $(X,V)$ is a straightened structure on $(\mathcal{D}(X)/(\Z/2)^k, \mathcal{D}(V)/(\Z/2)^k)$ near $K$. \emph{In particular}, a straightened metric on $X$, resp. a straightened connection on $\mathcal{V}$, is a metric on $X$, resp. a connection on $\mathcal{V}$, satisfying a certain condition. We will say that a straightened structure on $(X, V)$ \emph{extends} the straightened metric on $X$. 
    
    The \cite[Lemma~3.15]{bai_xu} establishes that every $(X,V)$ has a straightned structure near a compact set $K \subset X$. One immediately sees that there is another tautological method for constructing straightened structures from global charts:
    
    \begin{lemma}
        Let $N$ be a $G$-$\langle k \rangle$ manifold and suppose that we have chosen a straightened structure on $[N/G]$. Let $\tilde{V}$ be a $G$-representation. Then the trivial connection on the bundle $\tilde{V} \times N$ over $N$ induces a straightened structure on $([N/G], [(\tilde{V} \times N])$. $\qed$
    \end{lemma}
    
    As for straightened structures on orbifolds, given a straightened structure on $X$ which extends to a straightened structure on $(X, V)$ and to a straightened structure on $(X, W)$ we have an associated straightened structure on $(X, V \oplus W)$.

    We can use a straightened structure on $(X, V)$ define a \emph{normal complex structure} on $V$ and $X$. As in the theory of orbifolds, every $\langle k \rangle$-orbifold chart $(U, \Gamma, \psi)$ extends to a \emph{bundle chart} $(U, E, \Gamma, \hat{\psi})$. Here, $E = E_0 \times U$ is a trivial $\Gamma$-equivariant vector bundle; writing $\pi_V: V \to X$ and $\pi_E: E \to U$ for the projections, we have that $\hat{\psi}$ is $\Gamma$-equivariant-map $E \to V$ such that $\pi_V\hat{\psi} = \psi\pi_E$. In particular, in each doubled bundle chart $(U, E, \Gamma, \hat{\psi})$ on $(\mathcal{D}(X)/(\Z/2)^k, \mathcal{D}(V)/(\Z/2)^k)$, the straightened structure defines a metric $g_U$ on $U$ and a connection $\Grad_U$ on $E$ such that 
    \begin{itemize}
        \item For each $G \subset \Gamma$, we have the exponential map $N_\epsilon U^G \to U$ from the $\epsilon$-disk bundle of the normal bundle of $U^G$, such in a neighborhood of $U^G$, the metric $g$ agrees with the pushforward under the exponential map of the bundle metric on $NU^G$ induced by using the Levi-Civita connection and the restriction of the metric to $NU^G$ under the canonical equivariant orthogonal splitting $TU|_{U^G} \simeq TU^G \oplus NU^G$ into trivial and nontrivial $G$-representations; and
        \item We have a connection $\Grad^E$ on $E$ induced from the straightened structure. Identifying this neighborhood of $U^G$ with $N_{\epsilon}U^G$, writing $\rho: (\pi_E|_{U^G})^*E \to E|_{N_{\epsilon}U^G}$ for the vector bundle isomorphism induced by parallel transport along normal geodesics, we have that that the pullback of $\Grad^E$ along $\rho$ agrees with the pullback of $\Grad^E|_{U^G}$ along $(\pi_E|_{U^G})^*$. 
    \end{itemize}
    Now, in every doubled bundle chart $(U, E, \Gamma, \hat{\psi})$, we have that $\Gamma \simeq \Gamma' \times (\Z/2)^k_S$ for some $S \in 2^{[k]}$ and some finite group $\Gamma'$. This doubled chart is the double of the $\langle k \rangle$-orbifold bundle chart $(V_+, E_+, \hat{\psi}|_{E_+})$ with $E_+=E|_{U'}$ and $U_+ = U \cap \R^{n-\ell} \times \R^{\ell}_+$, $\ell = k-|S|$, and $\Gamma'$ acting trivially on the last $\ell$ coordinates and $(\Z/2)^k_{S}$ flipping their signs. Let us consider $G \subset \Gamma$ for the form $G = G_+ \times 1$ and see what data the straightened structure on $\mathcal{D}(X)/(\Z/2)^k$ induces on $(V_+, E_+, \hat{\psi}|_{E_+})$. Since $U_+^{G_+} \simeq U \cap (\R^{n-\ell})^G \times \R^{\ell}_+$, we have that $NU_+^{G_+} = NU^G|_{U_+^{G_+}}$, which the straightened structure on the doubled chart equips with a metric. Moreover, by the $(\Z/2)^k_S$-equivariance of the metric, each boundary hypersurface of $U_+$ is totally geodesic in $U$, so the restriction of the exponential map to $N_\epsilon U_+^{G_+}$ has image in $U_+$. Thus the pushforward of the metric by the exponential map agrees with the restriction of the metric to $U_+$ near $U_+^{G_+}$. Similarly, the connection $\Grad^E$ restricts to a connection $\Grad^{E_+}$, and since normal geodesics stay in $U_+$, we have that the pullback of this metric along $\rho$ restricted to $(\pi_{E_+|_{U_+^{G_+}}})^*E_+$ agrees with the pullback of $\Grad^{E_+}$ along $(\pi_{E_+|_{U_+^{G_+}}})^*$.
    
    We now drop the pluses from our notation and work with the $\langle k \rangle$-orbifold instead of its double. In each orbifold bundle $\langle k \rangle$-chart, we have identifications of $N_\epsilon U^G$ with a neighborhood of $U^G$, and over this neighborhood the canonical splitting of $E|_{U^G} \simeq \oo{E}^G \oplus \check{E}^G$ into isotrivial and nonisotrivial $G$-representations extends to a $G$-equivariant splitting of $E|_{N_\epsilon U^G}$ \emph{which we will denote using the same notation}. As in \cite{bai_xu}, a normal almost complex structure on $X$ is an almost $H$-equivariant complex structure $I_H: NU_H \to NU_H$ for each $\langle k \rangle$-orbifold chart $(\Gamma, U, \psi)$ and every subgroup $H \subset \Gamma$ such that 
    \begin{itemize}
        \item for every $y \in U$ and every pair of subgroups $H \subset G \subset \Gamma_y$ (where we write $\Gamma_y$ for the stabilizer of $y$ in $\Gamma$), $(I_G)_y$ restricts to $(I_H)_y$ under the decomposition 
        \begin{equation}
        \label{eq:decomposition-1}
            N_yU^G \simeq (T_yU^H \cap N_y U^G)\oplus N_y U^H,
        \end{equation}
        and such that 
        \item For any chart embedding $\iota: V \to U$ between charts $(V, \Pi, \psi_V)$ and $(U, \Gamma, \psi_U)$, one has that the induced map $N_yV^{\Pi_y} \simeq N_\iota(y)U^{\Gamma_{\iota(y)}}$ is complex linear for each $y \in U$. 
    \end{itemize} 
    Similarly a normal almost complex structure on an orbibundle $V \to X$ is, for every bundle chart $(U, E, \Gamma, \hat{\psi})$ and every subgroup $G \subset \Gamma$, the data of a complex structure $J_G$ on  $\check{E}^G \to U^G$ for each bundle chart $(U, E, \Gamma, \hat{\psi})$ and each subgroup $G \subset \Gamma$, such that 
    \begin{itemize}
        \item for every $y \in U$ and every pair of subgroups $H \subset G \subset \Gamma_y$, $(J_G)_y$ restricts to $(J_H)_y$ under the decomposition 
        \begin{equation}
            \label{eq:decomposition-2}
            \check{E}^G_y \simeq (\check{E}^G_y)^H \oplus \check{E}^H_y,
        \end{equation} and such that 
        \item Under any equivariant bundle chart embedding $\hat{\iota}: F \to E$ from $(V, F, \Pi, \hat{\psi}_V)$ to $(U, E, \Gamma, \hat{\psi}_U)$ covering $\iota: V \to U$, we have that the induced map $\hat{\iota}: \check{F}^{\Pi_y} \to \check{E}^{\Gamma_{\iota(y)}}$ is complex linear for each $y \in V$.
    \end{itemize}

    Given a diffeomorphism of derived $\langle k \rangle$-orbifolds $F: (X_1, V_1) \to (X_2, V_2)$, $F = (f, g)$, we can ask $F$ to preserve certain structures, namely:
    \begin{itemize}
        \item If $(X_1, V_1)$ and $(X_2, V_2)$ have straightened structures, we can ask for $F$ to be a \emph{straightened diffeomorphism}, namely that $f$ is an isometry with respect to the straightened metrics, while $g^* \Grad^{V_2} = \Grad^{V_1}$ where $\Grad^{V_i}$ is the straightened connection on $V_i$. (Note that in general the pullback $g^* \Grad^{V_2}$ is only an Ehresmann connection on $V_1$ rather than a connection on this vector bundle, so this is a rather stringent condition;) 
        \item If $F$ is a straightened diffeomorphism and $(X_i, V_i)$ have fixed normal complex structures, we can ask for $F$ to be a \emph{normally complex diffeomorphism}, namely that $f$ is normally complex and $g$ is normally complex. These, mean, respectively:
        \item We say that $f$ is normally complex if for every $\langle k \rangle$-orbifold chart $(U_1, \Gamma, \psi_1)$ on $X_1$, we have a $\langle k \rangle$-orbifold chart $(U_2, \Gamma, \psi)$ with $\tilde{f}: U_1 \to U_2$ a $\Gamma$-equivariant map defining $f|_{U_1/\Gamma}$,
        then since $f$ is an isometry, we can identify $\tilde{f}$ with the restriction of linear bundle map $\tilde{f}': NU_1^\Gamma \to NU_2^\Gamma$, and \emph{we require} that $\tilde{f'}$ is complex linear,
        \item We say that $g$ is normally complex if for bundle charts $(U_1, E_1, \Gamma, \hat{\psi}_1)$ and $(U_2, E_2, \Gamma, \hat{\psi}_2)$ covering the $\langle k \rangle$-orbifold charts above, with  $\tilde{g}: E \to E'$ a $\Gamma$-equivariant bundle map $g$ (which is a fiberwise diffeomorphism) \emph{we require} that with respect to the decompositions $E_i|_{U_i^\Gamma} \simeq \oo{E}_i \oplus \check{E}_i$, we can write $\tilde{g}(\oo{e}_1, \check{e}_2) = (\oo{g}(\oo{e}_1), \check{g}(\check{e}_2))$, with $\check{g}$ an $I_H$-biholomorpic map. 
    \end{itemize}
    

    \paragraph{Normally complex sections} 
    Given two complex representations of a finite group $\Gamma$, let $V, W$, let $Poly_d^\Gamma(V, W)$
    be  $\Gamma$-equivariant polynomial maps from $V$ to $W$ of degree at most $d$, and let $ev: V \times Poly_d^\Gamma(V, W) \to W$ be the evaluation map. Given two $\Gamma$-equivariant complex vector bundles $V, W$ over a space $U$ with a trivial $\Gamma$-action, let $Poly_d^\Gamma(V, W)$ be complex vector bundle with fiber $Poly_d^\Gamma(V_x, W_x)$ at $x \in U$, and let $ev: V \oplus Poly_d^\Gamma(V, W)$ be the fiberwise evaluation map. The inverse images of $0$ under $ev$ are denoted by $Z^\Gamma_d(V, W)$ in each case, and they each admit \emph{canonical Whitney stratifications} according to \cite{bai_xu}.
    
    Given a bundle $\langle k \rangle$-orbifold chart $(U, E, \Gamma, \hat{\psi})$, a smooth $\Gamma$-equivariant section $s: U \to E$, and a point $x \in U$, we write $s$ in a neighborhood of $x$ as 
    \begin{equation}
        s = (\oo{s}_x, \check{s}_x): N_\epsilon U^{G_x} \to \oo{E}^{G_x} \oplus \check{E}^{G_x} \simeq E|_{N_\epsilon U^{G_x}}
    \end{equation}
    and call a normally complex lift of $s$ of degree at most $d$ near $x$ is a $G_x$-equivariant smooth bundle map over $U^{G_x}$
    \begin{equation}
    \begin{gathered}
        \mathfrak{s}_x: N_\epsilon U^{G_x} \to Poly_d^{G_x}(NU^{G_x}, \check{E}^{G_x}) \\ \text{ inducing } \text{graph}(\mathfrak{s}_x) := (id, \mathfrak{s}_x): N_\epsilon U^{G_x} \to NU^{G_x} \oplus Poly_d^{G_x}(NU^{G_x}, \check{E}^{G_x})
        \end{gathered}
    \end{equation}
    such that 
    \begin{equation}
        \check{s}_x = ev \circ \text{graph}(\mathfrak{s}_x). 
    \end{equation}
    We say that $s$ is a normally complex section, or an FOP section, if it has a normally complex lift near every $x \in U$. We say that a smooth section $s: X \to V$ is an FOP section if its pullback to every $\langle k \rangle$-orbifold bundle chart is an FOP section.
    
    \begin{remark}
    If $s$ admits a normally complex lift at $0 \in U$ then it admits a normally complex lift for all $x$ in a neighborhood of $0$, via the decompositions (\ref{eq:decomposition-1}, \ref{eq:decomposition-2}).
    \end{remark}
    
    An FOP section $s: X \to V$ is \emph{strongly transverse} at $p \in X$ if there is a $\langle k \rangle$-orbifold bundle chart $(U, E, \Gamma, \hat{\psi})$ such that, writing 
    \begin{equation}
        \mathfrak{s}: N_\epsilon U^{\Gamma} \to Poly^{\Gamma}_d(NU^\Gamma, \check{E}^\Gamma)
    \end{equation}
    for the normally concrete lift of $s_V: U \to E$ near $0$, we have that the section
    \begin{equation}\begin{gathered}
        \tilde{S}: N_\epsilon U^\Gamma \to \oo{E}^\Gamma \oplus NU^\Gamma \oplus Poly^{\Gamma}_d(NU^\Gamma, \check{E}^\Gamma)\\
        \tilde{S}(x) = (\oo{s}_0(x), x, \mathfrak{s}(x))
        \end{gathered}
    \end{equation}
    is transverse to $\{0\} \oplus Z^\Gamma_d(NU^\Gamma, \check{E}^\Gamma)$ with respect to the canonical Whitney stratification at the point $0 \in U$.

    Note that for any $S \in 2^{[k]}$, the subset $X(S)$ is a $\langle |S| \rangle$-orbifold. The orbibundle $V \to X$ restricts to an orbibundle $V(S) \to X(S)$, and straightened structures and normally complex structures on $(X,V)$ naturally restrict to the same on $(X,S), V(S)$. By restricting normally complex lifts we see that the restriction of an FOP section $s$ to $X(S)$ is FOP. 
    \begin{definition}
    An FOP section $s: X \to V$ is \emph{strongly $\langle k \rangle$-transverse} if $s|_{X(S)}$ is strongly transverse for each $s \in 2^{[k]}.$
    \end{definition}
    
    We can now prove the fundamental extension lemma:
    \begin{lemma}
    \label{lemma:fundamental-extension-lemma}
        Suppose $V \to X$ is an orbibundle over a $\langle k \rangle$-manifold, with $(X,V)$ equipped with a straightened structure and with normally complex structures. Let $s: X \to V$ be a continuous section with $s^{-1}(0)$ compact, and such that such that $s|_{X([k] \setminus \{j\})}$ is smooth, FOP, and strongly transverse for each $j = 0, \ldots, k$. Let $D \subset X$ be a precompact open neighborhood of $s^{-1}(0)$. Then there exists a section $s': X \to V$ which is FOP near $\overline{D}$, which agrees with $s$ on $X([k] \setminus \{j\})$ for each $j$, and which is arbitrarily $C^0$-close to $s$ on $D$.
    \end{lemma}
    \begin{proof}
        Write $\partial X = \cup_{j=0}^k X([k] \setminus \{j\}$ for the boundary of the $\langle k \rangle$-orbifold.
        It suffices to find any FOP extension of $s|_{\partial X}$ to $D$. Indeed, the proof of \cite[Lemma~6.2]{bai_xu} goes through verbatim to our setting, so given such an extension $s'$ it is necessarily strongly $\langle k \rangle$-transverse in a neighborhood of $\partial X \cap (s')^{-1}(0)$. Away from this neighborhood of the boundary we are in the the setting of orbifolds, and we can apply \cite[Theorem 6.3]{bai_xu}. 
        
        Let $(U, E, \Gamma, \hat{\psi})$ be a $\langle k \rangle$-orbifold bundle chart at $x \in X$. We will show how to extend $s$ from $\partial U$ to an FOP section on $U$. Now, we can construct smooth partitions of unity subordinate to any collection of $\langle k \rangle$-charts on $X$ by doubling the orbifold and applying the existence of partitions of unity for orbifolds \cite[Lemma~4.2.1]{ChenRuan}. Moreover, FOP sections are a module over smooth functions (this is an elementary generalization of the elementary \cite[Lemma 3.29]{bai_xu} to our setting), so this extension theorem for a single chart suffices. Assume that that $|S(x)| = \ell$, so $U \subset \R^{n-\ell} \times \R^\ell$, and think of $U$ as an $\langle \ell \rangle$-manifold via $U(T) = \R^{n-k} \times \R^\ell_+(T)$ for $T\in 2^{[\ell]}$. Shrink $U$ such that the ($\Gamma$-equivariant!) projections 
        \begin{equation}
        \label{eq:projection-maps}
            \pi_T: U \to U(T), T \in 2^{[\ell]}, \text{ defined by setting the coordinates of $\R^\ell_+$ not in $T$ to zero,}
        \end{equation}
        is well defined. 
        On these charts, one can extend $s|_{\partial U}$ to an FOP section on $U$ by applying the the following lemma to $\oo{s}_0$ and to $\mathfrak{s}_0$ separately:
        \begin{lemma}
        Let $U \subset \R^{n- \ell} \times \R^\ell_+$ be a neighborhood of zero, with $\Gamma$ a finite group acting linearly on $\R^{n-\ell}$ and on  a vector space $V$, and acting trivially on $\R^\ell_+$.  Write $\partial U =\cup_{j=0}^\ell U([\ell] \setminus \{j\}$, and let $s: \partial U \to E_0$ be a $\Gamma$-equivariant continuous map such that $s|_{U([\ell] \setminus \{j\}}$ is smooth for each $j$. Suppose that the projection maps (\ref{eq:projection-maps}) are well defined. Then $s$ admits a canonical smooth equivariant extension to $U$. 
        \end{lemma}
        \begin{proof}
        Write $s^T = s|_{U(T)}$. If we had that $s^T = 0$ for each $T \in 2^{[\ell]}$ with $|T| \leq \ell - 2$ then 
        \begin{equation}
            s = \sum_{j=1}^\ell s_{[\ell] \setminus \{j\}} \circ \pi_{[\ell] \setminus \{j\}}
        \end{equation}
        would be an extension of the desired form. We inductively modify $s$ until it is of this form. For the initial step of the induction, set $s_0 = s - s^{\emptyset} \circ \pi_{\emptyset}$; the second term is a smooth function on $U$. For each $j=1, \ldots, \ell$, we assume that $s_{j-1}^{T} = 0$ for $T \in 2^{[\ell]}$, $|T| = j-1$. We set 
        \begin{equation}
            s_j = s_{j-1} - \sum_{T \in 2^{[\ell]}, |T| = j} s^T_{j-1} \circ \pi_T.
        \end{equation}
        Each of the latter is a smooth function on $U$ and and now we can continue with the induction. 
        
        Thus combining the formulae above we will have written a smooth $\Gamma$-equivariant extension of $s$. 
        \end{proof}
    \end{proof}   
    
    \begin{definition}
    Let $V \to X$ be an orbibundle over a $\langle k \rangle$-orbifold such that $(X,V)$ are equipped with normal complex structures. We say that $(X,V)$ is \emph{oriented} if  for every $\langle k \rangle$-orbifold bundle chart $(U, E, \Gamma, \hat{\psi})$ we have an orientation of $U$ and an orientation of the bundle $E$, such $\Gamma$ acts on each by \emph{orientation-preserving maps}, and such that for any embedding of $\langle k \rangle$-orbifold bundle charts from $(V, F, \Pi, \hat{\psi}_V)$ to $(U, E, \Gamma, \hat{\psi}_U)$, the pushforwards of the orientations agree with the restrictions of the orientations.
    \end{definition}
        
    Finally, we have 
    \begin{lemma}
    \label{lemma:dimension-count}
        Suppose that $s: X \to V$ is strongly $\langle k \rangle$-transverse with $(X,V)$ oriented. Then for each $S \in 2^{[k]}$ for which $\dim X(S) = \dim V$, $s^{-1}(0)$ is in the interior of $X(S)$ and is a finite collection of points equipped with signs (i.e. an oriented zero-dimensional manifold) lying away from the isotropy locus, while for each $S$ for which $\dim X(S) = \dim V + 1$, $s^{-1}(0)$ is a manifold with boundary on strata $X(T)$, $T \leq S$, $|T|+1 =|S|$, lying away from the isotropy locus, with the boundary orientation of the manifold inducing the signs on the points of $\partial s^{-1}(0)\cap X(S)$ coming from the transversality of the restriction of $s$ to these strata.
    \end{lemma}
    \begin{proof}
        All statements except the compatibility of orientations follow from applying the dimension computations of \cite[Prop 6.6]{bai_xu} to the restriction of $s$ to the interior of each stratum of $X$. The statement about boundary orientations holds because away from the isotropy locus, strong transversality for $\langle k \rangle$-orbifolds reduces to strong transversality for $\langle k \rangle$-manifolds, and the oriented orbibundle restricts to an oriented vector bundle, where this statement is standard. 
    \end{proof}

		\section{Global Kuranishi charts for Hamiltonian Floer Homology}
	\label{sec:global-charts}
	Fix a compact symplectic manifold $(M^{2n}, \omega)$ with a compatible complex structure $J$, and a Hamiltonian diffeomorphism $\phi = \phi^1_H$. Let $g$ be the associated metric. 
	
	\subsection{An integral $2$-form for every Floer trajectory}
	\label{sec:integral-two-form}
	Given a closed $2$-form $\Omega$, we define
	$\mathcal{A}_\Omega: \widetilde{Fix}(\phi) \to \R$ to be $\int_x H dt- \int_{\tilde{x}} \Omega$. Write $Im(Fix(\phi))$ for the union of the images of all the maps $x:S^1 \to M$, $x \in Fix(\phi)$.

	\paragraph{Action Integrals}
	We say that a closed $2$-form $\Omega$ \emph{tames} $J$ if the quantity $g_\Omega(v, w) = \Omega(v, Jw) - \Omega(Jv, w)$ defines a metric on $M$. We will write $\|v\|_\Omega^2 = g_\Omega(v,v)$. By construction we have $\|Jv\|_\Omega = \|v\|_\Omega$. 
	
	If $\Omega$ is symplectic, we will write $X_H^\Omega$ for the Hamiltonian vector field of $H$ with respect to $\Omega$. Suppose $u$ is a Floer trajectory solving 
	\begin{equation}
		\partial_s u + J(\partial_tu - X_{H}^\omega) = 0
	\end{equation}
	with asymptotics $\bar{x}_\pm$. We will compute 
	\begin{equation}
		\label{eq:modified-floer-energy}
		\int \|\partial_s u\|^2_\Omega  = \frac{1}{2} \int \|\partial_s u \|^2_\Omega + \|J (\partial_t - X_{H}^\omega)\|_\Omega^2 = \frac{1}{2} \int - \langle g_\Omega(\partial_s u, J \partial_t u) + g(\partial_s u, J X^\omega_{H}). 
	\end{equation}
Now we have that 
\begin{equation}
	- g_\Omega(\partial_s u, J \partial_t u) = \Omega(\partial_s u, \partial_t u) + \Omega(J \partial_s u, J \partial_t u) = 2\Omega(\partial_s u, \partial_t u) + \Omega(X^\omega_{H}, \partial_s u) + \Omega(J \partial_s u, J X^\omega_{H}). 
\end{equation}
On the other hand, 
\begin{equation}
	g_\Omega(\partial_s u,J X^\omega_{H}) = - \Omega(\partial_s u, X^\omega_{H}) - \Omega(J \partial_s u, J X^\omega_{H}). 
\end{equation}
Thus the integral in (\ref{eq:modified-floer-energy}) is equal to 
\begin{equation}
\label{eq:modified-energy-integral-final}
	\int u^* \Omega + \Omega(X^\omega_{H}, \partial_s u) \; ds dt = \int u^* \Omega - \partial_s H(u(s,t)) ds dt + \Omega(X^\Omega_{H} - X^\omega_{H}, \partial_s u) \; ds dt. 
\end{equation}
In particular, if $\Omega = \omega$ then we see that
\begin{equation}
    \int \|\partial_s u\|_\omega^2 = \mathcal{A}_\omega(\bar{x}_-) - \mathcal{A}_\omega(\bar{x}_+). 
\end{equation}
confirming that the Floer equation is the downwards gradient flow of $\mathcal{A}_\omega$.

More generally, given a broken bubbled Floer trajectory $u: \Sigma \to M$, wrote $\Omega^F_u$ for the $2$-form on $\Sigma$ which which is $u|_{\Sigma_\alpha}^*\Omega$ on each bubble component $\Sigma_\alpha$, and is $u|_{\Sigma_i}^*\Omega - \partial_s H(u(s,t)) \,ds\,dt$ on each cylindrical component. Then 
\begin{equation}
\label{eq:energy-equation}
    E_\Omega(u) = \int_\Sigma \Omega^F_\Omega + \sum_{j=1}^{r_u} \int_{\Sigma_j} \Omega(X^{\Omega}_{X^k} - X^\omega_{H}, \partial_s u) \, ds\,dt 
\end{equation}
where $E_\Omega(u)$ is the geomeric Floer energy of $u$ with respect to $g_\omega$, and $\Sigma_j$, $j=1, \ldots, r_u$ runs over the cylindrical components of $\Sigma$.
\paragraph{Convenient integral forms.}

For $x \in Fix(\phi)$, let  $\mathcal{A}_\Omega^x = \{\mathcal{A}_\Omega(\tilde{x}) : \tilde{x} \text{ is capping of } x\}$. Write $Pairs_\Omega = \{\tilde{x}, \tilde{y}: \mathcal{A}_\Omega(x) - \mathcal{A}_\Omega(y) > 0\}$. Write $Pairs = \{(\tilde{x}, \tilde{y}) | \cM(\tilde{x}, \tilde{y}) \neq \emptyset\}$. It is a basic consequence of Gromov compactness for Floer trajectories that
\begin{equation}
	M = \min \{\mathcal{A}_\omega(\tilde{x}) - \mathcal{A}_\omega(\tilde{y}) | (\tilde{x}, \tilde{y}) \in Pairs\}
\end{equation}
exists and is positive.

\begin{lemma}
\label{lemma:make-form-zero-near-orbits}
    There exists a neighborhood $U \subset M$ of $Im(Fix(\phi))$ such that any closed $2$-form $\Omega$ on $M$ is cohomologous to a $2$-form $\tilde{\Omega}$ restricting to zero on $U$ and satisfying $\int_{\bar{x}} \Omega - \tilde{\Omega} = 0$ for all $x \in \widetilde{Fix}(\phi)$. 
\end{lemma}
\begin{proof}
    We first note that there is a neighborhood $U_0 \subset M$ of $Im(Fix(\phi))$ such that $H^2(U_0) = 0$. Indeed, this follows from the definition of Cech cohomology, together with the fact that $H^2(Im(Fix(\phi)) = 0$ for dimension reasons. Thus, $\Omega|_{U_0}  = d\lambda_0$ for some $\lambda_0 \in \Omega^1(U_0)$. We partition the elements of $Fix(\phi)$ into subsets $F_i$ with identical images; for each $x \in F_i, \int_x \lambda_0 = \lambda(i)$ only depends on $i$. If $\lambda_i \neq 0$ then the image of $x \in F_i$ is not constant, we can find an embedded interval $I_i$ inside the image of $x$ disjoint from the images of $y \in F_j$ for $j \neq i$. Thus we can find a closed $1$-form $\lambda_i$ on $U_0$ such that $\int_x \lambda_i = \int_x \lambda_0$, while $\int_y \lambda_i = 0$ for $y \in F_j$, $j \neq i$. Set $\lambda'_0 = \lambda_0 - \sum_i \lambda_i$; then $d \lambda'_0 = d\lambda_0 = \Omega|_{U_0}$ while $\int_x\lambda'_0=0$ for each $x \in Fix(\phi)$. Choose a smooth function $f$ which is zero outside of $U_0$ and $1$ on an open $U$ with $\overline{U} \subset U_0$ such that $Im(Fix(\phi)) \subset U$. Write $\lambda$ for the $1$-form equal to $0$ outside $U_0$ and equal to $f\lambda'_0$ on $U_0$. Then 
    $\tilde{\Omega} = \Omega - d \lambda$ satisfies the desired properties, since $\int_{\bar{x}} \Omega - \tilde{\Omega} = \int_x \lambda = 0$. 
\end{proof}

We fix this neighborhood $U$.

\begin{lemma}
\label{lemma:small-l2-near-fixed-locus}
	Let $x:S^1 \to M$ be a map. Then there exists an $\epsilon_0 > 0$ such that if $\int_{S^1} \|x' - X_{H}\|^2_\omega < \epsilon_0$ then $Im x \subset U$. 
\end{lemma}
\begin{proof}
	This follows immediately from Lemma 5.1 of \cite{le-ono}. 
\end{proof}

\begin{lemma}
	\label{lemma:integral-taming-form}
	There exists a rational symplectic form $\Omega$ taming $J$, and a closed $2$-form $\tilde{\Omega}$ cohomologous to $\Omega$ and restricting to zero in an open neighborhood of the image of each $x \in Fix(\phi)$, such that $Pairs \subset Pairs_\Omega = Pairs_{\tilde{\Omega}}$, and 
	\begin{equation}
	M' = \inf \{\mathcal{A}_\Omega(\tilde{x}) - \mathcal{A}_\Omega(\tilde{y}) | (\tilde{x}, \tilde{y}) \in Pairs\} \geq M/4.
\end{equation}
\end{lemma}

\begin{proof}
	
	Given $\Omega$, we can construct $\tilde{\Omega}$ by applying Lemma \ref{lemma:make-form-zero-near-orbits}, since the conditions on $\tilde{\Omega}$ imply that $\int_\Sigma u^* \Omega = \int_{\Sigma} u^* \tilde{\Omega}$ for $u: \Sigma \to M$ any broken bubbled Floer trajectory.
	
	Define for $\epsilon_1 > 0$ the set
    \begin{equation}
        S^{\epsilon_1}_\omega = \left\{ \Omega\in \Omega^2(M) | \Omega \text{ symplectic }, |\Omega(X^\Omega_H - X^\omega_H, v)| < \epsilon_1\|v\|_\omega, \|v\|_\Omega \geq \frac{1}{2} \|v\|_\omega \text{ for } v \in TM, \Omega|_U = \omega \right\}.
    \end{equation}
    We have a map $S^{\epsilon_1}_\omega \to H^2(M, \R)$ given by taking cohomology classes, and this map is surjective onto a neighborhood of $[\omega]$ for every $\epsilon_1>0$. This follows from the fact that we can choose a set of closed $2$-forms whose cohomology classes span $H^2(M, \R)$, apply Lemma \ref{lemma:make-form-zero-near-orbits} produce cohomologous $2$-forms which are zero on $U$, and add sufficiently small multiples of these $2$-forms to $\omega$ to produce elements of $S^{\epsilon_1}_\omega$. 
    
    We will show that for $\Omega \in S^{\epsilon_1}_\omega$ and $\epsilon_1$ sufficiently small, there is a constant $C$ such that 
    \begin{equation}
    \label{eq:goal-energy-comparison}
        \mathcal{A}_\omega(\tilde{x}) - \mathcal{A}_\omega(\tilde{y}) \leq C(\mathcal{A}_\Omega(\tilde{x}) - \mathcal{A}_\Omega(\tilde{y})).
    \end{equation}
    whenever $(\tilde{x}, \tilde{y}) \in Pairs$.
    
    To prove this, it suffices to prove this inequality whenever there exists a $u \in \overline{\M}(\tilde{x}, \tilde{y})$ where the domain of $u$ has only one cylindrical component.  
    
    In that case, by (\ref{eq:energy-equation}) we have the inequality 
    \begin{equation}
    \label{eq:initial-energy-comparison}
        \mathcal{A}_{\omega}(\tilde{x}) - \mathcal{A}_\omega(\tilde{y}) = E_\omega(u) \leq 2  E_{\Omega}(u) = 2 (\mathcal{A}_\Omega(\tilde{x}_-) - \mathcal{A}_\Omega(\tilde{x}_+)) + 2\int_{\R \times S^1} \Omega(X^\Omega_H - X^\omega_H, \partial_su)\, ds dt
    \end{equation}
    where $\R \times S^1$ is the cylindrical component of the domain of $u$. We can write $E_\omega(u) = E^{sphere}_\omega(u) + E^{cyl}_\omega(u)$ and 
    $E_\Omega(u) = E^{sphere}_\omega(u) + E^{cyl}_\Omega(u)$ as the sums of energies of the sphere components and the cylindrical components.  
    
    We write $L = \{s \in \R | \int_{S^1} \|\partial_s u(s, t)\|^2_\omega dt > \epsilon_0\}$ and note that $|L|$, the length of $L$, satisfies $|L| \leq E^{cyl}_\omega(u)/\epsilon_0$. By Cauchy Schwartz and the defining condition for $S^{\epsilon_1}_\omega$ we thus have that
    \begin{equation}
    \label{eq:bound-on-vector-field-error}
        \int_{\R \times S^1}|\Omega(X^\Omega_H - X^\omega_H, \partial_su)|\, ds dt \leq \frac{\epsilon_1 \sqrt{E^{cyl}_\omega(u)}}{\sqrt{\epsilon_0}} \left(\int_L \int_{S^1} \|\partial_s u\|_\omega^2\right)^{1/2} \leq 2\frac{\epsilon_1}{\sqrt{\epsilon_0}} E^{cyl}_\omega(u)
    \end{equation}
    Pick $\epsilon_1 = \sqrt{\epsilon_0}/8$. Moving the last term of (\ref{eq:initial-energy-comparison}) to the left hand side and applying the inequality derived above, we see that 
    \begin{equation}
        \left(1-4\frac{\epsilon_1}{\sqrt{\epsilon_0}}\right)E^{cyl}_\omega(u) + E^{sph}_\omega(u) \leq 2(\mathcal{A}_\Omega(\tilde{x}) - \mathcal{A}_\Omega(\tilde{y})). 
    \end{equation}
    But the left hand side is greater than $\frac{1}{2}(E^{cyl}_\omega(u) + E^{sph}_\omega(u)$, so we can take $C = 4$ in \ref{eq:goal-energy-comparison}.

    Thus any $\Omega \in S^{\epsilon_1}_\omega$ representing a rational cohomology class suffices.

\end{proof}

For $\Omega$ as in the lemma above, we have that $\mathcal{A}_\Omega^x \in r_x + 1/m\Z$ for some $r_x \in \R$ and $m \in \Z$. Choose $0 < \epsilon_x < min(1/m, M')$ such that
$\mathcal{A}_\Omega^x + \epsilon_x \subset \Q$, and choose $N_1 \in \Z$ such that $N_1 (\mathcal{A}_\Omega^x + \epsilon_x) \subset \Z$. 

\begin{equation}
\label{eq:background-form}
    \text{For the rest of the paper, fix }\Omega_0 = N_1 \tilde{\Omega}.
\end{equation}

\begin{equation}
\label{eq:define-a-bar}
	\bar{\mathcal{A}}: \tilde{Fix}(\phi) \to \Z, \bar{\mathcal{A}}(\tilde{x}) = \int_{\tilde{x}} N_1 \tilde{\Omega} - \int_x N(H + \epsilon_x) dt.  
\end{equation}
Then $\{\bar{\mathcal{A}}(\tilde{x}) - \bar{\mathcal{A}}(\tilde{y}) | (\tilde{x}, \tilde{y}) \in Pairs\} \in \Z_{> 0}$. The proof of the previous lemma also shows:

\begin{lemma}
\label{lemma:gromov-compactness-with-modified-energy}
    There are real constants $C$ and $D$ such that
    \begin{equation}
        \mathcal{A}_{\omega}(\tilde{x}) - \mathcal{A}_\omega(\tilde{y}) \leq C (\bar{\mathcal{A}}(\tilde{x}) - \bar{\mathcal{A}}(\tilde{y})) + D.
    \end{equation}
    In particular, the space of of broken bubbled Floer trajectories
    \begin{equation}
        \cM(H, d) = \{ u \in \cM(\tilde{x}, \tilde{y}), (\tilde{x}, \tilde{y}) \in Pairs \, | \, \bar{\mathcal{A}}(\tilde{x}) - \bar{\mathcal{A}}(\tilde(y)) \leq K \}
    \end{equation}
    is compact for every $K > 0$. $\qed$
\end{lemma}

Choose a smooth bump function $h: [0, 1] \to \R$, $h_{Op(0)} = 1, h_{Op(1)} = 0$, $h(x) + h(1-x) = 1$. Choose $\epsilon_2 > 0$ such that (when measured using, say, $g$) the exponential map from the $\epsilon$-ball in $T_{x(t)}M$ to $M$ is injective for each $t \in S^1$ and each $x \in Fix(\phi)$, and such that for each fixed $t$ the images of these balls do not overlap.  Define $\bar{H}: M \times S^1 \to \R$ to be equal to $N_1H(x,t)$ in the complement of $\cup_{x, t} \exp(D_{\epsilon_2} T_{x(t)}M \times \{t\}$, and define 
\begin{equation}
    \label{eq:bar-h}
	\bar{H}(y,t) = h(1-d(y, x(t))/\epsilon_2)N_1H(y,t) + h(d(y, x(t))/\epsilon_2) N_1(H(x(t)) + \epsilon_x) \text{ for } y \in \exp(D_{\epsilon_2}T_{x(t)}M). 
\end{equation}

\begin{lemma}
\label{lemma:bar-h}
	The function $\bar{H}$ is smooth, and if, for every smooth map
	\begin{equation}
		u: \R \times S^1 \to M, \lim_{s \to \pm\infty} u(s,t) = x_\pm(t) \text{(uniformly)}, u \# \bar{x}_- = \bar{x}_+
	\end{equation}
	we define the $2$-form 
	\begin{equation}
		\label{eq:floer-stabilizing-form}
		\Omega_u \in \Omega^2(\R \times S^1), \Omega_u(s,t) = u^*\Omega_0(s,t) - \partial_s \bar{H}\; ds dt
	\end{equation}
	then $\Omega_u$ is zero near $s = \pm \infty$ and $\int \Omega_u \in \Z$ agrees with $\bar{\mathcal{A}}(\tilde{x}) - \bar{\mathcal{A}}(\tilde{y})$, and takes positive values on $u \in \M(\bar{x}_-, \bar{x}_+)$.
\end{lemma}

Summarizing, we have defined the quantities $(\Omega_0,\bar{\mathcal{A}}, \Omega_u)$. These are needed to make sense of McLean stabilization of Floer trajectories. The quantity $\bar{\mathcal{A}}$ is the \emph{integralized action}. To define these quantities, we had to choose data $(\tilde{\Omega}, N_1, \bar{H})$, which we will call \emph{integralization data}. 

\subsection{Moduli Spaces for McLean Stabilizations}
In this section we describe the Floer-theoretic analog of the moduli spaces of stable maps to projective space in some detail, as well as an alternative way for thinking about elements of the universal curves over these moduli spaces.

\subsubsection{Moduli of decorated stable maps to projective space}
\label{sec:moduli-of-decorated-stable-maps}
Write $\M_{0, k}(\CP^\ell, d)$ for the space of connected stable maps $u: \Sigma \to \CP^\ell$ of degree $d$, where $\Sigma$ is a nodal domain with $k$ marked points. Any such map is contained in a $d$-dimensional linear subspace of $\CP^\ell$. Let $\mathcal{D}_k(\CP^\ell; d)$ be the Zariski-open subset consisting of those curves whose image is not contained in any $(d-1)$-dimensional linear subspace. 

Given an element $u: \Sigma \to \CP^\ell$ in $\M_{0, 2}(\CP^\ell, d)$, call the marked points $p_-, p_+$. There is a unique natural number $r_u$ and sequence of components $S^2_1, \ldots, S^2_{r_u}$ of $\Sigma$, with $S^2_j$ joined to $S^2_{j+1}$ along the nodal point $p_j$ for $j=1, \ldots, r_{u}-1$, such that $p_- = p_0$ lies on $S^2_1$ and $p_+ = p_{r_u}$ lies on $S^2_{r_u}$. 

We define $\mathcal{F}(\CP^\ell, d)$ to be the set of choices $(u, \gamma_1, \ldots, \gamma_{r_u-1})$ of element of $\mathcal{D}_2(\CP^\ell, d)$ together with a choice, for each $j=1, \ldots, r_u-1$, of a minimal geodesic $\gamma_j$ from $p_{j-1}$ to $p_j$ on $S^2_j$. The choice of minimal geodesic on each component fixes a preferred class of biholomorphisms 
\begin{equation}
\psi_{u, j}: \to \times S^1 \to S^2_j \setminus \{p_-, p_+\} 
\end{equation}
taking $\R \times \{0\}$ to the geodesic, with any two biholomorphisms related by an $\R$-translation. We write $\mathcal{C}(\CP^\ell, d)$ for the universal curve over $\mathcal{F}(\CP^\ell, d)$.

\paragraph{Real oriented blow-up of $\C^n_+$.}
The \emph{real oriented blowup} of a smooth submanifold $Y$ of a smooth manifold $X$ is the union 
\begin{equation}
    Bl_Y(X) = X \setminus Y \bigsqcup SN_YX, \text{ where }SN_YX\text{ is the unit sphere bundle of the normal bundle of }Y\text{ in }X.
\end{equation}
This carries a canonical smooth structure \cite{KM}. 

The stratified space $\R^n_+$ has a natural complexification $\C^n_+$. Let $\C^n_+$ denote $\C^n \supset \R^n_+$ with the stratification indexed by $2^{[n]}$, with $\C^n(S)$ equal to the Zariski closure of $\R^n_+(S)$. Each $\C^n(S)$ is a linear subspace of $\C^n$, and is given by an intersection of coordinate hyperplanes of $\C^n$. Write $H_1, \ldots, H_n$ for the coordinate hyperplanes of $\C^n$. 

We explain how to generalize this notion to construct the real oriented blow-up of $\C^n_+$ along its stratification, which will be a smooth manifold of real dimension $2n$ with codimension $n$ corners. This model case will allow us to construct real-oriented blow-ups of $\M_{0, 2}(\CP^d, d_1)$, which will endow $\mathcal{F}(d; d_1)$ with the structure of a smooth manifold with corners.

We define $\Bl \C^n_+ = \R^n_+ \times (S^1)^n $. There is a diffeomorphism $(\R^n_+ \setminus \partial\R^n_+) \times (S^1)^n  \to \C^n_+ \setminus \bigcup_{i =1}^n H_i$ with inverse given by 
$(z_1, \ldots, z_n) \mapsto ((|z_1|, \ldots, |z_n|), (\arg z_1, \ldots, \arg z_n))$. Write $\pi: \Bl \C^n_+ \to \C^n_+$ for the unique continuous (in fact, smooth) extension of this diffeomorphism. We have that $\Bl \C^n_+$ is stratified by the poset $2^{[n]}$ with stratum associated to $S$ given by $\R^n_+(S) \times (S^1)^n$. The map $\pi$ takes strata surjectively to strata, and exhibits open strata of $\Bl \C^n_+$ as trivial torus bundles over open strata of $\C^n_+$. 

Let $\C^n_+ \times \C^m $ denote the space $\C^n \times \C^m$ with stratification by $2^{[n]}$ given by $(\C^n_+ \times \C^m )(S) =\C^n_+(S) \times \C^m$. The real oriented blow-up $\Bl \C^m \times \C^n_+$ is simply $\R^n_+ \times (S^1)^n \times \C^m$, which admits a smooth map $\pi$ to $\C^n_+ \times \C^m $ given by the product of the map $\pi$ in the paragraph above with the identity map. 

The fundamental property of this construction is 
\begin{lemma}
\label{lemma:lifted-diffeomorphisms}
    Given a diffeomorphism $\phi: \C^n_+ \times \C^m  \supset U \to V \subset \C^n_+ \times \C^m $ which preservers the stratification, there is a canonical induced diffeomorphism $\tilde{\phi}: \pi^{-1}(U) \to \pi^{-1}(V)$ such that $\phi \circ \pi = \pi \circ \tilde{\phi}$.
\end{lemma}
\begin{proof}
    The differential of $\phi$ preserves the tangent bundle to each stratum $\C^n_+ \times \C^m(S)$ thus induces a  map $N\C^n_+ \times \C^m(S) \simeq N\C^n_+ \times \C^m(S)$ covering $\phi$ which is a linear isomorphism on each fiber. Because it preserves the strata containing $\C^n_+ \times \C^m(S)$ this map is block diagonal under the decomposition 
    \begin{equation}
        N\C^n_+ \times \C^m(S)= \bigoplus_{j \notin S} N_{\C^n_+ \times \C^m(S \cup \{j\})}\C^n_+ \times \C^m(S)
    \end{equation}
    where each of the latter bundles is canonically a trivial $\C$-bundle. But a real-linear isomorphism $\C \to \C$ induces a smooth diffeomorphism $\C\setminus 0/ \R \to \C \setminus 0 /\R$ where the $\R$-action is by scaling. As fibers over the open strata are canonically the unit circle bundles of these normal bundles, this defines the lift $\tilde{\phi}$. To see that $\tilde{\phi}$ is a diffeomorpism, we note that we can write $\tilde{\phi}$ as the unique continuous extension of 
    \begin{equation}
        \phi^m(w, \vec{z}) = ((|\phi_1(w, \vec{z})|, \ldots, |\phi_n(w, \vec{z})|), \arg \phi_1(w, \vec{z}), \ldots, \arg \phi_n(w, \vec{z})), \text{ where } \vec{z} = (|z_1|e^{\arg z_1}, \ldots, |z_n| e^{\arg z_n}), w \in \C^m
    \end{equation}
    where $\phi_j$ are are the compositions of $\phi$ with projection to the $j$-th complex coordinate in the $\C^n_+$ factor, while $\phi^m$ is the composition of $\phi$ with projection to the first $m$ coordinates. Thus $\tilde{\phi}$ pulls smooth functions back to smooth functions and thus is smooth. The inverse of $\tilde{\phi}$ is of the same form. 
\end{proof}

\paragraph{Real oriented blow-up of a manifold with stratification modeled on $\C^k_+$.}

Let $X$ be a smooth manifold stratified by $2^{[k]}$. We say that this stratification is modeled on $\C^k_+$ if $X$ is $2(m+k)$-dimensional, and if there exists a subset $Charts(X, \C^k_+)$ of smooth charts on $X$ such that for each chart $\phi: X \supset U \to V \subset \R^{2(m+k)}= \C^k \times \C^m$, $\phi \in Charts(X, \C^k_+)$, we have that $\phi^{-1}(\C^k_+ \times \C^m(S)) = X(S) \cap U$. 


Given such a smooth manifold $X$ with stratification modeled on $\C^k_+$, we can define a new smooth $\langle k \rangle$-manifold $\Bl X$ by covering $M$ by a collection of charts $\phi: U_\phi \to V_\phi$, $\phi \in Charts(X, \C^k_+)$ and letting $\Bl X$ the union of $\pi^{-1}(V_\phi)$ with the equivalence relation that for each $\phi_1, \phi_2$ a pair of charts in the collection, we identify  $\pi^{-1} \phi_1(U_{\phi_1} \cap U_{\phi_2})$ with $\pi^{-1}\phi_1(U_{\phi_1} \cap U_{\phi_2})$  using the lifts $\widetilde{\phi_1 \circ \phi_2^{-1}}$ of $\phi_1 \circ \phi_2^{-1}$ produced by Lemma \ref{lemma:lifted-diffeomorphisms}. This manifold is manifestly independent of the choice of charts up to canonical diffeomorphism.

Any stratum-preserving diffeomorphism $f: X \to X$ thus lifts to a unique diffeomorphism $\bar{f}: \Bl X \to \Bl X$ such that $\pi \bar{f} = f \pi$. Given a smooth $G$-action on $X$ which preseves the strata, this lifting procedure gives an induced action $G \times \Bl X \to \Bl X$ which makes $\Bl X$ into a smooth $G-\langle k \rangle$-manifold.

\paragraph{Manifold with corner structure of McLean-stabilized Floer moduli spaces.}

We now use the construction of the previous paragraph to give $\mathcal{F}(\CP^\ell, d)$ the structure of a smooth $\langle d-1 \rangle$-manifold. 

There are evaluation maps $ev_\pm: \mathcal{F}(\CP^\ell, d) \to \CP^\ell$ given by $(u, \gamma_1, \ldots, \gamma_{r_u-1}) \mapsto u(p_\pm)$. For $S \in 2^{[d-1]}$, we define
\begin{equation}
    \tilde{\mathcal{F}}_S(\CP^\ell, d) = \mathcal{F}(\CP^\ell, h_1) \times_{\CP^\ell} \times \ldots \times_{\CP^\ell} \mathcal{F}(\CP^\ell, h_{d-1-|S|}), \; (h_1,\ldots, h_{d-1-|S|}) = h(S),
\end{equation}
with each fiber product taken over $ev_+$ on the left and $ev_-$ on the right. There is a codimension $0$ open submanifold  
\begin{equation}
    \mathcal{F}(\CP^\ell, d)(S) \subset \tilde{\mathcal{F}}_S(\CP^\ell, d)
\end{equation}
consisting of tuples of maps $(u_1, \ldots, u_{d-1-|S|})$ such that the dimension of the minimal subspace containing all of their images is $d$; by gluing these stable maps along their marked points we can think of $\mathcal{F}(\CP^\ell, d)(S) \subset \mathcal{F}(\CP^\ell, d)$, and we will identify $\mathcal{F}(\CP^\ell, d)(S)$ with the image of this injection. There is a map $F: \mathcal{F}(\CP^\ell, d) \to \mathcal{D}_2(\CP^\ell, d)$ induced by forgetting the choices of geodesics, and it sends 
$\mathcal{F}(\CP^\ell, d)(S)$ to $\mathcal{D}_2(\CP^\ell, d)(S)$ which is a codimension zero open submanifold in 
\begin{equation}
    \mathcal{D}_2(\CP^\ell, d)(S) \subset \mathcal{D}_{2, S}(\CP^\ell, d) \simeq \mathcal{D}_2(\CP^\ell, h_1) \times_{\CP^\ell} \times \ldots \times_{\CP^\ell}\mathcal{D}_2(\CP^\ell, h_{d-1-|S|}).
\end{equation}

We think of $\mathcal{F}(\CP^\ell, d)(S)$ and $\mathcal{D}_2(\CP^\ell, d)(S)$ as defining stratifications on the corresponding spaces. 

We have the following
\begin{lemma}
\label{lemma:stratification-model}
    The stratification of $\mathcal{D}_2(\CP^\ell, d)$ defined by the sets above makes this manifold into a manifold with stratification modeled on $\C^{d-1}_+$.
\end{lemma}
\begin{proof}
    This is essentially a standard fact in Gromov-Witten theory.  First, assume $\ell=d$. Then given $(u: \Sigma \to \CP^d) \in \mathcal{D}_2(\CP^d,d)$, let $v: \Sigma' \to \CP^d$ be the corresponding stable map to $\CP^d$ with no marked points. We can choose $d+3$ hypersurfaces $H_1, \ldots, H_{d-1} \subset \CP^d$ such that every nonconstant component of $\Sigma$ is transverse to each of the $H_j$, such that there exists a choice $x_j \in v^{-1}(H_j)$ such that $(\Sigma', x_1, \ldots, x_{d+3})$ is a stable curve, and such that $v(x_j) \notin H_k$ for $k \neq j$. Choose small neighborhoods $W_j$ of $x_j$ which do not overlap and do not intersect $H_k$ for $k \neq j$. The set of stable maps $V \subset \cM_{0,0}(\CP^d, d)$ which are transverse to all $H_j$ and intersect $H_j$ in $V_j$ once is an open set diffeomorphic to its image $V_{marked} \subset \cM_{0, d+3}$ send a curve $u: \Sigma \to \CP^d$ to $(\Sigma, v^{-1}(H_1), \ldots, v^{-1}(H_{d+3}))$. The inverse image of this open set $U$ in $\cM_{0,2}(\CP^d, d)$ is thus diffeomorphic to its image $U_{marked}$ in $\cM_{0, d+5}$ under the diffeomorphism $(u, p_-, p_+) \mapsto (\Sigma, p_-, p_+, v^{-1}(H_1), \ldots, v^{-1}(H_{d+3})$. Given $\Sigma' \in \cM_{0,k}$ its neighborhood contains an intersection of hyperplanes $T_{1}, \ldots, T_e$, one for each node of of $\Sigma'$, which is modeled on $\C^{e}_+$ \cite[Theorem~D.4.7]{McDuffSalamon-BIG}. The image of our stratification is simply the stratification corresponding to the hyperplanes associated to the nodes on the path between $p_-$ and $p_+$. There are at most $d-1$ such nodes, and since we are proving a local statement this establishes the lemma for $\ell=d$.
    
    To prove the lemma for $\ell > d$, can use the fact that there is a smooth stratum-preserving diffeomorphism (see the next paragraph)
    \begin{equation}
    U(\ell+1) \times_{(U(d+1) \times U(\ell-d))/U(1)} \mathcal{F}(\CP^d, d) \to \mathcal{F}(\CP^\ell, d).
    \end{equation}
    and thus the stratification of $\mathcal{F}(\CP^\ell, d)$ is also modeled on $\C^{d-1}_+$.
\end{proof}


Let $\Theta_-$ denote the circle bundle over $\mathcal{D}_2(\CP^\ell, d)$ with fiber over $u: \Sigma \to \CP^\ell$ given by the unit circle of $Tp_-\Sigma$. The smooth structure on $\mathcal{F}(\CP^\ell, d)$ is defined by the following
\begin{lemma}
\label{lemma:equip-mclean-moduli-spaces-with-k-manifold-structure}
    Write $\pi: \Bl \mathcal{D}_2(\CP^\ell, d) \to \mathcal{D}_2(\CP^\ell, d)$ for the canonical projection. There is a bijection $\mathcal{B}: \mathcal{F}(\CP^\ell, d) \to \pi^*\Theta_-$ intertwining $F: \mathcal{F}(\CP^\ell, d) \to \mathcal{D}_2(\CP^\ell, d)$ and $\pi \circ \pi_{\pi^*\Theta_-}: \pi^*\Theta_- \to \mathcal{D}_2(\CP^\ell, d)$. The space $\pi^*\Theta_-$ is a smooth $\langle d-1 \rangle$-manifold with the strata given by $\mathcal{B}(\mathcal{F}(\CP^\ell, d)(S))$ as $S$ runs over $2^{[d-1]}$.
\end{lemma}
\begin{proof}
    We think of a point in the fiber of $\pi^*\Theta_-$ as specifying the tangent vector to the geodesic on $S^2_1$ at $p_-$; as such, it determines this geodesic, and this defines $\pi$ on the interior of the top strata of each of the spaces. The fiber of $\pi$ over a curve with $r_u$ components between the component containing $p_-$ and the component containing $p_+$ is $(S^1)^{r_u-1}$ as in Lemma \ref{lemma:stratification-model}, and we interpret these circles as measuring the angle between geodesics. One verifies that this bijection extends from the interior of the top-dimensional strata by analyzing local models for the universal curve over neighborhoods of points on the stratification of $\mathcal{D}_2(\CP^\ell, d)$, which, as in Lemma \ref{lemma:stratification-model}, can be extracted from \cite[Theorem~D.4.7]{McDuffSalamon-BIG}.
\end{proof}

By using the smooth charts from $\C^k_+ \times \C^\ell$ on $\mathcal{D}_2(\CP^\ell, d)$ induced by the choices of hypersurfaces in the proof of Lemma \ref{lemma:stratification-model}, we can also see that the pullback $\widetilde{\mathcal{C}}_2(\CP^{\ell}, d)$  of the universal curve $\mathcal{C}_2(\CP^{\ell}, d)$ over $\mathcal{D}_2(\CP^\ell, d)$ to $\Bl \mathcal{D}_2(\CP^\ell, d)$ is also a smooth $\langle d-1 \rangle$-manifold, with the induced map to $\Bl \mathcal{D}_2(\CP^\ell, d)$ a smooth map. The universal curve $\mathcal{C}(\CP^{\ell}, d)$ is thus in natural bijection with the pullback of $\pi^* \Theta_1$ to $\widetilde{\mathcal{C}}_2(\CP^{\ell}, d)$, and thus is a smooth $\langle d-1 \rangle$-manifold with a smooth map to $\mathcal{F}(\CP^\ell, d)$. We have proven:

\begin{proposition}
    \label{prop:smooth-structures-on-mclean-stabilizations}
    $\mathcal{F}(\CP^\ell, d)$ and $\mathcal{C}(\CP^{\ell}, d)$ are smooth $\langle d-1\rangle$-manifolds and the projection map from the latter to the former is a submersion. $\qed$
\end{proposition}

\subsubsection{Unitary group actions}
\label{sec:unitary-group-actions}
There is a natural $U(\ell+1)$ action on $\CP^d$ given by the linear action on homogeneous coordinates; it factors through the projective unitary group $PU(\ell+1)$, which acts effectively.  This induces an action of $PU(\ell+1)$ on $\mathcal{D}_2(\CP^\ell, d)$ by postcomposition.  This action is smooth on $\mathcal{D}_2(\CP^\ell, d)$, and induces a smooth action on $\Theta_-$. Thus, it lifts canonically to a smooth action $\mathcal{F}(\CP^\ell, d)$, making the latter into a smooth $PU(\ell+1)-\langle d-1 \rangle$-manifold. 

It will be convenient for later sections to note that we can give an alternative formula for $\mathcal{F}(\CP^\ell, d)(S)$ in terms of \emph{induction and restriction}, without ever taking fiber products over evaluation maps. 

Fix $p \in \CP^\ell$. We have subspaces $\mathcal{F}_\pm(\CP^\ell, d) = ev^{-1}_\pm(p) \subset \mathcal{F}(\CP^\ell, d)$. The subgroup $Stab(p) \subset PU(\ell+1)$ which fixes $p$ is isomorphic to $U(\ell) \subset PU(\ell+1)$, and the smooth surjective map $PU(\ell+1) \times \mathcal{F}_\pm(\CP^\ell, d) \to \mathcal{F}(\CP^\ell, d)$ descends to a map 
\begin{equation}
    PU(\ell+1) \times_{U(\ell)} \mathcal{F}(\CP^\ell, d; p, \pm) \simeq \mathcal{F}(\CP^\ell, d).
\end{equation}
The geometry immediately implies that this is a diffeomorphism Lemma \ref{lemma:check-restriction}, showing that the spaces $\mathcal{F}(\CP^\ell, d; p, \pm)$ are restrictions of $\mathcal{F}(\CP^\ell, d)$.

Since every linearly nondegenerate curve of degree $d$ in $\CP^\ell$ is contained inside a unique $d$-dimensional subspace, we have canonical maps
\begin{equation}
\begin{gathered}
    U(\ell+1) \times_{(U(d+1) \times U(\ell-d))/U(1)} \mathcal{F}(\CP^d, d) \to \mathcal{F}(\CP^\ell, d),\\
    U(\ell+1) \times_{U(d) \times U(\ell-d)} \mathcal{F}_\pm(\CP^d, d) \to \mathcal{F}_\pm(\CP^\ell, d).
\end{gathered}
\end{equation}
which are also diffeomorphisms by \ref{lemma:check-restriction} that show that $\mathcal{F}(\CP^\ell, d)$ is an induction of $\mathcal{F}(\CP^d, d)$ and similarly for $\mathcal{F}_\pm(\CP^\ell, d)$

Now, write $d = d_1 + d_2$. Choose a copy of $\CP^{d_1} \subset \CP^d$ and choose $p \in \CP^{d_1}$. The stabilizer of $(p, \CP^{d_1})$ is in $PU(d+1)$ is a copy of $U(d_1) \times U(d_2) \subset PU(d+1)$. Define the subset 
\begin{equation}
\label{eq:split-mclean-domain-subset}
\begin{gathered}
    \mathcal{F}^{tr}_{S(d_1,d_2)}(\CP^d,d)  \subset  \mathcal{F}_{S(d_1,d_2)}(\CP^d, d) \\
    \mathcal{F}^{tr}_{S(d_1,d_2)}(\CP^d,d)  =
    \left\{
    	u: \Sigma_1 \cup_{p_1} \Sigma_2 \to \CP^d ; \; u|_{\Sigma_1} \subset \CP^{d_1}, \deg u|_{\Sigma_1} = d_1, u(p_1) = p \right\}/\{\text{equiv.}\},
\end{gathered}
\end{equation}
where $\Sigma_1 \cup_{p_1} \Sigma_2$ denotes a nodal curve decomposed into two nodal curves $\Sigma_1$ and $\Sigma_2$ joined along a nodal point $p_1$. This subset is a slice for the $PU(d+1)$-action on $\mathcal{F}_{S(d_1,d_2)}(\CP^d, d)$, and so $\mathcal{F}^{tr}_{S(d_1,d_2)}(\CP^d,d)$ induced along  $U(d_1) \times U(d_2) \subset PU(d+1)$ is $\mathcal{F}_{S(d_1,d_2)}(\CP^d, d)$.

Finally, we have an inclusion 
\begin{equation}
    \mathcal{F}_+(\CP^{d_1}, d_1) \times \mathcal{F}_-(\CP_{d_2}, d_2) \subset \mathcal{F}^{tr}_{S(d_1,d_2)}(\CP^d,d)
\end{equation}
as the subspace of $u$ where $u|_{\Sigma_2}$ lies in the complementary subspace $\CP^{d_2} \subset \CP^d$ which is spanned by $p$ and $d_2$ additional vectors which are all orthogonal to the vectors projecting to $\CP^{d_1}$. Let $V^{tr}_{d, d_1} \subset PGL_{d+1}(\C)$ be the set of matrices which act isometrically on $\CP^{d_1} \subset \CP^d$ and send the complementary subspace $\CP^{d_2} \subset \CP^d$ isometrically onto another copy of $\CP^{d_2} \subset \CP^d$ intersecting $\CP^{d_1}$ exactly at $p$. 

One sees immediately that there is a $U(d_1) \times U(d_2)$-equivariant diffeomorphism
\begin{equation}
\label{eq:generalized-induction}
    V_{d, d_1}^{tr} \times_{(U(d) \times U(\ell-d)} \mathcal{F}_+(\CP^{d_1}, d_1) \times \mathcal{F}_-(\CP_{d_2}, d_2) \to \mathcal{F}^{tr}_{S(d_1,d_2)}(\CP^d,d). 
\end{equation}
(sending $(g, u_1, u_2)$ to the gluing of $g u_1$ with $g u_2$ along $p$) 
making its codomain a \emph{generalized induction} if its domain.

By iterating the constructions of the above subsection above, we conclude:
\begin{lemma}
    \label{lemma:induction-relation-between-strata}
    The $U(\ell-1)$-$\langle d-1 \rangle$-manifold $\mathcal{F}_S(\CP^d, d)$ can be obtained from the $U(h_i)-\langle h_i-1\rangle$-manifolds $\mathcal{F}(\CP^{h_i}, h_i)$, with $(h_1, \ldots, h_{d-1-|S|}) = h(S)$, by taking products, (generalized) induction, and restriction. $\qed$
\end{lemma}

\subsubsection{Framed bubbled pre-Floer trajectories.}
We recall \cite[Definition 6.9]{AMS} based on \cite[Lemma 6.8]{AMS}. 
\begin{definition}
Let $\Sigma$ be a genus zero nodal curve, and let $\Omega \in \Omega^2(\Sigma)$ be a closed $2$-form representing an integral cohomology class. There is a unique Hermitian line bundle $L_\Omega$ with curvature $-2\pi i \Omega$ and Hermitian inner product $\langle\,  \cdot\,,\, \cdot\,\rangle_\Omega$, up to isomorphism. 
\end{definition}

\begin{definition}
	\label{def:framed-floer-trajectory} A framed bubbled pre-Floer trajectory (for the Hamiltonian $H$) is a tuple $(u, \Sigma, F)$, where 
\begin{itemize}
    \item $\Sigma$, the \emph{domain} of the trajectory, is a genus zero nodal curve with two marked points $p_-, p_+$, equipped with geodesic segments from $p_{j-1}$ to $p_j$ $S^2_J$, with $S^2_1, \ldots, S^2_j, \ldots, S^2_{r_u}$ the sequence of components  from $p_- = p_0$ to $p_+ = p_{r_u}$,
    \item The map $u: \Sigma \setminus \{p_0, \ldots, p_{r_u}\} \to M$ is a smooth function such that the integral $2$-form $\Omega_u \in \Omega^2(\Sigma)$ defined subsequently has a strictly positive integral on every unstable component. The form $\Omega_u$ is defined to be $u|_{S^2_\pi}^* \Omega_0$ (see \ref{eq:background-form}) on each component $S^2_\pi$ of $\Sigma$ that is not one of the $S^2_j$, and on each $S^2_j$ equal to the extension by zero of the form denoted $\Omega_u$ in (\eqref{eq:floer-stabilizing-form}) under any preferred biholomorphism $S^2_j \setminus \{p_{j-1}, p_j\} \simeq \R \times S^1$.
    \item Additionally, the map is required converge uniformly to closed Hamiltonian orbits $(x_- = x_0, x_1, \ldots, x_{r_u} = x_+)$ with $x_i \in Fix(H)$ the limiting condition the ends corresponding to the $p_i$ under the biholomorphisms described above; we write $\lim_{z \to p_i} u(z) = x_i$. 
    \item Moreover, writing $u^*g$ for the pull-back of the metric on $M$ by $u$, and $dV_{u^*g}$ for the associated volume form (which will in general not be defined on the $p_i$), we have the volume of $\Sigma$ with respect to $dV_{u^*g}$ is finite. 
    \item The datum $F = (f_0, \ldots, f_d)$ is a $\C$-basis for $H^0(L_{\Omega_u})$. 
\end{itemize}   
\end{definition}

Two framed bubbled pre-Floer trajectories $(u_1, \Sigma_1, F_1)$ and $(u_2, \Sigma_2, F_2)$ are \emph{equivalent} when there is biholomorphsm $f:\Sigma_1 \to \Sigma_2$ intertwining the marked points, geodesic segments, the map $u$, and such that $f_*F_1 = cF_2$ for some nonzero complex number $c$. 

We note that due to the finiteness of the volume, there is a nongenerate Hermitian inner product on $H^0(L_{\Omega_u})$ given by 
\begin{equation}
    H(f, g) = \int_{\Sigma} \langle f, g \rangle_\Omega dV_{u^*g}.
\end{equation}
Thus, to any basis $F$, we have an associated Hermitian matrix $H(F)_{ij} = H(f_i, f_j)$.
There is an action of $U(d+1)$ on bases $F$ and thus on framed bubbled pre-Floer trajectories; this action does \emph{not} involve the form $H$ above. Given an element $U \in U(d+1)$, we have that $H((UF)) = UH(F)U^{-1}$. 

The matrix $H(F)$ has positive eigenvalues, so $\log H(F)$ is a well-defined Hermitian matrix. The equivalence class. Write $i\mathfrak{u}_n$ for the set of $n\times n$ Hermitian matrices; then the equivalence class of $H(F) \in i \mathfrak{u}_{d+1}/\R I$ only depends on the equivalence class of the framed bubbled pre-Floer trajectory. 

Any framed bubbled pre-Floer trajectory $(u, \Sigma, F)$ gives a well-defined map 
$u_F: \Sigma \to \P^{d_1}$ given by 
\begin{equation}
   u_F(x) = [f_0(x), \ldots, f_d(x)], \text{ under any $\C$-isomorphism } L_x \simeq \C.
\end{equation}
Thus, we have a map $\phi_F: \Sigma \to \mathcal{C}(\CP^d; d)$  identifying $\Sigma$ with a fiber of the universal curve over $u_F$. Moreover, the $U(d+1)$-action on bases $F$ makes $((u, \Sigma, F) \mapsto u_F$ into a $U(d+1)$-equivariant map. 

Conversely, to stable map $u_F: \Sigma \to \P^d$ such that the degree of each component agrees with the degree of a fixed closed integral $2$-form $\Omega$ on that component, we can associate a basis $F$ of $H^0(L_\Omega)$ that is well defined up to scaling every element of $F$ by a nonzero complex number. Such a scaling action changes the matrix $H(u, \Sigma, F)$ by scaling the matrix by a positive real number, and we write the corresponding element of $i\mathfrak{pu}_{d+1} := i\mathfrak{u}_{d+1}/\R id$ as $H(u, \Sigma, u_F)$. 

\begin{remark}
It is fundamental that $H(u, \Sigma, u_F)$ actually depends on $u$, and \emph{cannot} be recovered from the point $u_F \in \mathcal{F}(\CP^d,d)$. 
\end{remark}

The set of equivalence classes of framed bubbled pre-Floer trajectories can thus be viewed as a subset of the space of choices of an integer $d$, an element $u_F \in \mathcal{F}(\CP^d, d)$, and a smooth map from the fiber of the universal curve $u: \mathcal{C}(\CP^d, d)_{u_F}\setminus \{p_j\}_{j=0}^{r_u} \to M$ satisfying several conditions. We use this identification to equip this set with the topology induced from the Hausdorff topology on the closures of the graphs of these maps in $\mathcal{C}(\CP^d, d) \times M$. The different choices of total degree $d$ give different components of the space of maps. In particular, the map from the degree $d$ classes of framed bubbled pre-Floer trajectories to $\mathcal{F}(\CP^d,d)$ are continuous and equivariant with respect to the natural $PU(d+1)$ action on this component.

\subsection{Perturbation data}
In this section we explain the kind of data we need to specify to build a candidate Kuranishi chart using the McLean stabilization technique.  

We say that a differential form on a nodal curve is a differential form on its normalization. We write $\Omega^{0,1}(\mathcal{C}(\CP^\ell, d)_u, C^\infty(TM))$ for the $(0,1)$-forms on the fiber of $\mathcal{C}(\CP^\ell, d)$ over $u \in \mathcal{F}(\CP^\ell, d)$ which are valued in vector fields on $M$. We write $\tilde{\Omega}^{0,1}$ for the (\emph{finite dimensional}) vector bundle of fiber-wise differential forms on $\mathcal{C}(\CP^\ell, d)$ over $\mathcal{F}(\CP^\ell, d)$ which are defined away from the nodes and the marked points; this is a vector bundle over $\tilde{\mathcal{C}}(d; d_1) = \mathcal{C}(d; d_1) \setminus \{\text{nodes, points}\}$. An element of $\tilde{\Omega}^{0,1}(\mathcal{C}(\CP^\ell, d)_u, C^\infty(TM))$ defines a section of  $\tilde{\Omega}^{0,1}|_u \tensor TM$ over $\mathcal{C}(\CP^\ell, d)_u\setminus\{\text{nodes, points}\} \times M$.

\begin{definition}
\label{def:perturbation-data}
Let $S$ be a subset of a product $\mathcal{F}(\CP^{\ell_1}, d_1) \times \ldots \times \mathcal{F}(\CP^{\ell_r}, d_r)$ which is preserved under the action of $G \subset U(\ell_1+1) \times \ldots \times U(\ell_r+1)$.
An \emph{perturbation datum} on $S$ is a $G$-equivariant vector-bundle $V$ on $S$ together with a smooth map of $G$-equivariant vector bundles over $(\tilde{\mathcal{C}}(\CP^{\ell_1}, d_1) \times \ldots \times \tilde{\mathcal{C}}(\CP^{\ell_r}, d_r))|_S \times M$ of the form
\begin{equation}
    \lambda: \tilde{\Pi}^*V \to \bigoplus_{j=1}^r\tilde{\Omega}^{0,1}_j \tensor TM 
\end{equation}
where
\begin{equation}
\tilde{\Pi}: (\tilde{\mathcal{C}}(\CP^{\ell_1}, d_1) \times \ldots \times \mathcal{F}(\CP^{\ell_r}, d_r))|_S \times M \to S
\end{equation}
is the projection to all but the last factor composed with the projection to $S$, $\tilde{\Omega}^{0,1}_j$ is the fiber-wise differential forms over $\mathcal{F}(\CP^{\ell_j}, d_j)$, and such that there is an open neighborhood $U$ of the nodal locus and the set of marked points for which the linear maps $\lambda_{z, m}$ are identically zero whenever $z \in U$.

Equivalently, a perturbation datum is a collection of linear maps 
\begin{equation}
\label{eq:peturbation-data-equiv-formulation}
    \lambda_u: V_u \to \bigoplus_{j=1}^r \Omega^{0,1}(\mathcal{C}(\CP^{\ell_j}, d_j)_{u_j}, C^\infty(TM)), (u_1, \ldots, u_r) \in S
\end{equation}
which are equivariant with respect to the $G$ action on all data, satisfy a natural smoothness condition, and are zero in open neighborhoods of the nodes and the marked points.
\end{definition}

An isomorphism of perturbation data $\Lambda_1 = (V_1, \lambda_1) \to \Lambda_2 = (V_2, \lambda_2)$ is an isomorphism map of $G$-bundles $\phi: V_1 \to V_2$ such that $\lambda_2 \circ \Pi^* \phi = \lambda_1$. 

The \emph{direct sum} of two perturbation data $\Lambda_i = (V_i, \lambda_i), i=1,2$ over $S$ is a perturbation datum over $S$ given by $\Lambda_1 \oplus \Lambda_2 = (V_1 \oplus V_2, \lambda_1 + \lambda_2)$, with 
\begin{equation}
    \lambda_1 + \lambda_2: \Pi^*V_1 \oplus \Pi^* V_2 \to \bigoplus_{j=1}^r(\tilde{\Omega}_j^{0,1} \tensor TM , \; (\lambda_1 + \lambda_2)(v_1, v_2) = \lambda_1(v_1) + \lambda_2(v_2). 
\end{equation}

The \emph{product} of two perturbation data $(\Lambda_1, \Lambda_2)$ over $(S_1, S_2)$ is the perturbation datum over $S_1 \times S_2$ given by $\Lambda_1 \times \Lambda_2 = (V_1 \times V_2, \lambda_1 \times \lambda_2)$ with 
\begin{equation}
(\lambda_1\times \lambda_2)_{((u_1, \ldots, u_{r_1}), (u'_1,\ldots, u'_{r_2}))}(v_1, v_2) = (\lambda_1(u_1), \lambda_2(u_2)).
\end{equation}

Given a $G$-equivariant vector bundle $V$ on $S$, we have an associated \emph{zero perturbation datum} $\Lambda^V_0 = (V, 0)$. A \emph{stabilization} of a perturbation datum $\Lambda$ is any perturbation datum given by $\Lambda \oplus \Lambda^V_0$ for some $G$-equivariant vector bundle $V$. We say that a perturbation datum $\Lambda =(V, \lambda)$ is \emph{convenient} if $G_x$ acts faithfully on $V_x$ for each $x \in S$, and we say that a stabilization is \emph{convenient} if it is obtained by stabilizing by a zero perturbation datum which is convenient. Finally, we say that a perturbation datum is \emph{elementary} if $V \simeq \bar{V} \times S$ as a an equivariant vector bundle, where $\bar{V}$ is a $G$-representation and the $G$-action on $\bar{V} \times S$ is the diagonal action. 

We have the following 
\begin{proposition}[Lashof \cite{Lashof1979}, Prop. 1.1]
\label{prop:complement-bundle}
    Let $V$ be a $G$-equivariant vector bundle over a $G$-space $S$ with finite covering dimension. Then there exists a $G$-vector bundle $W$ over $X$ such that $V \oplus W \simeq \bar{V} \times S$ for $\bar{V}$ a $G$-representation.
\end{proposition}

In particular, every perturbation datum can be stabilized to an elementary perturbation datum, and the stabilization can be chosen to be convenient. 
Finally, given $S' \supset S$, $S \subset \mathcal{F}(d; d_1)$, any perturbation datum $\Lambda_{S'} = (V, \lambda)$ over $S'$ can be restricted to a perturbation datum $\Lambda_S$ over $S$ by restricting $V$ to $S$ and $\Lambda$ to $\Pi^{-1}(S)$. Conversely, if $\Lambda_S$ is the restiction of $\Lambda_{S'}$ we call $\Lambda_{S'}$ is an extension of $\Lambda_S$. Since extending smooth functions from submanifolds of manifolds with corners is straightforward, we have that if $S = \mathcal{F}(d; d_1)(T)$ for $T \in 2^{[d]}$, then an elementary perturbation datum over $S$ can always be extended to $S' = \mathcal{F}(d; d_1)(T')$ for $T \leq T' \in 2^{[d]}$.  Moreover, if the elementary perturbation datum over $S$ restricted to a zero datum on all strata strictly contained in $S$, then the extension to $S'$ can be chosen to restrict to the zero datum on all strata not strictly containing $S$; we will call such an extension \emph{minimal}.

\begin{remark}
The notion of perturbation datum described above is a natural generalization of the notion of a perturbation datum in \cite[Definition~6.11]{AMS}. The idea to use perturbation data as in Definition \ref{def:perturbation-data} with $V$ fixed as the trivial bundle, i.e. the data of a $U(n)$-equivariant map 
\begin{equation}
\label{eq:basic-soft-perturbation}
    \lambda: \R^k \to \{ \xi \in C^\infty(\mathcal{C}(\CP^d,d) \times M,\Omega^{0,1}_{\mathcal{C}(d;d)/\mathcal{F}(d;d)} \tensor TM) \,| \, \xi=0 \text{ near nodes }\}
\end{equation}
was explained to the author in general outline by Mark McLean \cite{mclean_secret}, although the idea to use the existence of forms $\Omega_u$ as in Section 2 to make sense of McLean Stabilization for Hamiltonian Floer trajectories, and the many compatibility conditions needed to set up Floer homology, were not present in this interaction, had to be worked out by the author.
\end{remark}

\paragraph{Induction and restriction of perturbation data}
If we have a perturbation datum $\Lambda = (V, \lambda)$ over $S$ which is preserved by $H$, and $S'$ is a (generalized) induction of $S$ along $H \subset G \supset V$, then $\Pi^{-1}(S')$ must be the (generalized) induction of $\Pi^{-1}(S)$. Thus, applying the generalized induction construction to the vector bundle $V$ and the map $\lambda$ defines a new perturbation datum $Ind_{H}^{G, V}\Lambda$, which is $G$-equivariant when $V = G$ in which case we drop $V$ from the notation, and is otherwise $H$-equivariant. 

In particular, by utilizing Lemma \ref{lemma:induction-relation-between-strata}, we conclude
\begin{lemma}
    Given perturbation data $\Lambda_{S_i} = (V_{S_i}, \lambda_{S_i})$, $S_i \subset [d_i-1]$, $i=1, \ldots, r$ on each of $\mathcal{F}(\CP^{d_i}, d_i)(S_i)$, we have a canonically associated perturbation datum $Ind(\Lambda_{S_1}, \ldots, \Lambda_{S_r}) = (V_{\vec{S}}, \lambda_{\vec{S}})$, $\vec{S} = (S_1, \ldots, S_r)$ on $\mathcal{F}_{S}(\CP^d, d)$, where 
    \begin{equation}
    \label{eq:S-composition}
        S = S(S_1, \ldots, S_r) = S_1 \cup \{d_1\} \cup S_2 + d_1 \cup \{d_1+d_2\} \cup \ldots \cup S_r + \sum_{i=1}^{r-1} d_i.
    \end{equation} 
    Here, 
    $V_{\vec{S}}$ is the restriction to $\mathcal{F}_{\vec{S}}(\CP^d, d)$ of the bundle on 
    \begin{equation}
        \tilde{\mathcal{F}}_{\vec{S}}(\CP^d, d) = \mathcal{F}(\CP^d, d_1)(S_1) \times \ldots \times \mathcal{F}(\CP^d, d_r)(S_r)
    \end{equation}
    with fiber over 
    $(u_1, \ldots, u_r)$ given by 
    \begin{equation}
    \label{eq:vector-bundle-fiber-product}
        \bigoplus_{i=1}^r (Ind_{PU(d_i+1)}^{PU(d+1)} V_{S_i})_{u_i} 
    \end{equation}
    and 
    \begin{equation}
        \lambda_{\vec{d}}(v_1, \ldots, v_{r}) = \sum_{i=1}^r (Ind_{(U(d_i+1)}^{U(d+1)}\lambda_{S_i})(v_i)
    \end{equation}
    where we extend the terms of the sum by zero away from the codomain of $Ind_{(U(d_i+1)}^{U(d+1)}\lambda_{S_i}$.
    
    From this above formula, it is clear that 
    \begin{equation}
    \label{eq:associativity-of-induction}
        Ind(Ind(\Lambda_{d_1^1} \ldots, \Lambda_{d_1^{r_1}}), Ind(\Lambda_{d^2_1}, \ldots, \Lambda_{d_2^{r_2}}), \ldots ) = 
        Ind(\Lambda_{d_1^1}, \ldots, \Lambda_{d_1^{r_1}}, \Lambda_{d_2^1}, \ldots). 
    \end{equation}
    Where $S_j = S(d_j^1, \ldots, d_j^{r_j})$ and we write $\Lambda_{d}$ for a perturbation datum defined on all of $\mathcal{F}(\CP^d, d)$.
    Finally, this datum can be produced from the $\Lambda_{S_i}$ by taking products, restrictions, and (generalized) inductions.
\end{lemma}

\begin{definition}
\label{def:compatible-choice-of-perturbation-data}
The data of a \emph{compatible choice of peturbation data} (up to degree $D$) is a choice, for each positive integer $d \leq D$, of a perturbation datum $\Lambda_d$ on $\mathcal{F}(\CP^d, d)$, which, for every $S \in 2^{[d-1]}$, restricts to a convenient stabilization of the perturbation datum $Ind_{h(S)}(\Lambda_{h_1}, \ldots, \Lambda_{h_{d-1-|S|}})$ on $\mathcal{F}_{S(d_1, \ldots, d_r)}(\CP^d, d)$, where $(h_1, \ldots, h_{d-1-|S|})= h(S)$. When $D$ is not specified, then this refers to a choice of $\Lambda_d$ for all $d >0$ such that the above conditoins are satisfied for all $D$.
\end{definition}
\begin{remark}
This could be weakened significantly to require that the restriction of $\Lambda_d$ restricts to $\mathcal{F}_{S(d_1, \ldots, d_r)}(\CP^d, d)$ admits $Ind_{h(S)}(h_1, \ldots, h_{d-1-|S|})$ as direct summand, at the cost of additional argument later on. 
\end{remark}


\subsection{Global Kuranishi Charts}
\label{sec:global-kur-charts}

For each positive integer $d$, write $\cM(H, d)$ for the set of bubbled Floer trajectories $u: \Sigma \to M$ in $\bigsqcup_{\tilde{x}_-, \tilde{x}_+} \cM(\tilde{x}_-, \tilde{x}_+)$ such that $\int_\Sigma \Omega_u = d$. By Lemma \ref{lemma:gromov-compactness-with-modified-energy}, this closed subspace is compact. We will write $\cM(\leq d) = \cup_{d' \leq d} \cM(d')$. 

\begin{definition}
	\label{def:thickening}
Let $\Lambda_d = (V_d, \lambda_d)$ be a perturbation datum on $\mathcal{F}(\CP^d; d)$. The Kuranishi chart $\mathcal{K}_{\Lambda_d} = (U(d+1), T_{\Lambda_d} \mathcal{V}_{\Lambda_d} \oplus i\mathfrak{pu}_{d+1}, \sigma_{\Lambda_d} \oplus \sigma_H )$ associated to $\Lambda_d$ has \emph{thickening} $T_{\Lambda_d}$ with elements consisting of equivalence classes of pairs of
\begin{itemize}	\item a framed bubbled pre-Floer trajectory $(u, \Sigma, F)$, 
	\item An element $\eta \in (V_d)_{u_F}$ 
\end{itemize}
such that $\deg u_F|_{S^2_\alpha} = \int_{S^2_\alpha} \Omega_u$ for every component $S^2_\alpha$ of $\Sigma$, and such that we have 
\begin{equation}
\label{eq:perturbed-floer-equation}
\begin{gathered}
    \bar{\partial}u_\alpha + \lambda_i(\eta) \circ di_F = 0 \text{ for } u_\alpha = u|_{S^2_\alpha},  S^2_\alpha\alpha \neq S^2_1, \ldots, S^2_{r_u}, \\
    (d\tilde{u}_j + dt \tensor X_H)^{0,1}_J + \lambda_i(\eta) \circ di_F \circ d \psi_{u, j}= 0 \text{ for } \tilde{u}_j = u|_{S^2_j} \circ \psi_{u, j}
\end{gathered}
\end{equation}
Here as usual $S^2_\alpha$ is a component of $\Sigma$. 

We write $V_{\Lambda_d}$ for the vector bundle over $T_{\Lambda_d}$ with fiber over $((u, \Sigma, F), \eta)$ given by $(V_d)_{u_F}$, and $\sigma_{\Lambda_d}$ for the section of $\mathcal{V}_{\Lambda_d}$ given by $((u, \Sigma, F), \lambda) \mapsto \lambda$, and $\sigma_H$ for the section of the trivial bundle with fiber $i \mathfrak{pu}_{d+1}$ given by $((u, \Sigma, F), \lambda) \mapsto H(u, \Sigma, u_F)$. 
\end{definition}

We also define the Kuranishi chart $\mathcal{K}_{\Lambda_d, \epsilon}=(PU(d+1), T_{\lambda_d}^\epsilon, \mathcal{V}_d \oplus i \mathfrak{pu}_{d+1}, \sigma_{\Lambda_d}\oplus \sigma_H)$ where $T_{\Lambda_d}^\epsilon \subset T_{\Lambda_d}$ consists of those $((u, \Sigma, F), \eta)$ for which $\|\eta\| < \epsilon$. We will take $\epsilon$ to be a sufficiently small positive number depending on $d$.

The space $\mathcal{T}_{\Lambda_d, \epsilon}$ carries a (continous) $PU(d+1)$-action induced from the diagonal action on pre-Floer trajectories and the elements in the fibers of the perturbation datum. The map $\pi$ to $\mathcal{F}(\CP^d, d)$ is equivariant. 

For $S \in 2^{[d-1]}$, we write $T_{\Lambda_d}^\epsilon(S)$ for the set of pairs $((u, \Sigma, F), \eta) \in T_{\Lambda_d}^\epsilon$ with $u_F \in \F(\CP^d, d)(S)$, and we write $\mathcal{K}_{\Lambda_d}^\epsilon(S) = (PU(d+1), T_{\Lambda_d}^\epsilon(S), \mathcal{V}_{\Lambda_d} \oplus i\mathfrak{pu}_{d+1}, \sigma_{\Lambda_d} \oplus \sigma_H)$ (where we have restricted the vector bundles and the sections to $T_{\Lambda_d}^\epsilon(S)$).

Ampleness of $L_{\Omega_u}$ on all unstable components, smallness of $\eta$, and Gromov compactness implies as in \cite{AMS}, 
\begin{lemma}
\label{lemma:bounded-stabilizers}
    The $PU(d+1)$ action on $\mathcal{T}_{\Lambda_d}^\epsilon$ has finite stabilizers. \qedsymbol
\end{lemma}

\begin{remark}
In \cite{AMS} it is claimed that the $U(d+1)$ action has finite stabilizers, which is clearly incorrect as the center $S^1$ of the unitary group fixes projective space. That is the reason for the switch to $PU(d+1)$ throughout this paper; with this modification the argument of \cite{AMS} is correct.
\end{remark}

\begin{lemma}
\label{lemma:induced-topology}
    The induced topology on $\cM(a)$ agrees with the usual Gromov topology. 
\end{lemma}
\begin{proof}
    Recall \cite[p. 96]{PardonHam} that a sequence of Floer trajectories $u_i$ Gromov converges to $u_\infty$ if we can find a set of divisors transverse to the images of all maps sufficiently far in the sequence such that the graphs of the maps $u'_i$ obtained from $u_i$ by stabilizing their domains with divisors converge in the Hausdorff [ on the closures of their graphs in $\overline{\mathcal{C}}^{\$}_{0, r+2}/\Gamma \times M$, where $r$ is the number of added stabilizing points, $\Gamma$ is the group permuting the stabilizing points arising from the same divisor, and $\overline{\mathcal{C}}^{\$}_{0, r+2}$ is the modification of the universal curve over $\overline{\mathcal{M}}(0, r+2)$ obtained by replacing each of the marked points on the path between the two original marked points by a circle (which is possible since the curves being considered are now stable). Choosing bases $F_i$ to extend $u_i$ to framed bubbled Floer trajectories gives a sequence of stable maps from the domains $\Sigma$ to $\CP^d$ with $r+2$ marked points. These latter spaces of stable maps are compact, so a subsequence of these $F_i$ converges, and thus the tuples $((u_i, \Sigma_i, F_i), 0)$ converge in the induced topology. Conversely if such a sequence converges in the induced topology, then adding stabilizing divisors gives stable maps to $\CP^d$ with $r+2$ marked points which still converge, and thus (by forgetting the map to $\CP^d$) the graphs of the $u_i$ in $\overline{\mathcal{C}}^{\$}_{0, r+2} \times M$ Hausdorff-converge as well. 
\end{proof}

Choose $k > 1$, $p> 2$.
Given an element $((u, \Sigma, F), \eta) \in T_{\Lambda_d}$, the linearization of the defining equation \eqref{eq:perturbed-floer-equation} is the Fredholm operator
\begin{equation}
\label{eq:linearized-perturbed-floer-equation}
\begin{gathered}
D_{u, F, \eta, \Lambda_d}: \bigoplus_{S^2_\alpha \subset \Sigma}W^{k, p}(u_\alpha^*TM) \oplus \bigoplus_{j=1}^{r_u} W^{k, p}(\tilde{u}_j^*TM) \oplus (V_d)_{u_F} \to \bigoplus_{S^2_\alpha \subset \Sigma}W^{k-1, p}(u_\alpha^*TM) \oplus \bigoplus_{j=1}^{r_u} W^{k-1, p}(\tilde{u}_j^*TM) \\
D_{u, F, \Lambda_d}(\ldots, \xi_\alpha, \ldots, \xi_1, \ldots \xi_{r_u}, \eta')  = (\ldots, D\bar \partial u_\alpha +  \lambda_i(\eta') \circ di_F, \ldots, D_{\tilde{u}_j} \tilde{u}_j + \lambda_i(\eta') \circ di_F \circ d\psi_{u, j} ) .
\end{gathered}
\end{equation}
Here we use notation as in \eqref{eq:perturbed-floer-equation}, $\Sigma_\alpha \subset \Sigma$ runs over components that are not marked by a geodesic, and $D\bar\partial$ and $D_u$ are the linearization of the Cauchy-Riemann and Floer equations respectively, which can be found in e.g. \cite{McDuffSalamon-BIG}. By elliptic regularity the surjectivity of such an an operator is independent of $(k, p)$. Moreover, if $D_{u, F, \Lambda_d}$ is surjective then so is $D_{u, UFU^{-1}, \Lambda_d}$ for any $U \in U(d+1)$. 

\begin{lemma}
\label{lemma:surjective-perturbation-data-exist-2}
    We can choose a sequence of compatible perturbation data $\Lambda_1, \ldots \Lambda_{d}, \ldots $  as in Definition \ref{def:compatible-choice-of-perturbation-data}, and such that the operators defined \eqref{eq:linearized-perturbed-floer-equation} are surjective for each $(u: \Sigma \to M) \in \cM(d)$ and each framing $F$ of $(u, \Sigma)$.
\end{lemma}
\begin{proof}
    We will construct the perturbation data $\Lambda_d$ by induction on $d$. 
    
    For each $d$, we will choose an elementary perturbation datum $\Lambda^o_d$ on $\mathcal{F}(\CP^d, d)$ which will restrict to a zero datum on every proper substratum of $\mathcal{F}(\CP^d, d)$. The perturbation datum  $\Lambda_d$ will then take the form 
    \begin{equation}
    \label{eq:perturbation-ansatz}
        \Lambda_d = \bigoplus_{S \in 2^{[d-1]}}\Lambda^o_{d, S}
    \end{equation} 
    where $\Lambda^o_{d, [d-1]} = \Lambda^o_d$, and $\Lambda^o_{d, S}$ is a \emph{minimal} extension of a convenient stabilization of $Ind(\Lambda^o_{h_1, [h_1-1]}, \ldots, \Lambda^o_{h_r, [h_r-1]})$ for $(h_j) = h(S)$. 
    
    The compatibility condition on $\Lambda_d$ is satisfied if the above ansatz \eqref{eq:perturbation-ansatz} holds, as well as the condition below: for all $T \supset S$, writing $h(T) = (d_1 \ldots ,d_r)$ and $(d_1^1, \ldots, d_1^{a_1}, d_2^1,  \ldots,  d_2^{a_2}, \ldots, d_r^{a_r})$, where $\sum_{\ell=1}^{a_j} d_j^\ell = d_j$ for $j=1,\ldots r$, we have that $\Lambda^0_{d, S}$ restricted to $\mathcal{F}(\CP^d, d)(T)$ is a convenient stabilization of 
    \begin{equation}
    \label{eq:lambda-o-compatibility}
        Ind(\Lambda^o_{d_1, S(d_1^1, \ldots, d_1^{a_1})}, \ldots, \Lambda^o_{d_r, S(d_r^1, \ldots, d_r^{a_r})}).
    \end{equation}
    (We only need to consider $T \supset S$ by the minimality condition; the restriction of $\Lambda^o_{d, S}$ to all other strata is a zero datum.)
    
    Assume that we have constructed compatible perturbation data $\Lambda_1, \ldots, \Lambda_{d-1}$ satisfying the ansatz \eqref{eq:perturbation-ansatz}, and thus the condition \eqref{eq:lambda-o-compatibility}. Let's write $\Lambda^o_{d, S} = (V^o_{d, S}, \lambda^o_{d, S})$. This latter condition almost specifies $\Lambda^o_{d, S}$ for $S \neq [d-1]$; specifically, it is required to be a convenient stabilization of a perturbation datum that is already defined on each codimension $1$ boundary stratum of $\mathcal{F}(\CP^d, d)$. Let us first understand how to find a compatible extension of the underlying vector bundles of all of these perturbation data specified on the boundary of $\mathcal{F}(\CP^d, d)$ to the interior. By the inductive hypothesis, we can apply \eqref{eq:lambda-o-compatibility} to each of the $\Lambda^o_{d_j, S'}$ appearing in \eqref{eq:lambda-o-compatibility} for the restriction of $\Lambda^o_{d, S}$ to each codimesion $1$ stratum $\mathcal{F}(\CP^d, d)(T)$; associativity of induction then implies that there exist vector bundles $V^o_{d, S, T}$ over $\mathcal{F}(\CP^d, d)(T)$ for $T \supset S$, $T \neq [d-1]$, such that there is a $PU(d+1)$-equivariant inclusion $\iota_{d, S, T', T}: V^o_{d, S, T} \to V^o_{d, S, T'} $ for $T' \supset T$, $|T' \setminus T|=1$; such that composing a sequence of these inclusions along pair of paths in the poset $\{T' \in 2^{[d-1]}: T \leq T' \leq [d-1]\}$ gives the same map; and moreover that $V^o_{d, S, S} = Ind(\Lambda^o_{h_1, [h_1-1]}, \ldots, \Lambda^o_{h_r, [h_r-1]})$ where $(h_j) = h(S)$. The first step is to find a compatible extension of all of these bundles: 
    \begin{lemma}
        We can find a faithful representation $R(\sigma_{d, S})$ such that there are $PU(d+1)$-equivariant bundle inclusions $\iota_{d,S,T}: V^o_{d, S, T} \to R(\sigma_{d, S}) \times \mathcal{F}(\CP^d, d)(T)$ such that 
    $\iota_{d,S,T'} \iota_{d, S, T', T} = \iota_{d,S,T}$. 
    \end{lemma}
    \begin{proof}
    This follows repeated application of \ref{prop:complement-bundle} along inclusions in the poset $\{T' \in 2^{[d-1]}: T \leq T' \leq [d-1]\}$, together with the fact that if we take the direct sum of two maps of bundles, then if one of them is a bundle inclusion then so is the direct sum. 
    \end{proof}
    Now, \eqref{eq:lambda-o-compatibility} specifies $\lambda^o_{d, S}$ on each $\tilde{\Pi}^*V^o_{d, S, T}$, with the property that it is zero on the complement of $\tilde{\Pi}^*V^o_{d, S, T}$ in $\tilde{\Pi}^*V^o_{d, S, T'}|_{\mathcal{F}(\CP^d, d)(T)}$ for $[d-1] \supset T' \supset T\supset S$, $|T' \setminus T| = 1$. By extending $\lambda^o_{d, S}$ by zero to the complements of $\tilde{\Pi}^*V^o_{d, S, T}$ over all $T$, we have defines $\lambda^o_{d, S}: \tilde{\Pi}^* (R(\sigma_{d, S}) \times \mathcal{F}(\CP^d, d))|_{\partial \mathcal{F}(\CP^d, d)}$. It then remains to extend this smooth function from the boundary of $\mathcal{F}(\CP^d, d)$ to the interior to produce $\Lambda^o_{d, S}$ for all $S \neq [d-1]$ satisfying the conditions claimed. 

    By ansatz $\eqref{eq:perturbation-ansatz}$ and the inductive hypothesis, we have that $D_{u, F, \bar{\Lambda}_d}$ is surjective for each $(u, \Sigma, F)$ with $\pi(\phi_F(\Sigma)) \in \partial\mathcal{F}(\CP^d, d)$. By the gluing theorem for linear Cauchy-Riemann operators, this operator is thus surjective for all $(u, \Sigma, F)$ with $\pi(\phi_F(\Sigma)) \in U$ for $U$ some open neighborhood of the boundary. Let $U'$ be the set of bubbled Floer trajectories $(u: \Sigma \to M) \in \cM(d)$ which extend to such a framed bubbled pre-Floer trajectory, and let $K$ be the complement of $U$ in $\cM(d)$. Then $K$ is compact. Given any $(u, \Sigma, F)$ and any element $\theta \in \Omega^{0,1}(\Sigma)$ supported away from the marked points and with $\Sigma$ having only one component, there exists a smooth $(0,1)$-form $\bar{\theta}$ on the universal curve over $\mathcal{F}(\CP^d,d)$ supported away from all marked points and nodes and restricting to zero on a neighborhood over $\partial{F}(\CP^d,d)$, which restricts to $\theta$ on the fiber associated to $(u, \Sigma, F)$. Taking the $PU(d+1)$-orbit of such a $(0,1)$-form defines a perturbation datum $\Lambda_{\theta}$ with a trivial underlying vector bundle (which has fiber the span of the orbit of $\bar{\theta}$). For each $(u: \Sigma \to M) \in K$, choose an $F$ such that $(u, \Sigma, F)$ is a framed bubbled pre-Floer trajectory, and choose a finite collection of $1$-forms $\theta_i^{u}$ on $\Omega^{0,1}(\Sigma)$, which surject onto the cokernel of $D_u$ and are supported away from the marked points of $\Sigma$. Applying the previous construction gives a perturbation datum $\Lambda_u$ such that $\mathcal{D}_{u, F, \Lambda_{u}}$ is surjective, and by continuity $\mathcal{D}_{u', F', \Lambda_u}$ is surjective for $(u', \Sigma', F')$ arising by framing a Floer trajectory sufficiently near $u$. By compactness of $K$ we can take $\Lambda^o_F$ to be the direct sum of $\Lambda_{u, F}$ over a finite collection of $u \in K$. 
\end{proof}

By compactness of $\cM(d)$, by taking a sufficiently small $\epsilon$, we can arrange for $D_{u, F, \Lambda_d}$ to be surjective for all $(u, F, \eta) \in T^{\epsilon}_{\Lambda_d}$

\begin{proposition}
\label{prop:fiberwise-c1loc-on-kuranishi-chart}
The spaces $T_{\Lambda_d}^\epsilon$ are topological $G-\langle k \rangle$-manifolds, and the projection $\Pi$ to $\mathcal{F}(\CP^d, d)$ a fiberwise smooth $C^1_{loc}$ $PU(d+1)$-equivariant topological submersion. 
\end{proposition}
\begin{proof}
    This follows analogously to \cite[Section 6.8]{AMS} from the gluing results of \cite{PardonHam}. 
    
    First, the fibers of $\Pi: T_{\Lambda_d}^\epsilon \to \mathcal{F}(\CP^d, d)$ are naturally smooth manifolds of the expected dimension, as we can identify them with thickened moduli spaces of Floer trajectories with a fixed domain. 
    
    Next, as in Lemma \ref{lemma:stratification-model}, we can identify a neighborhood $V$ of $\Pi((u, \Sigma, F), \eta) \in \mathcal{F}(\CP^{d}, d)$ with an open subset $j: U \to \widetilde{\cM_{0, d+5}}$ of a circle bundle of the real-oriented blowup of $\cM_{0, d+5}$ along its natural stratification; call this identification $\mu_H: V \to U$. Let $K \subset \Pi^{-1}(\Pi((u, \Sigma, F), \eta))$ be a compact neighborhood of $((u, \Sigma, F), \eta)$ in the fiber. We claim, as in \cite[Theorem~6.1]{AMS}, that, shrinking $K$ and $V$ as necessary, that there is a diagram 
    \begin{equation}
    \begin{tikzcd}
        U \times K \ar[r, "g"] \ar[d, "\pi_U"] & T' = \Pi^{-1}(V) \ar[d, "\mu_H \circ \Pi"]\\
        U \ar[r, "j"] & \widetilde{\cM_{0, d+5}}
    \end{tikzcd}
    \end{equation}
    which is a homeomorphism onto its image. Moreover, for each $a \in U$, the restriction $g|_{\{a\} \times K} \to \Pi^{-1}(\mu_H^{-1}(a))$ takes place in $\Pi^{-1}(\mu_H^{-1}(a))$ and is a diffeomorphism onto its image.
    
    Indeed, this is a special case of appendix C of \cite{PardonHam} where $n=0$, i.e. we are working with a single Hamiltonian rather than a continuation map. The map $g$ is defined by \cite[\nopp C.10.3]{PardonHam}  and \cite[\nopp C.10.4]{PardonHam}; the commutativity of the diagram follows from \cite[\nopp C.10.1]{PardonHam}. The horizontal arrow is a homeomorphism onto its image because of Proposition \cite[\nopp C.10.11]{PardonHam}, Lemmas \cite[\nopp C.10.3]{PardonHam} and \cite[\nopp C.10.4]{PardonHam}, all combined with Lemma \cite[\nopp B.12.1]{PardonHam}.
    The restriction of $g$ to $\{a\} \times K$ is obtained by Newton-Picard  iteration (\cite[\nopp C.9.5]{PardonHam} and \cite[\nopp C.9.6]{PardonHam}) with initial conditions smoothly depending on $K$ for each $a \in U$; hence this map is smooth for each fixed $a$.
    
    We have proven that $\Pi$ has a fiberwise smooth structure. The fact that it has a $C^1_{loc}$ structure follows as in \cite[Corollary 6.29]{AMS} but with \cite[\nopp B.11.1]{PardonHam} replaced with \cite[\nopp C.11.1]{PardonHam}. 
\end{proof}

\begin{definition}
A pair of global Kuranishi charts $\mathcal{K}_1 = (G, T_1, V_1, \sigma_1)$ and $\mathcal{K}_2 = (G, T_2, V_2, \sigma_2)$ are \emph{equivalent} if there are $G$-equivariant diffeomorphisms $f: T_1 \to T_2$ and a bundle isomorphism $g: V_1 \to V_2$ such that $\pi_2 \circ g = f \circ \pi_1$, where $\pi_i: V_i \to T_i$ are the vector bundle projections, and $g \circ \sigma_1 = \sigma_2 \circ f$. 
\end{definition}

\begin{proposition}
    \label{prop:induction-relation-between-kuranishi-charts}
    The Kuranishi charts $\mathcal{K}^{\epsilon}_{\Lambda_d}(S)$  can be obtained by induction, restriction, products, taking open subsets,  and equivalences from the Kuranishi charts 
    $\mathcal{L}_\ell$, $\ell < d$. 
\end{proposition}
\begin{proof}

It is clear that since $\Lambda_d$ restricts to a stabilization of $\Lambda_{\vec{S}} = Ind(\Lambda_{S_1}, \ldots, \Lambda_{S_r})$, the Kuranishi chart $K^\epsilon_{\Lambda_d}$ is a stabilization of 
\begin{equation}
    K(S_1, \ldots, S_r)^{\epsilon} = (U(d+1), T(\vec{S}), \mathcal{V}(\vec{S}) \oplus i\mathfrak{pu}_{d+1}, \sigma_{\Lambda_{\vec{S}}} \oplus \sigma_H)
\end{equation}
where
\begin{equation}
    T(S_1, \ldots, S_r) = \{ ((u, \Sigma, F), \eta)) \in T^{\epsilon}_{\Lambda_d}  | u_F \in \mathcal{F}(\CP^d, d)(S), \eta \in (V_{\vec{S}})_{u_F}\}, S \text{ as in } \eqref{eq:S-composition}.
\end{equation}
and $\mathcal{V}(\vec{S})_{((u, \Sigma, F), \eta)} = (V_{\vec{S}})_{u_F}$. 

More generally, taking products, inductions/restrictions, and stabilizations of perturbation data corresponds to taking products, inductions/restrictions, and stabilizations of associated Kuranishi charts.

Thus, we wish to apply Lemma \ref{lemma:induction-relation-between-strata}. The crucial point is to show that the generalized inductions (\ref{eq:generalized-induction}) correspond to \emph{stabilizations} of the corresponding Kuranishi charts. This follows from an explicit computation given below. Given this claim, the lemma follows from associativity of induction (\ref{eq:associativity-of-induction}).

Fix $d_1 + d_2 = d$. Write $\Lambda_a = \Lambda_{d_1}$ and $\Lambda_b = \Lambda_{d_2}$; write $\Lambda_{ab} = Ind(\Lambda_a, \Lambda_b)$. We fix embeddings $\CP^{d_1} \subset \CP^{d_2}$ as the subset with the last $d_2$ coordinates zero, and $\CP^{d_2} \subset \CP^d$ as the subset with the first $d_1$ coordinates zero. They intersect in $p=[0, \ldots, 0, 1, 0, \ldots, 0]$ where the $1$ is in the $d_1+1$-st coordinate. Consider the following Kuranishi charts:
\begin{equation}
\label{eq:k1-chart}
\begin{gathered}
    \mathcal{K}_1 = (U(d_1), T_1^{\epsilon}, \mathcal{V}_{\Lambda_a} \oplus i \mathfrak{pu}_{d_1+1}, \sigma_{\Lambda_1} \oplus \sigma_H)\\
    T_1^\epsilon = \{((u, \Sigma, F), \eta) | u_F \in \mathcal{F}_+(\CP^{d_1}, d_1), \eta \in (V_{\Lambda_b})_{u_F}, \|\eta\|< \epsilon, (\ref{eq:perturbed-floer-equation})\text{ holds }\}; 
\end{gathered}
\end{equation}
\begin{equation}
\label{eq:k2-chart}
\begin{gathered}
    K_2^\epsilon = (U(d_2) , T_2^\epsilon, \mathcal{V}_{\Lambda_b} \oplus i \mathfrak{pu}_{d_2+1}, \sigma_{\Lambda_2} \oplus \sigma_H\} \\
    T^\epsilon_2 = \{((u, \Sigma, F), \eta) | u_F \in \mathcal{F}_-(\CP^{d_2}, d_2), \eta \in (V_{\Lambda_b})_{u_F}, \|\eta\|< \epsilon, (\ref{eq:perturbed-floer-equation})\text{ holds }\};
\end{gathered}
\end{equation}
\begin{equation}
\label{eq:third-kuranishi-chart}
\begin{gathered}
    K_3^{\epsilon, \epsilon} = \{(U(d_1) \times U(d_2), T^{\epsilon, \epsilon}_3, \mathcal{V}_{ab} \oplus i \mathfrak{pu}_{d+1}, \sigma_{\Lambda_{ab}} \oplus \sigma_H\} \\
    T^{\epsilon, \epsilon}_3 = \{((u, \Sigma, F), \eta) | u_F \in \mathcal{F}^{tr}_{S(d_1,d_2)}(\CP^d, d), \eta = (\eta_1, \eta_2) \in (V_{ab})_{u_F}, \|\eta_1\| < \epsilon, \eta_2 < \epsilon, (\ref{eq:perturbed-floer-equation})\text{ holds }\};
\end{gathered}
\end{equation}
\begin{equation}
\label{eq:k3-chart}
\begin{gathered}
    K^{\epsilon}_3  = (U(d_1) \times U(d_2), T^{\epsilon, \epsilon}_3, \mathcal{V}_{ab} \oplus i \mathfrak{pu}_{d+1}) \\
    T^{\epsilon}_3  = \{((u, \Sigma, F), \eta) \in T^{\epsilon, \epsilon}_3 | \|\eta\|< \epsilon\}. 
\end{gathered}
\end{equation}

Here, we can write $\eta = (\eta_1, \eta_2) \in (V_{ab})_{u_F}$ in (\ref{eq:third-kuranishi-chart}) due to (\ref{eq:vector-bundle-fiber-product}). We have an inclusion
\begin{equation}
\begin{gathered}
\label{eq:map-on-base-spaces}
   T_1^\epsilon \times T_2^\epsilon \to T_3^{\epsilon, \epsilon} \\
    (g, ((u_1, \Sigma_1, F_1), \eta_1), ((u_2, \Sigma_2, F_2), \eta_2) \mapsto ((u_1 \# \circ u_2,\Sigma_1 \# \Sigma_2, F_1 \# F_2), \eta_1 + \eta_2)
\end{gathered}
\end{equation}
where the  domain $\Sigma_1 \# \Sigma_2$ is obtained by gluing $p_+ \in \Sigma_1$ with $p_- \in \Sigma_2$, the map $u_1 \# u_2$ is the corresponding concatenation of maps, and $F_1\# F_2$ is obtained as follows. Write $F_1 = (f_0, \ldots, f_{d_1})$ and $F_2 = (f'_{d_1}, \ldots, f'_{d_2})$. Then we have evaluation maps $ev_1: H^0(L_{u_1}) \to (L_{u_1})_p$  and $ev_2: H^0(L_{u_2}) \to (L_{u_2})_p)$ and we must have that $f_0, \ldots f_{d_1-1} \in \ker$ and $f'_{d_1+1}, \ldots, f'_{d_2} \in \ker ev_2$. Then $F_0 \# F_1$ is obtained by globally rescaling $F_1$ and $F_2$ each so that $f_{d_1}$ and $f'_{d_1}$ such that $ev_1(f_{d_1}) = ev_2(f'_{d_1})$. This defines a basis of $L_{u_1 \# u_2}$ which is well defined up to scaling by $\C^*$, and which we denote by $(f_0, \ldots, f_{d-1}, \bar{f}_{d_1}, f'_{d_1+1}, \ldots, f'_d)$.  

The claim is that $\mathcal{K}^{\epsilon, \epsilon}_3$ is equivalent to the stabilization of $K_1^{\epsilon} \times K_2^{\epsilon}$ by the vector bundle $i\mathfrak{pu}_{d+1}^{d_1, d_2}$, which is a trivial bundle with fiber the set of $d_1+1 \times d_1+1$ Hermitian matrices on which only the top right $d_1 \times d_2$ and bottom left $d_2 \times d_1$ block are allowed to be nonzero; this vector space is naturally a $U(d_1) \times U(d_2)$-representation. Clearly $K^{\epsilon}_3$ is just the restriction of $K^{\epsilon, \epsilon}_3$ to an open subset, and an induction of $K^\epsilon_d$ along $U(d_1) \times U(d_2) \subset PU(d+1)$ gives the Kuranishi chart associated to $Ind(\Lambda_1, \Lambda_2)$.

To see this, we first explain the map $i \mathfrak{pu}_{d_1 + 1} \oplus i \mathfrak{pu}_{d_2+1} \to i \mathfrak{pu}_{d+1}|_{T_1^\epsilon \times T_2^\epsilon}$ covering the map (\ref{eq:map-on-base-spaces}) and sending $\sigma_H \oplus \sigma_H$ to $\sigma_H$. This map is the map which sends the equivalence class of $\log H_1$, where $H_1$ is an inner product on $H^0(L_{u_1})$ up to scale expressed in the basis $(f_0, \ldots, f_{d_1})$, and the equivalence class of $\log H_2$, where $H_2$ is in turn an inner product on $H^0(L_{u_2})$ expressed in the basis $(f_0, \ldots, f_{d_1})$, to the equivalence class of $\log H$, where $H$ is an inner product on $H^0(L_{u_1 \# u_2})$ expressed in the base $F_1 \# F_2$ determined by the condition that it restricts to $H_1$ on $F_1$, to $H_2$ on $F_2$, and such that we have chosen $(H_1, H_2)$ in their equivalence classes such that $H(\bar{f}_d, \bar{f}_d) = H_1(f_d, f_d) = H_2(f'_d, f'_d)$ and that $f_0, \ldots, f_{d-1}$ are orthogonal to $f'_{d_1}, \ldots, f'_d$. 

In fact, writing $(a_{i,j})_{i,j=0}^{d_1} \in i \mathfrak{pu}_{d_1 + 1}$ and $(b_{i,j})_{i,j=d_1}^{d_2} \in i \mathfrak{pu}_{d_2+1}$, the equivalence class of $\log H$ is represented by $(A_{i,j})_{i,j=0}^{d_2}$ given by 
\begin{equation}
\label{convenient-H-map}
A_{i,j} = \begin{cases}
a_{i,j} & \text{if }0 \leq i, j \leq d_1 \\
b_{i,j} & \text{if } d_1 \leq i, j \leq d_2, i \neq j \\
b_{i,i} + a_{d_1,d_1} - b_{d_1,d_1} & \text{if }  d_1 \leq i, j \leq d_2, i = j \\
0 & \text{else}.
\end{cases}
\end{equation}
In particular we see that this is a linear bundle map that is equivariant with respect to the $U(d_1) \times U(d_2)$ action. 

Now, given a matrix $M \in i\mathfrak{pu}_{d+1}^{d_1, d_2}$ as well an inner product $H$ as above, there is a \emph{unique} element $g_H(M) \in V^{tr}_{d, d_1}\in PU(d+1)$ such that the inner product expressed of the basis $F_1 \#_{g_H(M)} F_2(f_0, \ldots, f_{d-1}, g_H(M)\bar{f}_{d_1}, \ldots, g_H(M)f'_{d})$ is $H+M$. This defines a diffeomorphism 

\begin{equation}
\label{eq:map-on-base-spaces-2}
\begin{gathered}
   i \mathfrak{pu}^{d_1,d_2}_{d+1} \times T^\epsilon_1 \times T^\epsilon_2 \to T^{\epsilon, \epsilon}_3 \simeq V^{tr}_{d, d_1} \times_{U(d_1) \times U(d_2)} T^\epsilon_1 \times T^{\epsilon}_2 \\
   (M, ((u_1, \Sigma_1, F_1), \eta_1), ((u_2, \Sigma_2, F_2), \eta_2)) \mapsto ((u_1 \# u_2, \Sigma_1 \# \Sigma_2, F_1 \#_{g_{H}(M)} F_2), (\eta_1 + g_H(M) \eta_2)
\end{gathered}
\end{equation}
where $H = \sigma_H(u_1 \# u_2, \Sigma_1 \# \Sigma_2, F_1 \# F_2)$. Covering this map we have a diffeomorphism

\begin{equation}
    \label{eq:map-on-total-spaces}
    \begin{gathered}
     i \mathfrak{pu}^{d_1,d_2}_{d+1} \times i \mathfrak{pu}^{d_1, d_2}_{d+1} \mathcal{V}_a \times i \mathfrak{pu}_{d_1+1} \times \mathcal{V}_b \times i \mathfrak{pu}_{d_2+1} \to \mathcal{V}_{ab} \times i \mathfrak{pu}_{d+1}\\
     (M, N, ((u_1, \Sigma_1, F_1), \eta_1), \eta'_1, H_1, ((u_2, \Sigma_2, F_2), \eta_2), \eta'_2, H_2) \mapsto \\((u_1 \# u_2, \Sigma_1 \# \Sigma_2, F_1 \# g_H(M) F_2), \eta_1 + g_H(M) \eta_2), \eta_1' + g_H(M)\eta_2', H'+N)
     \end{gathered}
\end{equation}
where $H' = H(H_1, H_2)$. 

This proves that $\mathcal{K}^{\epsilon, \epsilon}_3$ is equivalent to a stabilization of $\mathcal{K}_1^\epsilon \times \mathcal{K}^2_\epsilon$. 

The final relation that needs to be shown is that $\mathcal{K}^{\epsilon}_3$ is induced to $\mathcal{K}_{\Lambda_d}$ under the inclusion $U(d_1) \times U(d_2) \subset PU(d+1)$. The only nontrivial point is to see that the induction acts trivially on the $i\mathfrak{pu}_{d+1}$ factor, which follows from Lemma \ref{lemma:induction-of-vector-bundles}. 

\end{proof}
\section{Compatibly smoothing global charts}
\label{sec:compatible-smoothing}
In this section, we show that we can stabilize the system of Kuranishi charts described in Section \ref{sec:global-kur-charts} such that after the stabilization, the new system of Kuranishi charts admits compatible smooth structures. The basic challenge is to extend equivariant stable smoothing theory the setting of of $\langle k \rangle$-manifolds. 
\subsection{Brief review of microbundles}
\begin{definition}
A microbundle of rank $n$ over a topological space $X$ is a diagram $\{ X \xrightarrow{s} E \xrightarrow{p} X \}$ where $E$ is a topological space, and $p$ and $s$ are continuous maps (called the \emph{projection} and the \emph{section}, respectively) such that $p \circ s = id_X$, and moreover for each $x \in X$, there are neighborhoods $x \in U \subset X$ and $s(x) \in V \subset E$, as well as a homeomorphism $h: U \times \R^n \to V$ with $p \circ h = pr_U$, and $h |_{U \times \{0\}} = s$. The space $E$ is the \emph{total space} of the microbundle. When the context is clear, the we will abbreviate the microbundle as $E$.

A morphism of microbundles is an equivalence class of continuous map between the total spaces which commute with the projections and the sections, respectively. Two such maps are equivalent if they agree in a neighborhood of the image of the section. 
\end{definition}

Given a vector bundle $\pi: V \to X$, the \emph{associated microbundle} is $\{X \xrightarrow{0} V \xrightarrow{\pi} X$, which we will also denote by $V$. Given a microbundle $E$, a vector bundle $V$ together with an isomorphism of the associated microbundle $V \to E$ is called a \emph{vector bundle lift} of $E$. 

Given a pair of microbundles $E_1, E_2$ over $X$, the direct sum $E_1 \oplus E_2$ is the microbundle 
\begin{equation}
\label{eq:microbundle-direct-sum}
    X \xrightarrow{s_1 \times s_2 \circ \Delta} (p_1 \times p_2)^{-1}(\Delta(X))  \xrightarrow{p_1 \times p_2} X. 
\end{equation}
where the total space is a subspace of $E_1 \times E_2$.

Given a map of spaces $f: X \to Y$ and a microbundle $E$ over $Y$, the pullback microbundle $f^*E$ is given by 
\begin{equation}
    X \xrightarrow{x \mapsto (s(f(x), x)} f^*E = \{(e, x) \in E \times X | p(e) = f(x) \} \xrightarrow{(e, x) \mapsto x} X.
\end{equation}

These notions admit equivariant versions. Here, $G$ is a compact Lie group. 
\begin{definition}
Let $X$ be a $G$-space. A $G$-microbundle on $X$ is a microbundle $X \xrightarrow{s} E \xrightarrow{p} X$ where $E$ is a $G$-space and such that $s$ and $p$ are $G$-equivariant maps. A morphism between microbundles $E_1, E_2$ over $X$ is a morphism of microbundles such that the map between total spaces is $G$-equivariant. The direct sum (\ref{eq:microbundle-direct-sum}) of two $G$-microbundles is a $G$-microbundle via restricting the diagonal action of $G$ on $E_1 \times E_2$ to $E_1 \oplus E_2$. 
\end{definition}

The associated microbundle of a $G$-vector bundle over a $G$-space is a $G$-microbundle; correspondingly, a $G$-vector bundle lift of a $G$-microbundle $E$ is a $G$-vector bundle which has an associated $G$-microbundle equipped with an isomorphism to $E$. 

\begin{definition}
Let $X$ be a topological manifold. The \emph{tangent microbundle} $T_\mu X$ is the microbundle $X \xrightarrow{\Delta} X \times X \xrightarrow{p_1} X$, where $\Delta$ is the diagonal map and $p_1$ is projection to the first factor. If $G$ is a topological $G$-manifold $T_\mu X$ is a $G$-microbundle via the diagonal action on $X \times X$.
\end{definition}

If $X$ is a smooth manifold, then the exponential map gives a canonical vector bundle lift $TX$ of $T_\mu X$. If $X$ is further equipped with a smooth $G$-action, then, by choosing a $G$-equivariant Riemannian metric and utilizing the $G$-equivariance of the exponential map, this vector bundle lift becomes a $G$-vector bundle lift. These lifts are canonical up to contractible choice, since the space of ($G$-equivariant) Riemannian metrics on $X$ is contractible. In particular, any pair of lifts is isotopic in the following sense:
\begin{definition}
Let $X$ be a space, $E \to X$ be a microbundle, and $V \to X$ be a vector bundle. An isotopy of vector bundle lifts $\phi_0: V_0 \to E$, $\phi_1: V_1 \to E$ is a vector bundle lift $\phi: V \to \pi^*E$, where $\pi: M \times [0,1] \to M$ is the projection, such that we have isomorphisms $V|_{M \times \{i\}} \simeq V_i$ which compose with $\phi|_{M \times \{i\}}$ to give $\phi_i$, for $i = 0,1$.
\end{definition}

The basic theorem of equivariant stable smoothing theory is 
\begin{proposition} [Lashof \cite{Lashof1979}]
\label{prop:lashof-smoothing}
Let $M$ be a topological manifold equipped with a continuous action of $G$. Assume that 
\begin{enumerate}
    \item There are finitely many orbit types, and
    \item The microbundle $T_\mu M$ admits a $G$-vector bundle lift $E$. 
\end{enumerate}
Then there is a $G$-representation $V$ or which the product $V \times M$ admits a $G$-equivariant smooth structure. Any two such smooth structures are \emph{stably slice concordant} in the following sense. Given any two such smoothings of $V_0 \times M$ and $V_1 \times M$ associated to isotopic vector bundle lifts $E_0 \to T_\mu M$, and $E_1 \to T_\mu M$, there is a vector space $V' \subset V_0 \oplus V_1$ and a smooth structure on $V' \times M \times [0,1]$ such that 
the smooth structure on $V' \times M \times \{j\}$ is induced from that on $V_j \times M$  for $j = 0,1$, and such that the projection to $[0,1]$ is a submersion. 
\end{proposition}
We will unpack this result and generalize it to the setting of topological $\langle k \rangle$-manifolds. 

\subsection{Smoothing of $G-\langle k \rangle$-manifolds}
We first define the tangent microbundle of a $\langle k \rangle$-manifold. 

\begin{definition}
Given a topological $G-\langle k \rangle$-manifold $M$, we define its tangent microbundle $T_\mu M$ to the be the pullback of $T_\mu \mathcal{D}(M)$ under the canonical inclusion $M \to \mathcal{D}(M)$. 
\end{definition}

\begin{definition}
Let $M$, $B$ be topological $G-\langle k \rangle$-manifolds and let $\pi: M \to B$ be a map of $G-\langle k \rangle$-manifolds. Let $p \in M$ and let $b = \pi(p) \in B$. A \emph{product neighborhood} of $p$ is a homeomorphism $\iota: W \to W \cap \pi^{-1}(b) \times \pi(W)$, where $W$ is a neighborhood of $p$ satisfying 
\begin{itemize}
    \item $\pi \circ \iota^{-1}$ is the projection to $\pi(W)$, and
    \item $\pi(W)$ is an open subset of $B$, and 
    \item $\iota|_{W \cap \pi^{-1}(b)}: W \cap \pi^{-1}(b) \to \{b\}$ is the identity map, and 
    \item We have that $\iota(W \cap M(S)) = W \cap \pi^{-1}(b) \times \pi(W) \cap B(S)$ for each $S \in 2^{[k]}$.
\end{itemize}
We say that $\pi$ is a \emph{topological submersion} if every point in $M$ admits a product neighborhood. In that case, the fibers of $\pi$ are manifestly \emph{topological manifolds}, and the \emph{vertical tangent microbundle} $T^{vt}_\mu(\pi)$ of $\pi$ is the microbundle 
\begin{equation}
    M \xrightarrow{\Delta} M \times_B M \xrightarrow{\pi \times \pi} M
\end{equation}
where $M \times_B M$ is the fiber product of $\pi$ with itself, $\Delta$ is the diagonal, and $\pi \times \pi$ is the canonical map out of the fiber product. This is manifestly a $G$-equivariant microbundle under the diagonal action of $G$ on the total space. 
\end{definition}

By looking at doubled charts and using the functoriality of doubling, one sees that
\begin{lemma}
    Let $\pi: M \to B$ be a topological submersion. Then the natural map of topological manifolds $\mathcal{D}(\pi): \mathcal{D}(M) \to \mathcal{D}(B)$ is a topological submersion. 
\end{lemma}
There is a natural map $\tau \mathcal{D}(pi): T_\mu \mathcal{D}(M) \to T_\mu\mathcal{D}(B)$ defined in \cite[Eq. 4.11, Corollary 4.26]{AMS}, and the restriction of this map to $M$ defines a map $\tau: T_\mu M \to T_\mu B$. 
\begin{proposition}
\label{prop:submersion-microbundle-splitting}
    There is a map $P: T_\mu M \to T^{vt}_\mu M$ such that $P \oplus \tau: T_\mu M \to T^{vt}_\mu M \oplus \pi^* T_\mu B$ is an isomorphism. 
\end{proposition}

The rest of this subsection is dedicated to proving this proposition. The arguments are a direct adaptation of \cite[Section 4.4]{AMS}.

Each $g \in G$ sends a product neighborhood $\iota$ of $p$ to a new product neighborhood $g_* \iota$ of $g \cdot p$ given by the formula 
\begin{equation}
    g_* i: g(W) \to g(W \cap \pi^{-1}(b)) \times g(\pi(W)), g_*(\iota)(w) = g \iota (g^{-1} \cdot w). 
\end{equation}
\begin{definition}
A $G_p$-invariant product neighborhood is a product neighborhood of $p \in M$ which is $G_p$-invariant under the action above. We say that $G$ is a \emph{$G$-equivariant topological submersion} if every point $p \in M$ admits a $G_p$ invariant product neighborhood.
\end{definition}
 \begin{definition}
 We say that $G$ acts on $\R^n$ \emph{linearly relative to $v \in \R^n$} if the the action becomes linear after translating $v$ to $0$.
 
 A $G$-action on a topological $\langle k\rangle$-manifold is \emph{locally linear} if for each point $p$ in this manifold there is a chart around $p$ such that $G_p$ acts linearly in this chart relative to the image of $p$; such a chart is called a \emph{linear chart}. We say that $\pi$ is \emph{fiberwise locally linear} if it is a topological submersion, and if for each $b \in B$, the group $G_b$ is locally linear on the manifold $\pi^{-1}(b)$.  
 \end{definition}
 
 By using the existence of collars $G$-invariant collars of $G-\langle k \rangle$-manifolds (Lemma \ref{lemma:collars-and-doubles-are-smooth}), we have 
 \begin{lemma}
     Given a smooth $\langle k \rangle$-manifold $B$, the $G$-action on $B$ is \emph{strongly linear} in the sense that for each $p \in B(S)$, there exists a smooth $\langle k \rangle$-manifold chart around $p$ from $\R^\ell_+ \times \R^{n - \ell}$ such that $G$ acts \emph{trivially} on the coordinates of $\R^\ell_+$.
 \end{lemma}

In particular,  the lemma above implies that the action of $G \times (\Z/2)^k$ is linear on the double is linear.

\begin{proof}[Proof of Proposition \ref{prop:submersion-microbundle-splitting}.]
Apply apply \cite[Prop. 4.25, Corollary 4.26]{AMS} to the topological submersion $\mathcal{D}(M) \to \mathcal{D}(B)$ to produce a natural morphism of microbundles $P: T_\mu \mathcal{D}(M) \to T^{vt}_{\mathcal{D}(\pi)}\mathcal{D}(B)$ such that $P \oplus \tau \mathcal{D}(\pi): T_\mu \mathcal{D}(M) \to  T^{vt}_{\mathcal{D}(\pi)}\mathcal{D}(B) \oplus \mathcal{D}(\pi)^* T_\mu \mathcal{D}(B)$ is an isomorphism. Restrict these maps to to $M$ along the inclusion into $\mathcal{D}(M)$. $\qed$
\end{proof}

\paragraph{Fiberwise $C^1_{loc}$ structures and smoothing $G-\langle k \rangle$-manifolds}
We assume that we have topological submersion of topological $G-\langle k \rangle$-manifolds $\pi: M \to B$ such that the fibers of $\pi$ (which are topological manifolds) are each equipped with smooth structures. The definition \cite[Definition~4.27]{AMS} of a pair of $C^1_{loc}$-compatible product neighborhoods of $\pi$ adapts verbatim to this setting, as do the subsequent notions of a \emph{fiberwise-smooth $C^1_{loc}$ $G$-bundle}, and  \emph{vertical tangent bundle} $T^{vt} M$ of a $C^1_{loc}$ $G$-bundle $\pi$. Moreover, the proof of \cite[Lemma~4.29]{AMS} adapts verbatim to this setting, establishing that for $\pi$ which are fiberwise-smooth $C^1_{loc}$ $G$-bundles, there is a $G$-equivariant lift of the vertical tangent microbundle to the vertical tangent bundle. 

\begin{proposition}
\label{prop:smooth-one-k-manifold}
    Let $\pi: M \to B$ be a $G$-equivariant topological submersion of $\langle k\rangle$-manifolds, where $B$ is in fact smooth. Moreover, suppose $\pi$ has a fiberwise smooth $C^1_{loc}$ structure. Then there is a $G$-representation $V$ such that $M \times V$ has the structure of a smooth $\langle k \rangle$-manifold. 
\end{proposition}
\begin{proof}
    The map $\mathcal{D}(\pi): \mathcal{D}(M) \to \mathcal{D}(B)$ is a $G \times (\Z/2)^k$-equivariant topological submersion of topological manifolds, and it is immediate to verify that te $C^1_{loc}$-compatible product neighborhoods for $\pi$ induce corresponding $C^1_{loc}$-compatible product neighborhoods for $\mathcal{D}(\pi)$. By Proposition \ref{prop:submersion-microbundle-splitting} we have an isomorphism of microbundles $T_\mu \mathcal{D}(M) \to T^{vt}_{\mathcal{D}(\pi}) \oplus \mathcal{D}(\pi)^* T_\mu B$. We have lifts of both of these latter microbundles to vector bundles by  \cite[Lemma~4.29]{AMS} and Proposition \ref{prop:smooth-structures-on-mclean-stabilizations}. We conclude by applying Theorem \ref{prop:lashof-smoothing}. 
\end{proof}

By combining Propositions \ref{prop:smooth-one-k-manifold} and \ref{prop:fiberwise-c1loc-on-kuranishi-chart} with Lemma \ref{lemma:unique-smoothing-of-vector-bundle}, we have 
\begin{lemma}
\label{lemma:can-smooth-a-single-chart}
    There exists an elementary stabilization $\mathcal{K}^{\epsilon}_{\Lambda_d, 1}$ of the Kuranishi chart $\mathcal{K}^\epsilon_{\Lambda_d}$ by a faithful complex $PU(d+1)$ representation such that $\mathcal{K}^{\epsilon}_{\Lambda_d, 1}$ admits a smooth structure. 
\end{lemma}
\subsection{Normally complex structures on Kuranishi charts}
\label{sec:normally-complex-structures}
We now explain how to put normally complex structures on the derived $\langle k \rangle$-orbifolds associated to the smoothed Kuranishi charts produced by Lemma \ref{lemma:can-smooth-a-single-chart}, as well as the compatility conditions between them. 
\subsubsection{Normally complex structures on global Kuranishi charts}
Given a vector bundle $V$ over a $G$-space $N$, we say that $V$ is \emph{G}-trivial when for each $x \in N$, $Stab(x) \subset G$ acts trivially on $V_{x}$. Moreover, when $N$ is smooth and $G$, a compact Lie group, acts with finite stabizers, there is a subvector bundle $\mathfrak{g} \times N \subset TN$ with $(\mathfrak{g} \times N)_x$ given by the tangent space to the $G$-orbit at $x$. 

    We note first that if a derived $\langle k \rangle$-orbifold chart arises from a smooth global Kuranishi chart $\mathcal{K} = (G, N, V, \sigma)$,  then to endow $[N/G]$ with a normal complex structure it suffices to write a decomposition into $G$-vector bundles $TN = \oo{TN} \oplus \check{TN}$, where $\oo{TN}_x$ is $G$-trivial while $\check{TN}/(\mathfrak{g} \times N)$ is a complex $G$-vector bundle. Similarly, to endow $[V/G]$  has a normal complex structure it suffices make $V$ into a complex $G$-vector bundle. We will call the data of such a decomposition of $TN$ together with the complex structures on $\check{TN}$ and on $V$ the data of a \emph{normal complex structure on a (smooth) global Kuranishi chart.}

\subsubsection{Normally complex structure on a single Kuranishi chart $T_{\Lambda_d}$.}
\label{sec:normally-complex-structures-single-chart}
    Recall that in the topological Kuranishi chart $\mathcal{K}^\epsilon_{\Lambda_d} = (PU(d+1), T^{\epsilon}_{\Lambda_d}, \mathcal{V}_d \oplus i \mathfrak{pu}_{d+1}, \sigma_{\Lambda_d} \oplus \sigma_H),$  there is a topological submersion $\pi: T^\epsilon_{\Lambda_d} \to \mathcal{F}(\CP^d, d)$. Moreover, the tangent microbundle of $T^\epsilon_{\Lambda_d}$ splits as $T_\mu^{vt}T^\epsilon_{\Lambda_d} \oplus \pi^* T_\mu \mathcal{F}(\CP^d, d)$. We have lifts of each of these summands to (topological) vector bundles; namely, $\pi^*T_\mu \mathcal{F}(\CP^d, d)$ lifts to $\pi^*T\mathcal{F}(\CP^d, d)$, while $T_\mu^{vt}T^\epsilon_{\Lambda_d}$ lifts to the vector bundle $T^{vt}T^{\epsilon}_{\Lambda_d}$ fiber at $((u, \Sigma, F), \eta) \in T^{\epsilon}_{\Lambda_d}$ given by the kernel of the surjective Fredholm operator \eqref{eq:linearized-perturbed-floer-equation}. 
    
    Moreover, after smoothing, the tangent bundle of the smoothed chart is isomorphic to (as a topological vector bundle) of this vector bundle, so it suffices to decompose this vector bundle into a direct sum of a $G$-trivial and a complex $G$-vector bundle. The vector bundle $V_d$ is complex. Moreover, after gluing a pair of $u$-independent Cauchy-Riemann operators (the linearized Cauchy Riemann operators associated to a pair of fixed caps of the Hamiltonian chords at the endpoints) onto the family of operators $((u, \Sigma,F), \eta) \mapsto D_u$, we can deform the entire family of these operators to a family complex linear operators as the complex antilinear-term is order $0$ \cite{FO3:book-vol12}. This implies that $\R^m \oplus T^{vt}T^{\epsilon}_{\Lambda_d}$ is a complex-linear $PU(d+1)$-vector bundle for some large $m$, where $PU(d+1)$ acts trivially on $\R^m$. This in turn implies that there is some $PU(d+1)$-trivial copy of a trivial vector bundle $\R^m \subset T^{vt}T^{\epsilon}_{\Lambda_d}$ such that the complement is a complex $PU(d+1)$-subbundle. Thus, $T^{vt}T^{\epsilon}_{\Lambda_d}$ has the desired decomposition. 
    
    Furthermore $T\mathcal{F}(\CP^d, d)$ splits as $\theta \oplus (\pi_2)^*T \Bl \mathcal{D}_2(\CP^d, d)$ (see Lemma \ref{lemma:equip-mclean-moduli-spaces-with-k-manifold-structure} where $\pi_2$ is denoted by $\pi$), where $\theta$ is a $1$-dimensional vector bundle vector bundle given by the vertical tangent bundle of $\Theta_-$. We claim that $\theta$ is $G$-trivial; indeed, the complex structure at $p_-$ gives a preferred orientation of the fibers of $\Theta_-$, and thus a preferred trivialization of $\eta$, and one sees that $PU(d+1)$ acts trivially on the fibers of $\eta$ with respect to this trivialization. Moreover,
    The bundle $T \Bl \mathcal{D}_2(\CP^d, d)$ is a complex $G$-vector bundle: indeed, $T\mathcal{D}_2(\CP^d, d)$ is a complex $G$-vector bundle since $\mathcal{D}_2(\CP^d, d)$ is a complex manifold, and the complex extends to the inverse image of the blow-up locus by requiring that in a complex chart on $\mathcal{D}_2(\CP^d, d)$ in which the stratification looks like $\C^\ell \times \CP^{d-1}_+$, in the corresponding chart modelled on $\C^{\ell} \times (S^1 \times \R_{>0})^{d-1}$ on $\Bl \mathcal{D}_2(\CP^d, d)$, the coordinate vector along the $i$-th $\R_{>0}$ factor is sent by the complex structure to clockwise rotation along the $i$-th $S^1$-factor. 
    
    Thus, stabilizing the global Kuranishi chart $\mathcal{K}^{\epsilon}_{\Lambda_d}$ by the $PU(d+1)$-representation $\mathfrak{pu}_{d+1}$, we get a global Kuranishi chart $(G, N,  \mathcal{V}_d \oplus i \mathfrak{pu}_{d+1} \oplus \mathfrak{pu}_{d+1}, \sigma_{\Lambda_d}\oplus \sigma_H \times id)$ where $N = T^{\epsilon}_{\Lambda_d}\times \mathfrak{pu}_{d+1}.$ Now $TN/\mathfrak{pu}_{d+1} \times N \simeq TT^\epsilon_{\Lambda_d}$ has the desired decomposition into a $G$-trivial and an equivariantly complex vector bundle. Similarly,  $\mathcal{V}_d \oplus i \mathfrak{pu}_{d+1} \oplus \mathfrak{pu}_{d+1}$ is now a complex $PU(d+1)$-vector bundle via the canonical complex structure on $\mathcal{V}_d$ and the action of $\C$ on $\C \mathfrak{pu}_{d+1}$. Thus the stabilization of any global Kuranishi chart produced by Lemma \ref{lemma:can-smooth-a-single-chart} is \emph{canonically} normally complex oriented.
    
\subsubsection{Normally complex structures and induction/restriction relations.}
\label{sec:normally-complex-structures-compatibility}
We now argue that the inductions, restrictions, and stabilizations of Proposition \ref{prop:induction-relation-between-kuranishi-charts} become compatible with the normal complex structures described above after stabilizing $\mathcal{K}^{\epsilon}_{\Lambda_d}$ by $\mathfrak{pu}_{d+1}$. 

 We now argue that these canonical normal complex structures are compatible in the sense described below. 
    
    First, a product of normally complex oriented global Kuranishi charts is normally complex. 
    
    Second, when we take the restriction of $\mathfrak{T}_{\Lambda_d}^\epsilon$ (or any stabilization of this chart by a complex vector bundle) along the inclusion $U(d) \subset PU(d+1)$, there the corresponding restriction of the stabilization of $\mathfrak{T}_{\Lambda_d}^\epsilon$ by $\mathfrak{pu}_{d+1}$ is the stabilization of the restriction by $\mathfrak{pu}_{d+1}$ by Lemma \ref{lemma:induction-of-vector-bundles}. In particular, smoothed stabilizations of $\mathcal{K}_1$ and $\mathcal{K}_2$ by complex vector bundles )as in \eqref{eq:k1-chart}, \eqref{eq:k2-chart} become normally complex after stabilizing by $\mathfrak{pu}_{d_i+1}$ for $i=1,2$. These latter charts  are normally complex because the complement of $\mathfrak{u}_d$ in $\mathfrak{pu}_{d+1}$ canonically a complex vector space. In particular, the smoothings of the charts $\mathcal{K}_{\Lambda_{d_i}}$ and $\mathcal{K}_i$ after stabilization by complex vector bundles and by $\mathfrak{pu}_{d+1}$  will correspond to the same normally complex derived $\langle k \rangle$-orbifold. 
    
    Similarly, when we induce a smoothed stabilization of $\mathcal{K}^\epsilon_3$ as in \eqref{eq:k3-chart} by a complex bundle along the inclusion $U(d_1) \times U(d_2) \subset PU(d_1 + d_2+1)$ to produce a corresponding smoothed stabilization of $\mathcal{K}^\epsilon_{\Lambda_d}$, stabilizing both of these charts by $\mathfrak{pu}_{d+1}$ will preserve the induction relation, will make them both normally complex and will make the corresponding normally complex derived $\langle k \rangle$-orbifolds equal. 
    
    Next, we study the map \eqref{eq:map-on-base-spaces-2} and its behavior upon stabilizing by $\mathfrak{pu}_{d_1+1} \oplus \mathfrak{pu}_{d_2+1}$. First, we note that $V^{tr}_{d, d_1}/U(d_1) \times U(d_2)$ is the moduli space of $d_2$-dimensional subspaces containing a fixed line and transverse to another $d_1$-dimensional subspace containing that line, and this has a complex structure. In fact, this moduli space is biholomorphic to $i \mathfrak{pu}_{d_1, d_2}$, where we view this latter space \emph{canonically} as a complex vector space via the complex-structure on its upper-right entries; indeed, the fact that the exponential of a Hermitian matrix is Hermitian shows that $M \mapsto g_{id}(M)$ induces such a biholomorphism.  
    
    Moreover, write $(i \mathfrak{pu}_{d_1, d_2})^{\perp} \subset i \mathfrak{pu}_{d+1}$ for the complement; then the map $i \mathfrak{pu}_{d_1+1} \oplus i \mathfrak{pu}_{d_2+1} \mapsto (i \mathfrak{pu}_{d_1, d_2})^{\perp}$, $(H_1, H_2) \mapsto H(H_1, H_2)$ \eqref{convenient-H-map} is a linear isomorphism. Write $-i(i \mathfrak{pu}_{d_1, d_2})^{\perp}$ for the image of $(i \mathfrak{pu}_{d_1, d_2})^{\perp}$ under multiplicaton by $-i$; then we also have an isomorphism 
    $\mathfrak{pu}_{d_1+1} \oplus \mathfrak{pu}_{d_2+1} \to -i(i \mathfrak{pu}_{d_1, d_2})^{\perp}$, $(H_1', H_2') \mapsto -iH(iH'_1, iH'_2)$, and the sum of this map with $H$ is complex linear. 
    
    Thus, extending \eqref{eq:map-on-base-spaces-2} to 
\begin{equation}
\label{eq:map-on-base-spaces-2-b}
\begin{gathered}
   \mathfrak{pu}_{d_1+1} \times \mathfrak{pu}_{d_2+1} \times i \mathfrak{pu}^{d_1,d_2}_{d+1} \times T^\epsilon_1 \times T^\epsilon_2 \to -i(i \mathfrak{pu}_{d_1, d_2})^{\perp} \times  T^{\epsilon, \epsilon}_3  \\
   (H_1', H_2', M, ((u_1, \Sigma_1, F_1), \eta_1), ((u_2, \Sigma_2, F_2), \eta_2)) \mapsto (-iH(iH'_1, iH'_2), (u_1 \# u_2, \Sigma_1 \# \Sigma_2, F_1 \#_{g_{H}(M)} F_2), (\eta_1 + g_H(M) \eta_2))
\end{gathered}
\end{equation}
induces, after stabilizing, a map preserving the decompositions into $U(d_1) \times U(d_2)$-trivial bundles and the complex structures on tangent bundles modulo $\mathfrak{u}_{d_1} \times \mathfrak{u}_{d_2}$. The above map is naturally covered by the stabilization 

\begin{equation}
    \label{eq:map-on-total-spaces-b}
    \begin{gathered}
      \mathfrak{pu}_{d_1+1} \times \mathfrak{pu}_{d_2+1} \times i \mathfrak{pu}^{d_1,d_2}_{d+1} \times i \mathfrak{pu}^{d_1, d_2}_{d+1} \times \mathcal{V}_a \times i \mathfrak{pu}_{d_1+1} \times \mathcal{V}_b \times i \mathfrak{pu}_{d_2+1} \to -i(i \mathfrak{pu}_{d_1, d_2})^{\perp} \times \mathcal{V}_{ab} \times i \mathfrak{pu}_{d+1}\\
     (H_1', H_2', M, N, ((u_1, \Sigma_1, F_1), \eta_1), \eta'_1, H_1, ((u_2, \Sigma_2, F_2), \eta_2), \eta'_2, H_2) \mapsto  \\(-iH(H'_1, H'_2), (u_1 \# u_2, \Sigma_1 \# \Sigma_2, F_1 \# g_H(M) F_2), \eta_1 + g_H(M) \eta_2), \eta_1' + g_H(M)\eta_2', H(H_1, H_2)+N).
     \end{gathered}
\end{equation}

This map is fiberwise $\C$-linear as a map of vector bundles. The coefficients of the map may not depend smoothly on the base parameter since $H$ does not depend smoothly on the underlying pre-Floer trajectory. Nonetheless, as we we will be inductively smoothing later, we will use this map to \emph{induce} a smooth structure on the target, and then textend, so that the final map is smooth, 

The remaining salient relation between these various charts is that $\mathcal{K}^{\epsilon, \epsilon}_3$ can then be stabilized by further by $\mathfrak{pu}^{d_1, d_2}_{d+1}$, which is naturally a complex vector bundle, to get the stabilization of $\mathcal{K}^{\epsilon, \epsilon}_3$ by  $\mathfrak{pu}_{d+1}$. We can then include it into the corresponding stabilization of $\mathcal{K}^{\epsilon}_3$ and then induce along the map $U(d_1) \times U(d_2) \subset PU(d+1)$ to produce the desired stabilization of $\mathcal{K}_{\Lambda_d}^{\epsilon}$.

\subsection{Compatibly smoothing Kuranishi charts}
We now describe how to make make the smoothings of the distinct Kuranishi charts compatible with one another. First we introduce a construction which will allow us to utilize a slightly upgraded version of  stable $G$-smoothing (Theorem \ref{prop:lashof-smoothing}): in essence, proving a multi-parameter version of uniqueness of smoothings up to concordance and stabilization.

\begin{definition}
Given a Kuranishi chart $\mathcal{K} = (G, T, E, s)$, we define the collaring of the Kuranishi chart $Coll(\mathcal{K})$ to be $(G, Coll(T), Coll(E), Coll(s))$. 
\end{definition}

Let $M$ be a topological $G-\langle k \rangle$-manifold and with a lift of the tangent microbundle. Consider $M \times [0,1]^{\ell}$, which is a topological $G-\langle k + \ell \rangle$-manifold. 
\begin{definition}
  We say that a \emph{parameterized smooth structure} on $M \times [0,1]^\ell$ is a $G \times (\Z/2)^k$-equivariant smooth structure on the $\langle \ell\rangle$-manifold $\mathcal{D}(M) \times [0,1]^\ell$ for which the projection to $[0,1]^\ell$  is a submersion of $\langle \ell\rangle$-manifolds. 
\end{definition}
 Lemma \ref{lemma:smoothing-double-smooths-original} and the functoriality of doubling immediately imply
 \begin{lemma}
 A parametrized smooth structure on $M \times [0,1]^\ell$ gives the topological $G-\langle k + \ell \rangle$-manifold a canonical smooth structure characterized by the property that the pullbacks of smooth functions on $\mathcal{D}(M) \times [0,1]^\ell$ to $M \times [0,1]^\ell$ are smooth.  With respect to this smooth structure, the projection to $[0,1]^\ell$ is a submersion from a $\langle k +\ell\rangle$-manifold to a $\langle k \rangle$-manifold with $S_f = [n]$ and $\tilde{f}(r) = r-n$ (see Definition \ref{def:general-submersion-of-k-manifolds}). $\qed$
 \end{lemma}

\begin{definition}\cite{AlbinMelrose}
    A \emph{p-submanifold} of a $\langle k \rangle$-manifold $M$ of dimension $m$ is a subset $N \subset M$ such that for any $x \in N$ there is a $\langle k \rangle$-manifold chart $\psi: \R^{[k] \setminus S(x)_+} \times \R^{m +k-|S(x)|} \supset U \to V \subset M$ such that $\psi^{-1}(V \cap N) = U \cap \R^{[k] \setminus S(x)_+}(T_\psi) \times \R^{m +k-|S(x)|}$ for some subset $T_\psi \subset [k - |S(x)|]$.
\end{definition}

If a $p$-submanifold of a $\langle k \rangle$-manifold is naturally a topological $\langle k \rangle$-manifold with the induced stratification, then it inherits a unique smooth structure (as the stratification in fact determines the $T_\psi$ in the definition above). 

\begin{proposition}
\label{prop:upgraded-smoothing}
    Fix natural numbers $k, \ell$. Write $\vec{0} = (0,\ldots,0) \in \R^\ell$ and $\vec{1} = (1, \ldots, 1) \in \R^\ell$. 
    
    Suppose that we have specified patermetrized smooth strutures on $M \times C_i$ for some collection $C_1, \ldots, C_r \subset [0,1]^\ell$ of faces of the cube,  as well as parameterized smooth structures on $M(S) \times [0,1]^{\ell}$ for all non-maximal $S \in 2^{[k]}$, all of which which restrict to one another.  Then there is a smooth structure on $M \times [0,1]^k \times R(\sigma)$, for some $G$-representation $R(\sigma)$, such that $M \times \vec{0} \times R(\sigma)$, $M \times \vec{1} \times R(\sigma)$, and $M(S) \times [0,1]^k \times R(\sigma)$ are $p$-submanifolds and the smooth structures on them are induced from the given smooth structures.
\end{proposition}
\begin{proof}
    
    We follow \cite[pp.~302]{Lashof1979}. We can find a $G \times (\Z/2)^k$-representation $R(\sigma)$ such that we have smooth equivariant embeddings of $\mathcal{D}(M(S)) \to R(\sigma)$ which restrict to one another. After stabilizing $\sigma$ further we can find smooth equivariant embeddings of $\mathcal{D}(M(S)) \times [0,1]^\ell \to R(\sigma) \times [0,1]^\ell$ which commute with the projections to $[0,1]^\ell$ and induce the fiberwise smooth structures on each strata and restrict to one another. 
    
    After stabilizing $\sigma$  further we can inductively (along the face lattice of the cube) extend the embeddings $\partial M \times C_j$ into  $R(\sigma) \times C_j$ to smooth equivariant embeddings of $M \times C_j$ for specified faces $C_j$ of $[0,1]^\ell$.
    
    After stabilizing $\sigma$ further we can find a continuous equivariant embedding of $\mathcal{D}(M) \times [0,1]^\ell$ to $R(\sigma)\times [0,1]^\ell$ commuting with the projections, restricting to stabilizations of all previously chosen embeddings on the subsets where they have been defined. Moreover by additional stabilization we can arrange  such that there is a neighborhood $U$ of the image of $\mathcal{D}(M) \times [0,1]^\ell$ equipped with a projection map $U \to \mathcal{D}(M) \times [0,1]^\ell$ commuting with the projections to $[0,1]^\ell$ such that the embedding into $U$ and the projection defines a $G \times (\Z/2)^k$-microbundles over $\mathcal{D}(M) \times [0,1]^\ell$ which are smooth vector bundles over the part of the domain with a smooth structure.
    
    We can find a family of lifts of the tangent microbundle $\lambda_t: \xi \to \tau \mathcal{D}(M)$, $t \in [0,1]^k$, such that $\lambda_{\vec 0}$ and $\lambda_{\vec{1}}$ agree with lifts of the tangent microbundle associated to the smooth structures specified on $\mathcal{D}(M) \times \vec{0}$ and $\mathcal{D}(M) \times \vec{1}$, and such that there are embeddings $\iota^S_t: T\mathcal{D}(M)(S) \to \xi$ such that $\lambda_t \iota^S_t \subset \tau\mathcal{D}(M)(S) \subset \tau\mathcal{D}$ and such that are lifts of of the tangent microbundles of $\mathcal{D}(M)(S)\times t$ associated to the given smooth structures. 
    
    Clearly $U$ has a smooth structure. Thus, so does $\pi^* \xi$ by by Lemma \ref{lemma:unique-smoothing-of-vector-bundle} since the latter is a vector bundle over a $\langle k \rangle$-manifold. We will show that an open neighborhood in $\pi^* \xi$ is equivariantly homeomorphic $\mathcal{D}(M) \times [0,1]^\ell \times R(\sigma)$. In fact, one simply performs the construction in \cite[pp.~302]{Lashof1979} which we explicate for convenience. Given a point $x \in \mathcal{D}(M)$ , an element $t \in [0,1]^k$, and a point $v \in R(\sigma)$, we can think of this as a point of $TU|_{\mathcal{D}(M) \times t}$ and apply the exponential map to get an element of $U$ if $v$ was sufficiently small. This gives us another element $y \times t \in \mathcal{D}(M) \times[0,1]^k$ via the projection $U \to \mathcal{D}(M) \times [0,1]^k$, and an element of the fiber of $U$ over $y \times t$. If $v$ was sufficiently small then $(y,x) \in \tau \mathcal{D}(M)$ will lie in the image of $\lambda_t$, and so applying $\lambda_t$ gives us an equivariant homeomorphism between an open subset of the zero section of $\mathcal{D}(M) \times [0,1]^k \times R(\sigma)$ and $\pi^* \xi$. Thus $\mathcal{D}(M) \times [0,1]^k \times R(\sigma)$ has a smooth structure. The construction shows that the map to $[0,1]^k$ is a submersion. The compatibility conditions follow from the fact that $U$ is a vector bundle in the regions where smooth structures have been specified, and from the compatibility conditions imposed on $\lambda_t$. 
    
\end{proof}

\begin{lemma}
\label{lemma:compatible-system-of-derved-orbifold-charts-1}
    There exist stabilizations $\mathcal{K}'_{\Lambda_d}$ of $Coll(\mathcal{K}_{\Lambda_d})$ such that $\mathcal{K}'_{\Lambda_d}$ admit normal complex structures and such that for each $(d_j)_{j=1}^r, \sum_j d_j = d$, writing $\mathcal{K}'_{\Lambda_d}//G$ for the associated normally complex derived $\langle k \rangle$-orbifold, $\mathcal{K}'_{\Lambda_d}//G)(S(d_1, \ldots, d_r))$ is, as a normal complex derived $\langle k \rangle$-orbifold, normally complex diffeomorphic to a stabilization of $\mathcal{K}'_{\Lambda_{d_1}}//G \times \ldots \mathcal{K}'_{\Lambda_{d_r}}//G$ by a complex vector bundle. 
\end{lemma}
\begin{proof}

    We first replace all $\mathcal{K}_{\Lambda_d}$ by the stabilizations by complex vector bundles $\mathcal{K}^{\epsilon}_{\Lambda_d, 1}$ arising from Lemma \ref{lemma:can-smooth-a-single-chart}, and then further stabilize them to $\mathcal{K}^{\epsilon}_{\Lambda_d, 2}$ by $\mathfrak{pu}_d$ such that they come equipped with canonical normally complex structures as in Section \ref{sec:normally-complex-structures-single-chart}. The collars of these new charts will admit normally complex structures after they are smoothed, as normally-complex structures are a condition on the tangent bundle and obstruction bundle of a global Kuranishi chart and collaring manifestly preserves this condition.

    We now apply Propostion \ref{prop:upgraded-smoothing} inductively $d$ and on the posets $[d-1]$ for each $d$ to extend our smoothings of $\mathcal{K}^{\epsilon}_{\Lambda_d, 2}$ to $Coll(\mathcal{K}^{\epsilon}_{\Lambda_d, 2})$ and achieve the compatibility condition of the Lemma.  The case $d=1$ is already complete, as $Coll(\mathcal{K}^{\epsilon}_{\Lambda_1, 2}) = \mathcal{K}^{\epsilon}_{\Lambda_1, 2} =: \mathcal{K}'_{\Lambda_1}$.  Assume that for all $1 \leq \ell < d$, we have smooth Kuranishi charts $\mathcal{K}'_{\Lambda_\ell}$ satisfying the condition. Now Proposition \ref{prop:induction-relation-between-kuranishi-charts} and the discussion in Section \ref{sec:normally-complex-structures-compatibility} shows how to realize a stabilization of $Coll(\mathcal{K}^{\epsilon}_{\Lambda_d}, 2)(S_(d_1, \ldots, d_r))$ from  $Coll(\mathcal{K}^\epsilon_{\Lambda_{d_1}, 2})$, \ldots, $Coll(\mathcal{K}^\epsilon_{\Lambda_{d_r}, 2})$ by taking products, stabilizations by complex vector spaces, taking open subsets, and equivalences. These previous charts have been smoothed earlier so these charts come equipped with smooth structures; we have smoothed the chart everywhere except the interior of the collar. By applying Proposition \ref{prop:complement-bundle} inductively for each $S \in 2^{[d-1]} \setminus [d-1]$,  there is a stabilization $\mathcal{K}'_{\Lambda_d}$ of $Coll(\mathcal{K}^{\epsilon}_{\Lambda_d})$ by a faithful complex $PU(d+1)$-representation such that the smooth structures specified everywhere except in the interior of the collar are extended to a smooth structure on the whole chart. In particular, we still have  $\mathcal{K}'_{\Lambda_d}(S(d_1, \ldots, d_r))$ arising from $\mathcal{K}'_{\Lambda_{d_j}}$ (which are themselves stabilizations of collars of $\mathcal{K}^\epsilon_{\Lambda_d, 2}$) by products, inductions, restrictions, stabiliations, and difeomorphisms. We conclude by passing to associated derived $\langle k \rangle$-orbifold charts.
   
\end{proof}

\section{Floer complex from Kuranishi charts via strongly transverse perturbations}
We now develop some axiomatics for compatibly perturbing the defining sections of a compatible system of normally complex derived $\langle k \rangle$-orbifold charts.

\paragraph{Categories of spaces and orbifolds}
Let $Fin$ denote the monoidal category with objects given by finite sets, morphisms given by injective maps, and monoidal product given by disjoint union of finite sets. We enlarge $Fin$ to a category $Fin'$ by adding on the object $\bullet$, which can be thought of informally as the ``set with $-1$ element''. This object is defined by the properties that $Fin'(\bullet, S) = Fin'(S, \bullet) = \emptyset$, and $\bullet \tensor S = S \tensor \bullet = \bullet$. 

Let $Top^{Fin'}$ be the monoidal category with objects given by pairs $(S, \M)$ with $S \in Fin'$ and $\M$ a $2^{S}$-space, with $2^{\bullet}$ spaces required to be empty, such that that the maximal open stratum of $\M$ is nonempty if $S \neq \bullet$; with morphisms $(S, \M_S) \to (T, \M_T)$ given by a map $f: S \to T$ in $FinSet'$, and a map $\bar{f}: \M_S \to \M_T$ which is a homeomorphism onto its image $\M_T(f(S))$, which induces homeomorphisms $\M_S(S') \to \M_T(f(S'))$ for $S' \subset S$. The monoidal structure structure is given by $(S, \M_S) \times (T, \M_T) = (S \cup T, \M_S \times \M_T)$. Note that $(\emptyset, \{*\})$ is the unit object for this monoidal structure. 

We declare the unique derived $\bullet$-orbifold to be the derived $\langle 0 \rangle$-orbifold with the empty thickening and empty orbibundle, which we call the empty orbifold. Let $DerOrb'$ be the monoidal category consisting 
of

\begin{itemize}
    \item Objects given by pairs $(S, \widetilde{\mathcal{K}})$, with $S \in Fin'$ and $\widetilde{\mathcal{K}}$ an $S$-orbifold, such that that the maximal open stratum of $\M$ is nonempty if $S \neq \bullet$, and 
    \item morphisms $(S, \widetilde{\mathcal{K}}_S) \to (T, \widetilde{\mathcal{K}})$ are given by maps $f: S \to T$ together with a  diffeomorphism of $\widetilde{\mathcal{K}}_T(f(S))$ with a stabilization of $\widetilde{\mathcal{K}}_S$ by a vector bundle.
\end{itemize}  The monoidal structure is given by the monoidal structure on $FinSet'$ and the product of derived $\langle k \rangle$-orbifolds. 

Similarly, we have the monoidal category $DerOrb'_\C$ with objects $(S, \widetilde{\mathcal{K}})$ as in $DerOrb'$ but with $\widetilde{\mathcal{K}}$ equipped with a normal complex structure, and for which morphisms $(S, \widetilde{\mathcal{K}}_S) \to (T, \widetilde{\mathcal{K}})$ are as in $DerOrb'_\C$ but with $\widetilde{\mathcal{K}}_T(f(S))$ now identified via normally complex diffeomorpism from a stabilization of $\widetilde{\mathcal{K}}$ by a \emph{complex} vector bundle. We have forgetful monoidal functors $DerOrb'_\C \to DerOrb' \to Top^{Fin'}$. As usual, we will at times implicitly identify  elements of $DerOrb'$, respectively $DerOrb'_\C$, with those corresponding to replacing the underlying $\langle k \rangle$-orbifold by an open subset $U$ containing the zero set of the Kuranishi section, and replacing the vector bundle and the section by their restrictions to $U$. 

\paragraph{Topological flow categories}
\newcommand{\CC}{\mathcal{C}}

Given a category $\CC$ enriched in one of $Top(Fin'), DerOrb'$, or $DerOrb'_\C$, we will use slightly nonstandard notation and say that the morphism object representing maps from $\tilde{x}$ to $\tilde{y}$ is the tuple $(\mathcal{S}_{\CC(\tilde{x}, \tilde{y}}), \CC(\tilde{x}, \tilde{y}))$, and we will call the elements of $\CC(\tilde{x}, \tilde{y})$ the \emph{morphisms} from $\tilde{x}$ to $\tilde{y}$. Moreover, given a sequence $\tilde{x}_0, \ldots, \tilde{x}_r \in \CC$ with $\CC(\tilde{x}_j, \tilde{x}_{j+1})$ nonempty, we will write $\mathcal{S}_{\CC}(\tilde{x}_0, \ldots, \tilde{x}_r) \subset \mathcal{S}_{\CC(\tilde{x}_0, \tilde{x}_r)}$ for the image of $\mathcal{S}_{\CC(\tilde{x}_0, \tilde{x}_1)} \sqcup \ldots \sqcup \mathcal{S}_{\CC(\tilde{x}_{r-1}, \tilde{x}_r)}$ induced by the morphism in $Fin$ underlying the composition operation.

\begin{definition}
\label{def:topological-flow-category}
A topological flow category $\CC$ is a category enriched in $Top(Fin')$ such that the morphisms from $\tilde{x}$ to $\tilde{x}$ are the unit object in $Top(Fin')$.  We define $|\mathcal{S}_{\CC(\tilde{x}, \tilde{y})}| =: r(\tilde{x}, \tilde{y})$ if $\CC(\tilde{x}, \tilde{y}) \neq \emptyset$ and $\tilde{x} \neq \tilde{y}$.

A $\Pi$-equivariant topological flow category $\mathcal{F}$ is a topological flow category with a free $\Pi$-action on objects and morphisms. We say that $\mathcal{F}$ is \emph{proper} if $\CC(\tilde{x}, \tilde{y})$ is compact for any pair $(\tilde{x}, \tilde{y}) \in \CC^2$, and if the set
\begin{equation}
(\CC^2/\Pi)_q := \{ (\tilde{x}, \tilde{y}) \in \CC^2 | \tilde{x} \neq \tilde{y},  \CC(\tilde{x}, \tilde{y}) \neq 0, r(\tilde{x}, \tilde{y}) = q\} / \Pi
\end{equation}
is finite for each natural number $q$. We set $(\CC^2/\Pi) = \cup_{q=0}^\infty (\CC^2/\Pi)_q$ and $(\CC^2/\Pi)_{< r} = \cup_{q=0}^{r-1} (\CC^2/\Pi)_q$.

A grading on a $\Pi$-equivariant topological flow category $\CC$ is a map $\mu: Ob(\CC) \to \Z$ such that 
$\mu(\tilde{x}) - \mu(\tilde{y}) - 1 = \mu(g\tilde{x}) - \mu(g\tilde{y}) - 1$ for each $g \in \Pi$.
\end{definition}

\begin{remark}
To extend the method of this paper to the case of symplectic cohomology with quadratic Hamiltonians, it is necessary to weaken the notion of a proper flow category and to modify the various inductive constructions of this paper appropriately. 
\end{remark}

\paragraph{Flow categories equipped with charts}
\begin{definition}A \emph{compatible system of derived $\langle k \rangle$-orbifold charts} for a graded proper $\Pi$-equivariant topological flow category $\CC$ is a category $\CC'$ enriched in $DerOrb'_\C$ equipped with a free action of $\Pi$ with $Ob(\CC') = Ob(\CC)$ and such that applying the functor $DerOrb'_\C \to Top^{Fin'}$ recovers $\CC$; and such that $vdim(\mathcal{F'}(\tilde{x}, \tilde{y}) = \mu(\tilde{x} - \mu(\tilde{y})-1$ for each $\tilde{x}, \tilde{y} \in Ob(\CC)$. 
\end{definition}

Given such a $\CC'$, for any tuple $\tilde{x}_0, \ldots, \tilde{x}_r$ of objects of $\mathcal{F'}$ with $\CC(x_j, x_{j+1}) \neq \emptyset$ for $j=0, \ldots, r-1$, we have a map
\begin{equation}
    Ind: (\sigma'_0, \ldots, \sigma'_{r-1}) \mapsto \sigma'
\end{equation}from tuples of $(\sigma'_0, \ldots, \sigma'_{r-1})$, where $s'_j$ is a strongly transverse section of $\mathcal{V}_j$ for $\CC'(\tilde{x}_j, \tilde{x}_{j+1}) = (T_j, \mathcal{V}_j, \sigma_j)$, to strongly transverse sections $\sigma'$ of $\mathcal{V}$ (where $\CC'(\tilde{x}_0, \tilde{x}_r) = (\mathcal{T}, \mathcal{V}, \sigma)$) given by taking products and taking direct sum with the identity section of the stabilizing complex vector bundles, and applying the diffeomorphism induced by the map in $DerOrb'_\C$. By construction this map is \emph{associative} in the sense that 
\begin{equation}
\label{eq:induction-is-associative}
    Ind(Ind(\sigma'_0, \ldots, \sigma'_{r_1}), Ind(\sigma'_{r_1+1}, \ldots ), \ldots, Ind(\ldots, \sigma'_{r-1})) = Ind(\sigma'_0, \ldots, \sigma'_{r-1})
\end{equation}
for all tuples $\tilde{x}_0, \ldots, \tilde{x}_r$ as above. 

\begin{definition}
A \emph{compatible system of perturbations} for $\CC'$ is a choice, for every $\tilde{x}, \tilde{y} \in \CC$ with $\tilde{x} \neq \tilde{y}$ and $\CC(\tilde{x}, \tilde{y}) \neq \emptyset$, of choice of strongly transverse section $\sigma'(\tilde{x}, \tilde{y})$ of  $\mathcal{F'}(\tilde{x}, \tilde{y})$  such that 
\begin{equation}
\label{eq:compatible-system-of-perturbations-condition}
    Ind(\sigma'(\tilde{x}_0, \tilde{x}_1), \ldots, \sigma'(\tilde{x}_{r-1}, \tilde{x}_r)) = \sigma'(\tilde{x_0}, \tilde{x}_r)|_{T(\tilde{x}_0, \tilde{x}_r)(\mathcal{S}_{\CC}(\tilde{x}_0, \ldots, \tilde{x}_r))}
\end{equation}
for any tuple $\tilde{x}_0, \ldots, \tilde{x}_r$ of objects of $\mathcal{F'}$ with $\CC(\tilde{x}_j, \tilde{x}_{j+1}) \neq \emptyset$ for $j=0, \ldots, r-1$. 

We say that a compatible system of perturbations is $\Pi$-equivariant when $g \cdot \sigma(\tilde{x}, \tilde{y}) = \sigma(g\tilde{x}, g\tilde{y})$ for all $g \in \Pi$ and all pairs $\tilde{x}, \tilde{y} \in \CC$. 
\end{definition}

A sub-flow-category of a $\Pi$-equivariant topological flow category $\CC$ is an enriched subcategory $\CC_2 \subset \CC$ closed under the $\Pi$-action, where $\CC_2(\tilde{x}, \tilde{y}) \subset \CC(\tilde{x}, \tilde{y})$ is an inclusion of connected components for each $(\tilde{x}, \tilde{y}) \in Ob(\CC_2)$. By restricting to open subsets of the derived $\langle k \rangle$-orbifold charts, if $\CC'$ is a compatible system of derived $\langle k \rangle$-orbifold charts for $\CC$ then there is a corresponding compatible system of derived $\langle k \rangle$-orbifold charts $\CC'_2$ for $\CC_2$, which (after passing to open susbsets of morphism objects of $\CC'$) is a subcategory of $\CC'$ with morphisms given by connected components of morphisms in $\CC'$.

\paragraph{Existence of compatible systems of perturbations}
\begin{proposition}
\label{prop:transverse-perturbations-exist-inductively}
    There exists a $\Pi$-equivariant compatible system of perturbations for any proper $\Pi$-equivariant topological floer category $\CC$ equipped with a compatible system of derived $\langle k \rangle$-orbifold charts $\CC'$. Moreover, if $\CC_2$ is a sub-flow-category of $\CC$ such that the Kuranishi sections of the associated system of derived $\langle k \rangle$-orbifold charts $\CC'_2$ has $\sigma_2(\tilde{x}, \tilde{y})$ already strongly transverse (where we write $\sigma_2(\tilde{x}, \tilde{y})$ for the Kuranishi section associated to $\CC'_2(\tilde{x}, \tilde{y})$ for $(\tilde{x}, \tilde{y}) \in \CC'$) then we can choose $\sigma'(\tilde{x}, \tilde{y})|_{\CC'_2(\tilde{x}, \tilde{y})} = \sigma_2(\tilde{x}, \tilde{y})$.
\end{proposition}
\begin{proof}
    We induct on $r$. Suppose that we have chosen $\sigma'(\tilde{x}, \tilde{y})$ strongly transverse for $([\tilde{x}], [\tilde{y}]) \in (\CC^2/\Pi)_{<r}$, such that
    \begin{itemize}
        \item Condition (\ref{eq:compatible-system-of-perturbations-condition}) holds for any tuple $\tilde{x}_0, \ldots, \tilde{x}_r$ as in (\ref{eq:compatible-system-of-perturbations-condition}) with  
    $([\tilde{x}_0], [\tilde{x}_{r}]) \in (\CC^2/\Pi)_{<r}$, and
    \item we have $g \cdot \sigma(\tilde{x}, \tilde{y}) = \sigma(g\tilde{x}, g\tilde{y})$ for all $g \in \Pi$ and all pairs $([\tilde{x}, [\tilde{y}]) \in  (\CC^2/\Pi)_{<r}$, and 
    \item we have $\sigma'(\tilde{x}, \tilde{y})|_{\CC'_2(\tilde{x}, \tilde{y})} = \sigma_2(\tilde{x}, \tilde{y})$ for $([\tilde{x}], [\tilde{y}]) \in (\CC^2/\Pi)_{<r}$.
    \end{itemize}
    This hypothesis is vacuously satisfied for $r=0$. We wish to apply Lemma \ref{lemma:fundamental-extension-lemma} to each element of $(\CC^2/\Pi)_r$. Write $\CC(\tilde{x}, \tilde{y}) = (T(\tilde{x}, \tilde{y}), V(\tilde{x}, \tilde{y}, \sigma(\tilde{x}, \tilde{y})$, and $\CC_2(\tilde{x}, \tilde{y}) = (T_2(\tilde{x}, \tilde{y}), V_2(\tilde{x}, \tilde{y}, \sigma_2(\tilde{x}, \tilde{y})$ Write $T(\tilde{x}, \tilde{y}) = \sqcup_{a \in \pi_0(T(\tilde{x}, \tilde{y}))} T(\tilde{x}, \tilde{y})_a$ for the connected components, and write $\sigma(\tilde{x}, \tilde{y})_a = \sigma(\tilde{x}, \tilde{y})|_{T(\tilde{x}, \tilde{y})_a}.$ If $a \in \pi_0(T_2(\tilde{x}, \tilde{y}))$ set $\sigma'(\tilde{x}, \tilde{y})|_{a} = \sigma_2(\tilde{x}, \tilde{y})|_{T(\tilde{x}, \tilde{y})_a}$. Otherwise for $a \neq \pi_0(T_2(\tilde{x}, \tilde{y}))$,  write 
    \begin{equation}
        \partial T(\tilde{x}, \tilde{y})_a = \bigcup_{\substack{\tilde{z} \in \CC \\ \CC(\tilde{x}, \tilde{y}) \neq \emptyset, \CC(\tilde{y}, \tilde{z}) \neq \emptyset}} T(\tilde{x}, \tilde{y}, \tilde{z})_a
    \end{equation}
    where $T(\tilde{x}, \tilde{y}, \tilde{z})_a \simeq V(\tilde{x}, \tilde{y}, \tilde{z})_a \to (T(\tilde{x}, \tilde{y}) \times T(\tilde{y}, \tilde{z}))_a$ is restriction of the stabilizing vector bundle to the corresponding family of connected components. (We are identifying $V(\tilde{x}, \tilde{y}, \tilde{z})$ with $T(\tilde{x}, \tilde{y}, \tilde{z}) \subset T(\tilde{x}, \tilde{y})$ via the normally complex diffeomorphism.) We can then write 
    \begin{equation}
        V(\tilde{x}, \tilde{y})|_{T(\tilde{x}, \tilde{y}, \tilde{z})} = 
        V(\tilde{x}, \tilde{y}, \tilde{z}) \oplus V(\tilde{x}, \tilde{y}) \oplus V(\tilde{y}, \tilde{z}) 
    \end{equation}
    where the direct sum is really the direct sum of the vector bundles pulled back from $V_{\tilde{x}, \tilde{y}, \tilde{z}}$, $T(\tilde{x}, \tilde{y}$, and $T(\tilde{y}, \tilde{z})$ respectively. We set $\sigma'(\tilde{x}, \tilde{y})|_{V(\tilde{x}, \tilde{y}, \tilde{z})_a} = Ind(\sigma'(\tilde{x}, \tilde{y}), \sigma'(\tilde{y}, \tilde{z}))|_{T(\tilde{x}, \tilde{y}, \tilde{z})_a}$. 
    By (\ref{eq:induction-is-associative}) and the inductive hypotesis this is a well-defined section. We extend this section to the interior as a contnuous section $C^0$ close to $\sigma(\tilde{x}, \tilde{y})_a$, and we apply Lemma \ref{lemma:fundamental-extension-lemma} to produce $\sigma'(\tilde{x}, \tilde{y})_a$. This defines $\sigma'(\tilde{x}, \tilde{y})$ by defining it on each connected component. We have now satisfied the inductive hypothesis for $r+1$. 
\end{proof}

Given a compatible system of perturbations $\sigma'$, Lemma \ref{lemma:dimension-count} implies that there is a finite number of zeros of $\sigma'(\tilde{x}_-, \tilde{x}_+)$ and they all lie away from the isotropy locus. In this case we define $\#\CC(\tilde{x}_-, \tilde{x}_+; \sigma')$ to be the sum of the number of zeros of $\sigma'(\tilde{x}_-, \tilde{x}_+)$ counted with multiplicity if $\vdim \CC'(\tilde{x}, \tilde{y}) = 0$, and otherwise we set $\#\CC(\tilde{x}_-, \tilde{x}_+; \sigma')$.

\paragraph{Chain complex associated to a compatible system of perturbations}
Given a compatible system of perturbations $\sigma'$, one can define a chain complex $|\mathcal{C}(\sigma')|$ when a few additional decorations are specified. 

\begin{definition}
	A \emph{Novikov group} is a triple $(\Pi, E, | \cdot |)$ where $\Pi$ is a finitely generated free abelian group, $E: \Pi \to \R$ is a homomorphism called the \emph{energy function}, and $\mu: : \Pi \to \Z$ is a homomorphism called the \emph{grading}. We will use $\Pi$ to refer to $(\Pi, E, \mu)$ when the additional data are clear from the context. We will write $N_\Pi$ for the positive generator of the image of the grading, unless $\mu(\Pi) = 0$ in which case we will say $N_\Pi = 0$.
	We have the group $\Pi_e = \Pi/\ker E$. 
	
	The group rings $\Z[\Pi]$, $\Z[\Pi_E]$, 
	inherit filtrations from $E$, and we write $\Lambda^{\Pi}_\Z$,  $\Lambda^{\Pi_E}_\Z$,
	for the respective completions of these rings with respect to this filtration. 
	

\end{definition}

\begin{lemma}
\label{lemma:qis-lemma}
    Let $C_1^\bullet$ be a cochain complex of free modules over $\Z$. Let $C_2^\bullet$ be the $\Z/2N_\Pi \Z$-graded version of $C_1^\bullet  \tensor_R \Lambda^{\Pi_E}_\Z$, that is 
    \begin{equation}
        C_2^{k} = \bigoplus_{r = k \mod 2N_\pi} C_1^\bullet \tensor \Lambda^\Pi_\Z, k \in \Z/2N_\Pi \Z.
    \end{equation}
    Then any $\Z/N_\Pi \Z$-graded cochain complex $D^\bullet$ of free $\Lambda^{\Pi_E}_\Z$-modules quasi-isomorphic to $C_2^\bullet$ must have at least the same number of generators as the minimum number of generators of a chain complex of free $\Z$-modules quasi-isomorphic to $C_1^\bullet$. 
\end{lemma}
\begin{proof}
    Note that  $\Lambda^{\Pi_E}_\Z$ is a PID with principal ideals given by $(n)$ for $n \in \Z$ \cite[Chapter 4]{HS-Novikov}. 
    
    By using the condition that $d^2=0$, we can choose bases in $D^\bullet$ such that all matrices of all the differentials in $D^\bullet$ are diagonal matrices with nonnegative integer entries.
    
   Thus $D^\bullet \simeq D_2^\bullet \tensor_R \Lambda^{E(\Pi)}_\Z$ for an $\Z/N_\Pi\Z$-graded complex of $\Z$-modules $D_2^\bullet$. We can thus write the cohomology of $D^\bullet$ as $H^k(D^\bullet) \simeq (\Lambda^{E(\Pi)}_\Z)^{k_f} \oplus \oplus_{i=1}^{k_t} \Lambda^{E(\Pi)}_\Z/(d^{k, i})$ with $d^{k, i} | d^{k, i+1}$ a sequence of positive integers at least $2$, and all the numbers $k_f, k_t$ and $d^{k, i}$ are well defined.  The  number of generators of $D^\bullet$ is then at least $\sum_{k \in \Z/N_\Pi \Z} k_f + k_t$, which, since $D^\bullet$ is isomorphic to $C_2^\bullet$, agrees with the minimal number of generators of $C_1^\bullet$. 
\end{proof}
\begin{definition}
	When $\Pi$ is a Novikov group, a $\Pi$-equivariant flow category is a $\Pi$-equivariant flow category $\CC$ in the sense of Definition \ref{def:topological-flow-category}, which is equipped with a grading $\mu$, as well as the additional data of an \emph{energy function} $E: Ob(\CC) \to \R$ satisfying
\begin{equation}
		E(g \cdot \tilde{x}) = E(g) + E(\tilde{x}) \text{ and } \mu(g \cdot \tilde{x}) = \mu(g) + 2\mu(\tilde{x}) \text{ for } \tilde{x} \in \CC, g \in \Pi. 
\end{equation}

	We say that $\CC$ as above is $E$-proper if it is proper and if for every pair of elements $x, y \in Ob(\mathcal(\C))/\Pi$, the sets
	\begin{equation}
		\{ (\tilde{x}, \tilde{y}) \in \CC | [\tilde{x}] = x, [\tilde{y}] = y, \CC(\tilde{x}, \tilde{y}) \neq 0, E(\tilde{x}) - E(\tilde{y}) < A\}/\Pi
	\end{equation}
	are finite for all $A \in \R$. We call the quantity $E(\tilde{x}) - E(\tilde{y})$ the \emph{energy} of elements in $\CC(\tilde{x}, \tilde{y})$. 
	
	We say that $\CC$ is $E$-positive if whenever $\tilde{x} \neq \tilde{y}$, $\mathcal{C}(\tilde{x}, \tilde{y}) \neq \emptyset$ implies $E(\tilde{x}) - E(\tilde{y}) > 0$. 
\end{definition}

From now on, $\Pi$ will be a Novikov group. 

If $\sigma'$ is a compatible system of perturbations associated to $\Pi$-equivariant $E$-proper flow-category $\CC$, we have $\Z/2N_{\Pi}\Z$-graded cochain complex $|\CC(\sigma')|_{\Lambda^\Pi_\Z}$ defined as follows. We set also
\begin{equation}
	|\CC(\sigma')|_{\Lambda^\Pi_\Z}^d = \{ \sum r_{\tilde{x}} \tilde{x} | r_{\tilde{x}} \in \Z, \tilde{x} \in \CC, \mu(\tilde{x}) = d + 2N_\Pi \Z, \#\{r_{\tilde{x}} \neq 0 | E(\tilde{x})< A\} < \infty \text{ for all } A \}. 
\end{equation} 
Each of these is a free $\Lambda^{\Pi}_\Z$-module. 

We define the differential by
\begin{equation}
	d \tilde{y} = \sum_{\substack{\tilde{x} \in \CC \\ \mu(\tilde{x}) = \mu(\tilde{y})+1}} \#\CC(\tilde{x}_-, \tilde{x}_+; \sigma') \tilde{x}. 
\end{equation}
The fact that the category $\CC$ is $E$-proper and proper implies that the sums above lie in $|\CC(\sigma')|_{\Lambda^\Pi_\Z}^{\mu(\tilde{y})+1}$. Extending $d$ by $\Lambda{\Pi}_\Z$-linearity, we apply the second part of Lemma \ref{lemma:dimension-count} to conclude 
\begin{lemma}
	$|\CC(\sigma')|_{\Lambda^{\Pi}_\Z}^d$ is a $\Z/2N_\Pi\Z$-graded cochain complex of free $\Lambda^{\Pi}_\Z$-modules.
\end{lemma}

We note that $\Pi_E$ is also a Novikov group, and we can take the quotient of $\mathcal{CC}$ by $\ker E$ to get a $\Pi_E$-equivariant proper $E$-proper flow category, and $\sigma'$ descends to a compatible system of perturbations for this flow category.  We will then have 
\begin{equation}
    |\CC(\sigma')|_{\Lambda^{\Pi_E}_\Z} = |\CC(\sigma')|_{\Lambda^{\Pi}_\Z} \tensor_{\Lambda^{\Pi}_\Z} \Lambda^{\Pi_E}_\Z.
\end{equation}

\begin{definition}
	We say that a $\Pi$-equivariant $E$-positive flow category is \emph{gapped} if the minimum of $\{E(\tilde{x}) - E(\tilde{y}) \in \CC | \CC(\tilde{x}, \tilde{y}) \neq 0, \tilde{x} \neq \tilde{y}\}$ exists. 
\end{definition}

We define 
\begin{equation}
    |\CC(\sigma')|_{\Lambda^{univ}_\Z} = |\CC(\sigma')|_{\Lambda^{\Pi_E}_\Z} \tensor_{\Lambda^{\Pi_E}_\Z} \Lambda^{univ}_\Z.
\end{equation}

\begin{lemma}
\label{lemma:descend-to-lambda_0}
    The complex $|\CC(\sigma')|_{\Lambda^{univ}_\Z}|$ is the tensor product 
    \begin{equation}
        |\CC(\sigma')|_{\Lambda^{univ}_{\Z, 0}} \tensor_{\Lambda^{univ}_{\Z, 0}} \Lambda^{univ}_{\Z}
    \end{equation}
    where $|\CC(\sigma')|_{\Lambda^{univ}_{\Z, 0}}$ is a $\Z/2N_\Pi$-graded complex of $\Lambda^{univ}_{\Z, 0}$-modules. 
\end{lemma}
\begin{proof}
    We note that $\Lambda^{\Pi_E}_\Z \subset \Lambda^{univ}_\Z$. A basis for $|\CC(\sigma')|^k_{\Lambda^{\Pi_E}_\Z}|$ is given by choosing an element $x \in Ob(\CC/\ker E)$ such that $\mu(x) = k \mod 2 N_\Pi$ for each $[x] \in Ob(\CC/\Pi)$. We define a new basis for $|\CC(\sigma')|_{\Lambda^{univ}_\Z}|$ to be given by $T^{-E(x)}(x)$. Now, this basis is independent of the choices of $x$ made. The $E$-positivity of $\CC$ implies that with respect to this basis, the diferentials have entries with coefficients in $\Lambda^{univ}_{\Z, 0}$, proving the claim.
\end{proof}

We can use flow categories to produce maps between chain complexes associated to other flow categories. We give two such constructions below; these can be verified by using the fact that the differential is defined entirely by looking at derived $\langle k \rangle$-orbifold charts of virtual dimension zero, and looking at the combinatorics of the virtual dimension $0$ boundaries of virtual dimension $1$ charts. 
\begin{lemma}
\label{lemma:cone-lemma}
	Say $\sigma'$ is a compatible system of perturbations for a $\Pi$-equivariant $E$-proper proper flow category $\CC$. Suppose also that there exists a pair of full sub-categories $\CC_1, \CC_2$ which are preserved by the $\Pi$-action, and which satisfy $Ob(\CC) = Ob(\CC_1) \sqcup Ob(\CC_2)$, and $\CC(\tilde{x}_2,\tilde{y}_1) = \emptyset$ whenever $\tilde{x}_2 \in \CC_2$ and $\tilde{y}_1 \in \CC_1$. Note that $\sigma'$ restricts to a compatible system of perturbations for $\CC_1$ and $\CC_2$, which we will also denote by $\sigma$. Then we have that 
	\begin{equation}
		|\CC(\sigma')| = Cone(f: |\CC_2[1](\sigma')|_R \to |\CC_1(\sigma')|_{R})
	\end{equation}
	for $f$ a map of chain complexes. 
	Here $\CC_2[1]$ is $\CC_2$ with a new grading $\mu'$ given by $\mu'(\tilde{x}) = \mu(\tilde{x}) +1$. 
	
	If $\CC$ is gapped then the above formula is induced tensoring to $\Lambda^{\Pi_0}_{R}$ from $\Lambda^{\Pi_0}_{R, +}$-linear versions of the corresponding chain complexes.
\end{lemma}

The following lemma is, in essense, an iterated version of the previous lemma: 
\begin{lemma}
\label{lemma:square-diagram}
    Say we are given a $\Pi$-equivariant $E$-proper proper flow category $\mathcal{C}$ equipped with a system of derived $\langle k \rangle$-orbifold charts, which contains four full subcategories $\mathcal{C}_1, \ldots, \CC_4$, preserved by the $\Pi$-action. Suppose that we only have morphisms in $\CC$ between objects which are monotonically non-strictly decreasing under the following order relations: $Ob(\CC_1) \leq Ob(\CC_2) \leq Ob(\CC_4), Ob(\CC_1) \leq Ob(\CC_3) \leq Ob(\CC_4)$. Let $\sigma'$ be a compatible system of perturbations for $\CC$, which naturally restricts to the same for the aforementioned subcategories.
    
    Then there is a homotopy-commutative diagram 
    \begin{equation}
    \begin{tikzcd}
    \|\CC_4[2](\sigma')\|_{\Lambda^{\Pi_E}_\Z} \ar[r] \ar[d]  & \|\CC_3[1](\sigma')\|_{\Lambda^{\Pi_E}_\Z} \ar[d] \\
    \|\CC_2[1](\sigma')\|_{\Lambda^{\Pi_E}_\Z} \ar[r] & \|\CC_1[2](\sigma')\|_{\Lambda^{\Pi_E}_\Z}
    \end{tikzcd}
    \end{equation}
    where the horizontal and vertical arrows are induced as in Lemma \ref{lemma:cone-lemma} by restricting to the corresponding full subcategories on pairs in $\{\CC_1, \ldots \CC_4\}$. 
\end{lemma}

\paragraph{Definition of Floer complex}
The work of Sections \ref{sec:global-charts} and \ref{sec:compatible-smoothing} can be summarized by the statement that there is a compatible system of derived $\langle k \rangle$-orbifold charts for a flow category with morphisms given by Floer trajectories. We spell out the details below. 

We define $\CC(H, J)$ as the topological flow category with objects given by $\widetilde{Fix}(H)$, with $\CC(H, J)(\tilde{x}, \tilde{y}) = \cM(\tilde{x}, \tilde{y})$, with $\mathcal{S}_{\CC(H, J)(\tilde{x}, \tilde{y})} = [\bar{\mathcal{A}}(\tilde{x}) - \bar{\mathcal{A}}(\tilde{y}) - 1]$ if $\cM(\tilde{x}, \tilde{y}) \neq \emptyset$, and otherwise $\mathcal{S}_{\CC(H, J)(\tilde{x}, \tilde{y})} = \bullet$.  Composition (whenever $\cM(\tilde{x}, \tilde{y}) \neq 0$ and $\cM(\tilde{y}, \tilde{z}) \neq 0$) is given by the gluing of of flow trajectories with the corresponding maps on the $\mathcal{S}_{\CC(H, J)(\tilde{x}, \tilde{y})}$ given by the isomorphism of $\mathcal{S}_{\CC(H, J)(\tilde{x}, \tilde{y})} \sqcup \mathcal{S}_{\CC(H, J)(\tilde{y}, \tilde{z})}$ with 
\begin{equation}
\label{eq:set-map-for-flow-category}
     \mathcal{S}_{\CC(H, J)(\tilde{x}, \tilde{y})} \cup (\mathcal{S}_{\CC(H, J)(\tilde{y}, \tilde{z})} + |\mathcal{S}_{\CC(H, J)(\tilde{x}, \tilde{y})}|) \subset \mathcal{S}_{\CC(H, J)(\tilde{x}, \tilde{z})}.
\end{equation}
We set $E(\tilde{x})$ to be the usual energy $\mathcal{A}_\omega(\tilde{x})$ of a cap. We set $\mu(\tilde{x})$ to be its Conley-Zehnder index. 
The recapping action maps $\CC_{Floer}(H, J)$ into a $\Pi$-equivariant flow category (\ref{eq:novikov-group}). 

We have 
\begin{lemma}
   The flow category $\CC(H, J)$ is a proper and $E$-proper $\Pi$-equivariant flow category admitting a compatible system of derived $\langle k \rangle$-orbifold charts. 
\end{lemma}
\begin{proof}
    The fact that this category is proper follows from Gromov-Floer compactness and Lemma \ref{lemma:gromov-compactness-with-modified-energy}. The $\langle k \rangle$-orbifold charts are constructed in Lemma \ref{lemma:compatible-system-of-derved-orbifold-charts-1}. 
\end{proof}

We define $CF^\bullet(H, J; \sigma') := |\CC(H, J)(\sigma')|_{\Lambda^{\Pi_E}_{\Z}}$. . 

Explicitly, this complex depended on 
\begin{itemize}
    \item The integralization data $(\tilde{\Omega}, N_1, \bar{H})$;
    \item The perturbation data $\Lambda_d$;
    \item The stabilizing vector bundles arising during the smoothing process;
    \item The compatible system of perturbations $\sigma'$. 
\end{itemize}

later in this paper, we will sketch a proof of how one can prove independence of the quasi-isomorphism class of this complex on the data above. However, this is not necessary for a verification of the integral Arnol'd conjecture, which we proceeed with below. 

\section{Continuation map and proof of integral Arnold conjecture}

In this section, we explain how to use the techniques described above to prove Theorem \ref{thm:integral-arnold}.
We will choose a $C^2$-small morse function $H_0$, a small time-dependent perturbation $H'_0$ of it, and a PSS map from the morse complex of $H_0$ to the Floer complex of $H'_0$, which will be an isomorphism. We will also construct continuation maps betwen the Floer complex of $H'_0$ and the Floer complex of our given Hamiltonian $H$ and show that they are quasi-isomorphisms, by showing that they compose to maps homotopic to the continuation map from the Floer complex of a Hamiltonian to itself that is induced by an $s$-independent family of Hamiltonians. 

\subsection{Moduli spaces for continuation map}
We let $\mathcal{F}^{cont}(\CP^{\ell}, d)$ be the set of non-nodal points in $\mathcal{C}(\CP^\ell, d)$ lying on one of the \emph{distinguished geodesics} of the universal curve. Thus we have a projection 
\begin{equation}
\label{eq:continuation-submersion}
    \pi: \mathcal{F}^{cont}(\CP^{\ell}, d) \to \mathcal{F}(\CP^{\ell}, d);
\end{equation}
The universal curve $\mathcal{C}^{cont}(\CP^\ell, d)$ over $\mathcal{F}^{cont}(\CP^{\ell}, d)$ is the pullback of the universal curve $\mathcal{C}(\CP^{\ell}, d)$ over $\mathcal{F}(\CP^{\ell}, d)$ via $\pi$, and its fibers over points of $\mathcal{F}^{cont}(\CP^{\ell}, d)$ are \emph{McLean-stabilized continuation domains}, namely, nodal curves mapping to $\CP^{\ell}$ and equipped with a certain geodesics as well as a \emph{marked point} $c$ on the interior of one of the geodeiscs.

The space $\mathcal{F}^{cont}(\CP^{\ell}, d)$ has the structure of a $\langle 2d-2 \rangle$-manifold. The stratification is given as follows: think of $[2d-2]$ as the spaces in between $2d-1$ lines, with the middle line marked by a star.  Otherwise, place a dot on the spaces in between $2d-1$-lines corresponding to the elements of $[2d-2]\setminus S$. Write $d(S) = (d_1, d_{r_1}, d', d''_1, \ldots d''_{r_2})$ with $d_1, \ldots, d_{r_1}$ for the numbers of lines in the $j$-th contiguous segment, reading left to right, lying \emph{left} of the contiguous segment containing the line marked by the dot; write $d''_1, \ldots, d''_{r_2}$  for the numbers of lines for the numbers of lines in the $j$-th contiguous segment, reading left to right, lying \emph{right} of the contiguous segment containing the line marked by the dot; and write $d' = d - \sum_{j=1}^{r_1} d'_j - \sum_{j=1}^{r_2} d''_j$. We set $\mathcal{F}^{cont}(\CP^{\ell}, d)(S)=\emptyset$ if $d' \leq 0$. Otherwise, we set
\begin{equation}
    \mathcal{F}^{cont}(\CP^{\ell}, d)(S)  \subset \mathcal{F}(\CP^\ell, d'_1) \times_{\CP^\ell} \ldots \times_{\CP^\ell} \mathcal{F}^{cont}(\CP^{\ell}, d') \times_{\CP^\ell} \ldots \times_{\CP^\ell} \mathcal{F}(\CP^\ell, d''_{r_2})
\end{equation}
to consist of those elements in $\mathcal{F}^{cont}(\CP^\ell, d)$ arising from joining elements of the target of the inclusion above along the respective marked points. We write $S = S(d'_1, \ldots; d; \ldots, d''_{r_2})$ for the stratum defined above. 

We will denote the component of a McLean stabilized continuation domain as $S^2_{cont} = S^2_{j_c}$, where $1 \leq j_c \leq r_{u}$ in the notation of Section \ref{sec:moduli-of-decorated-stable-maps}. 

\paragraph{Morse functions on the simplex and continuation data}

Recall that the $n$-simplex $\Delta^n$ with vertices labeled $0, \ldots, n$ has a \emph{canonical Morse function} \cite{PardonHam} $f_{\Delta^n}$ with critical points at the vertices of index a linear function of the vertex label. We will arrange for flow of the Morse function to increase the vertex label. 

A \emph{continuation datum} from $H_+$ to $H_-$ is a map $H^c: \Delta^1 \to \mathcal{H}(M)$, where  $\mathcal{H}(M)$ is the space of $S^1$-dependent Hamiltonians on $M$, and $H^c(0) = H_-, H^c(1) = H_+$, and $H^c$ a locally constant function near the vertices. This implies that, for any Morse trajectory of $f_{\Delta^1}$ from $0$ to $1$, we get by pullback to the domain a family $H_s$ of $S^1$-dependent Hamiltonians on $M$ for $s \in \R$, which are constant in neighborhoods of infinity and are well defined up to $s$-translation. We can choose $H^c$ so that these are only non-constant in a $s$-strip of size $1$. Make a fixed choice of Morse trajectory and and a corresponding $H_s$.

Write, as usual, $z = s+it$.

One then has the moduli spaces of Floer trajectories
\begin{equation}
\begin{gathered}
    \M^{cont}(\tilde{x}, \tilde{y}) = \{ u: Z \to M \,|\,\partial_s u(z) + J_z(\partial_t u - X_{H_s}) = 0; \lim_{s \to -\infty} u(s,\cdot) = x(\cdot), \lim_{s \to +\infty} u(s,\cdot) = y(\cdot), \,\tilde{x}\# u = \tilde{y}\} \\\text{ where } \tilde{x} \in \widetilde{Fix}(H_-), \tilde{y} \in \widetilde{Fix}(H_+).
    \end{gathered}
\end{equation}
and their corresponding Gromov-Floer compactifications $\cM^{cont}(\tilde{x}, \tilde{y})$.

\subsubsection{Convenient integral forms for the continuation map.}
\label{sec:integral-forms-for-continuation-map}
Recall that in the usual Floer theoretic setup, we have that for a Floer trajectory $u \in \M(\tilde{x}_-, \tilde{x}_+)$
\begin{equation}
    E_\Omega(u) = \int_Z \|\partial_s u\|^2_{\Omega} = \int_Z u^*\Omega - \partial_s (H(u(s,t)) ds dt + \int \Omega(X^{\Omega}_H-X^\omega_H, \partial_s u) ds dt. 
\end{equation}
This identity follows from a pointwise computation, and the first integral is exactly $\int_Z u^*\Omega - d(H dt) = \mathcal{A}_\Omega(\tilde{x}_-) - \mathcal{A}_\Omega(\tilde{x}_+)$. Since this computation was pointwise, the same computation implies that for $u \in \M^{cont}(\tilde{x}_-, \tilde{x}_+)$,
\begin{equation}
    E_\Omega(u) = \int_Z \|\partial_s u\|^2_{\Omega} = \int_Z u^*\Omega - (dH_s)(\partial_s u(s,t)) ds dt + \int \Omega(X^{\Omega}_{H_s}-X^\omega_{H_s}, \partial_s u) ds dt. 
\end{equation}
We now have two distinct action functionals $\mathcal{A}_{\Omega}^\pm$ defined using $H_\pm$. We then have that 
\begin{equation}
    \mathcal{A}_{\Omega}^-(\tilde{x}_-) - \mathcal{A}_\Omega^+(\tilde{x}_+) = \int_Z u^*\Omega - d(H_s dt) = \left\{\int u^*\Omega - (dH_s)(\partial_s u(s,t)) ds dt \right\} -\int_0^1\int_{-\frac{1}{2}}^{\frac{1}{2}} H'_s(u(s,t)) ds dt
\end{equation}
where $H'_s(x) = \partial_s H(x)$ for $x \in M$. We thus conclude that 
\begin{equation}
\label{eq:energy-indentity-continuation-map}
    \mathcal{A}_{\Omega}^-(\tilde{x}_-) - \mathcal{A}_\Omega^+(\tilde{x}_+) = E_\Omega(u) -\int_0^1\int_{-\frac{1}{2}}^{\frac{1}{2}} H'_s(u(s,t)) ds dt  - \int \Omega(X^{\Omega}_{H_s}-X^\omega_{H_s}, \partial_s u) ds dt. 
\end{equation}

Note that the absolute value of the middle term is bounded on both sides by a rational constant $K$

Now, we have open neighborhoods $U_\pm$ of $Fix(\phi_{H_\pm})$ constructed by Lemma \ref{lemma:make-form-zero-near-orbits}. Taking $U = U_- \cup U_+$ is now an open neighborhood of $Fix(\phi_{H_-}) \cup Fix(\phi_{H_+})$, and the conclusion of Lemma \ref{lemma:make-form-zero-near-orbits} still holds for this $U$. We take $\epsilon_0 = \min(\epsilon_0^-, \epsilon_0^+)$ where $\epsilon_0^\pm$ is the constant in Lemma \ref{lemma:small-l2-near-fixed-locus} for $H = H_\pm$ and $U=U_\pm$. 

Defining
\begin{equation}
    S^{\epsilon_1}_{\omega, H_s}  = \left\{ \Omega\in \Omega^2(M) | \Omega \text{ symplectic }, |\Omega(X^\Omega_{H_s} - X^\omega_{H_s}, v)| < \epsilon_1\|v\|_\omega \text{ for all } s \in \R, \|v\|_\Omega \geq \frac{1}{2} \|v\|_\omega \text{ for } v \in TM, \Omega|_U = \omega \right\}.
\end{equation}
we can again pick a rational $\Omega \in S^{\epsilon_1}_{\omega, H_s}$ for $\epsilon_1 = \sqrt{\epsilon_0}/8$. We define $L$ as in the proof of Lemma \ref{lemma:integral-taming-form}, The computation of \eqref{eq:bound-on-vector-field-error} now gives 
\begin{equation}
    \int_{\R \times S^1} |\Omega(X^\Omega_{H_s} - X^\omega_{H_s}, \partial_s u)| ds dt \leq \frac{\epsilon_1}{\sqrt \epsilon_0} (E^{cyl}_\omega(u) + (E^{cyl}_\omega(u))^{\frac{1}{2}})
\end{equation}
where the new term is because we separately have to apply Cauchy-Schwartz to \[\int_{[-1/2, 1/2] \times S^1} |\Omega(X^\Omega_{H_s} - X^\omega_{H_s}, \partial_s u)| ds dt\] since we have not arranged Lemma \ref{lemma:make-form-zero-near-orbits} to apply for $H = H_s$. Plugging this into \eqref{eq:energy-indentity-continuation-map} we conclude that 
\begin{equation}
    \mathcal{A}_{\Omega}^-(\tilde{x}_-) - \mathcal{A}_\Omega^+(\tilde{x}_+) + K \geq 2 E_\omega(u) -\frac{1}{8} (E^{cyl}_\omega(u) + (E^{cyl}_\omega(u))^{\frac{1}{2}}), \text{ and thus } \mathcal{A}_{\Omega}^-(\tilde{x}_-) - \mathcal{A}_\Omega^+(\tilde{x}_+) + K +\frac{1}{4} \geq E_\omega(u). 
\end{equation}
We define $\tilde{\Omega}$ using Lemma \ref{lemma:make-form-zero-near-orbits}. Now our chosen $\Omega$  still satisfies the conclusion of Lemma \ref{lemma:integral-taming-form} for $\mathcal{A}_\Omega$ replaced by \emph{either} of $\mathcal{A}_\Omega^+$ and $\mathcal{A}_\Omega^-$; we set $M'_\pm$ to be the minimum of the quantities $M'$ for each of these functionals.  Choose $\epsilon_x$ such that that $\mathcal{A}^\pm_\Omega(\tilde{x}_\pm) + \epsilon_x \subset \Q$ for $\tilde{x}_\pm \in Fix(\phi_{H_\pm})$, and choose $N_1$ such that $K' = N_1 (K + \frac{1}{4}) \subset \Z$ and $N_1(\mathcal{A}^\pm_\Omega(\tilde{x}_\pm) + \epsilon_x) \subset \Z$. We fix $\Omega_0 = N_1 \tilde{\Omega}$ and write $\bar{\mathcal{A}}^\pm$ defined by the formula \eqref{eq:define-a-bar} for $H = H_\pm$. 

We have that 
\begin{equation}
    \{\tilde{x}_\pm \in Fix(\phi_{H_\pm}), u \in \cM^{cont}(\tilde{x}, \tilde{y}) | \bar{\mathcal{A}}^-(\tilde{x}_-) - \bar{\mathcal{A}}^+(\tilde{x}_+) \leq K \}/\Pi
\end{equation}
is compact for each $K$. We choose $\epsilon_2$ to be the minimum of the quantities $\epsilon_2$ in the paragraph above Lemma \ref{lemma:bar-h} for $H = H_\pm$, and we choose $\bar{H}: M \times \R_s \times S^1_t$ to be equal to be defined as in \eqref{eq:bar-h} for $s < \frac{1}{2}$ and $H=H_-$, similarly for $s > \frac{1}{2}$ for $H=_+$, and otherwise extended smoothly. We have 
\begin{lemma}
\label{lemma:H-bar-continuation}
    For every smooth map
	\begin{equation}
		u: \R \times S^1 \to M, \lim_{s \to \pm\infty} u(s,t) = x_\pm(t) \text{(uniformly)}, u \# \bar{x}_- = \bar{x}_+
	\end{equation}
	with $x_\pm(t) \in Fix(H_\pm)$, 
	 define the $2$-form 
	\begin{equation}
		\label{eq:floer-stabilizing-form}
		\Omega^c_u \in \Omega^2(\R \times S^1), \Omega_u(s,t) = u^*\Omega_0(s,t) - d(\bar{H}_s) dt. 
	\end{equation}
	The form $\Omega_u$ is zero near $s = \pm \infty$. Moreover, $\int \Omega_u \in \Z$ agrees with $\bar{\mathcal{A}}(\tilde{x}) - \bar{\mathcal{A}}(\tilde{y})$, and if $u \in \M^{cont}(\tilde{x}_-, \tilde{x}_+)$, then $\int_Z \Omega_u + K' > 0$.
\end{lemma}

\subsubsection{The continuation moduli space.}
We can now make $\cM^{cont}(\tilde{x}_-, \tilde{x}_+)$ into a $[2(\bar{\mathcal{A}}^-(\tilde{x}_-) - \bar{\mathcal{A}}^+(\tilde{x}_+) + K')-2]$-space. Recall that the elements of this spaces are tuples $u = (u^-_1, \ldots u^-_{r_1}, u^c, u^+_{1}, \ldots, u^+_{r_2})$ where $u^\pm_j$ is bubbled, $1$-level Floer trajectory from $\tilde{x}_\pm^{j-1}$ to $\tilde{x}_-\pm^j$ with respect to the Hamiltonian, where we have set $\tilde{x}_-^0 = \tilde{x}_-, \tilde{x}_+^{r_2} = \tilde{x}_+$, and where $u^c$ is a solution to the continuation equation with bubbles attached from $\tilde{x}_-^{r_1}$ to $\tilde{x}_+^{0}$. We assign to such a $u$ to open stratum associated to the set 
\begin{equation}
    S = S(\bar{\mathcal{A}}^-(\tilde{x}_-^0) - \bar{\mathcal{A}}^-(\tilde{x}^1_-), \ldots, ;\bar{\mathcal{A}}^-(\tilde{x}^-_{r_1}) - \bar{\mathcal{A}}^+(\tilde{x}^+_0) + K'; \ldots, \bar{\mathcal{A}}^+(\tilde{x}^+_{r_2-1}) - \bar{\mathcal{A}}^+(\tilde{x}^+_{r_2})). 
\end{equation}

\paragraph{Thickenings}
Write $S^{cont}_d \subset 2^{[2d-2]}$ for the set of nonempty strata of $\mathcal{F}^{cont}(\CP^\ell, d)$. There is a natural map of posets $S^{cont}_d \to 2^{[d-1]}$ which sends codmension strata to codimension $r$ strata; with respect to this map, $\pi$ of (\ref{eq:continuation-submersion}) is a submersion.
 We have that 
    \begin{equation}
    \label{eq:continuation-map-stratum-induction}
        \mathcal{F}_{S(d'_1, \ldots; d'_{r_1}; d'; d''_1, \ldots, d''_{r_2})}(\CP^d, d) = Ind(\mathcal{F}(\CP^{d'_1}, d'_1), \ldots, \mathcal{F}^{cont}(\CP^{d'}, d'), \ldots, \mathcal{F}(\CP^{d''_{r_2}}, d''_{r_2})). 
    \end{equation}
is expressible as a sequence of generalized inductions and restrictions from a product, as in Section \ref{sec:unitary-group-actions}, whenever $d'_1 + \ldots + d'_{r_2} + d'+ \ldots d''_1 + \ldots + d''_{r_2} = d$.

A perturbation datum over a substratum of $\mathcal{F}^{cont}(\CP^\ell, d)$ is defined exactly as in Definition \ref{def:perturbation-data}. Formula \ref{eq:compatible-system-of-perturbations-condition} regarding compatibility of inductions along decompositions \eqref{eq:continuation-map-stratum-induction} continues to hold, making ``compatibility of perturbation data'' adapt to this setting as in Definition \ref{def:compatible-choice-of-perturbation-data}, with $\mathcal{F}$ replaced by $\mathcal{F}^{cont}$. Assume that now we have fixed a compatible choice of perturbation data $\Lambda^{cont}_d = (V^{cont}_d, \lambda^{cont}_d)$ over $\mathcal{F}^{cont}(\CP^d, d)$ for all $d$, extending a compatible set of perturbation data $\Lambda^\pm_d$ over $\mathcal{F}(\CP^d, d)$ chosen for $H^+$ and $H^-$. We note that we can choose $\Lambda^+_d = \Lambda^-_d =: \Lambda_d$ as we only need to make the pertubation data ``large enough'' to achieve regularity of the thickenings for the Floer moduli spaces. 

For the remainder of this paper we \emph{fix} a sequence of Hermitian line bundles $L_k$ on $\CP^1$ for every integer $k$. 

\begin{definition}
	\label{def:framed-contination-trajectory} A framed bubbled pre-continuation trajectory is a tuple $(u, \Sigma, F)$, where 
\begin{itemize}
    \item $\Sigma$, the \emph{domain} of the trajectory, is a genus zero nodal curve with two marked points $p_-, p_+$, equipped with geodesic segments from $p_{j-1}$ to $p_j$ $S^2_J$, with $S^2_1, \ldots, S^2_j, \ldots, S^2_{r_u}$ the sequence of components  from $p_- = p_0$ to $p_+ = p_{r_u}$. One of the geodesic segment also carries a marked point on $S^2_{cont} = S^2_{j_c}$.
    \item The map $u: \Sigma \setminus \{p_0, \ldots, p_{r_u}\} \to M$ is a smooth function such that the integral $2$-form $\Omega^c_u \in \Omega^2(\Sigma)$ defined subsequently satisfies 
    \begin{equation}
    \label{eq:shift-the-line-bundle}
        \int_{\Sigma_\alpha} \Omega^c_u + K' > 0 \text{ for every unstable component } \Sigma_\alpha. 
    \end{equation}
    The form $\Omega^c_u$ is defined to be $u|_{S^2_\pi}^* \Omega_0$ (see \eqref{eq:floer-stabilizing-form}) on each component $S^2_\pi$ of $\Sigma$ that is not one of the $S^2_j$, and on each $S^2_j$ with $j \neq j_c$ equal to the extension by zero of the form denoted $\Omega_u$ in (\ref{eq:floer-stabilizing-form}) under any preferred biholomorphism $\phi_j: S^2_j \setminus \{p_{j-1}, p_j\} \simeq \R \times S^1$. Finally, on $S^2_{j_c}$ it is defined to be the extension by zero of the form denoted $\Omega^c_u$ above under the unique biholomorphism $\psi_{u, j_c}$ which maps $Z$ to $S^2_{j_c} \setminus \{p_{j_c-1}, p_{j_c}\}$ such that $\psi_{u, j_c}$ sends the marked point on the geodesic to a point to $s=0$. 
    \item Additionally, the map is required converge uniformly to Hamiltonian loops at the ends corresponding to the $p_i$ under the biholomorphisms described above;
    \item Moreover, writing $u^*g$ for the pull-back of the metric on $M$ by $u$, and $dV_{u^*g}$ for the associated volume form (which will in general not be defined on the $p_i$), we have the volume of $\Sigma$ with respect to $dV_{u^*g}$ is finite. 
    \item The datum $F = (f_0, \ldots, f_d)$ is a $\C$-basis for $H^0(L_{\Omega_u + K'})$, where $L_{\Omega_u+K'}$ is the line bundle on $\CP^1$ which is obtained from $L_\Omega$ by tensoring it with the line bundle which is trivial on every component except for $S^2_{j_c}$, and agrees with $L_{K'}$ on $S^2_{j_c}$. Note that $L_{\Omega_u + K'}$ has positive degree on every unstable component by Equation \eqref{eq:shift-the-line-bundle}. 
\end{itemize}   
\end{definition}
We continue to use notation as in the paragraph subsequent to Definition \ref{def:framed-floer-trajectory} for pre-continuation trajectories.

We now define thickenings $\mathcal{T}_{\Lambda^{cont}_d}$ to consist o pairs $((u, \Sigma, F), \eta)$ where $(u, \Sigma, F)$ is a framed bubbled pre-continuation trajectory, with $u_F \in \mathcal{F}^{cont}(\CP^d, d)$ and $\eta \in (V^{cont}_d)$, which satisfies 
\begin{equation}
    \begin{gathered}
    	\partial_s u_j+ J(\partial_t u_j - X_{H^-}) + \lambda_d(\eta) \circ di_F = 0 \text{ for } 1 \leq j < j_c, \text{ where } u_j= u|_{S_j} \circ \psi_{u, j}.  \\
    	\partial_s u_j+ + J(\partial_t u_j - X_{H^+}) + \lambda_d(\eta) \circ di_F = 0 \text{ for } j_c < j \leq r_u, \\
    	\partial_s u_{j_c} + J \partial_t(u_{j_c}- X_{H_s}) + \lambda_d(\eta) \circ di_F= 0 , \\
 	\bar{\partial}u|_{S^2_\alpha} + \lambda_d(\eta)\circ di_F= 0 \text{ for all other components } S^2_\alpha \neq S^2_1, \ldots, S^2_{r_u}
    \end{gathered}
\end{equation}

	We then define $\mathcal{K}_{\Lambda^{cont}_d} =(PU(d+1), T_{\Lambda^{cont}_d}, \mathcal{V}_{\Lambda^{cont}_d} \oplus i \mathfrak{pu}_{d+1}, \sigma_{\Lambda^{cont}_d} \oplus \sigma_H)$ defined exactly as in Definition \ref{def:thickening}. We write $\mathcal{T}^\epsilon_{\Lambda^{cont}_d}$ to to be the sub-chart of $\mathcal{T}_{\Lambda^{cont}_d}$ for which $\|\eta\| < \epsilon$. We have that 
 	    $\overline{\M}^{cont}(d)$, the quotient of the zero section of the Kuranishi section of $\mathcal{K}_{\Lambda^{cont}_d}$ by the $PU(d+1)$-action, is the set of bubbled continuation trajectories in $M$ (of $\Omega_u$-degree $d$). We write 
 	    \begin{equation}
 	        \cM^{cont}(d) = \bigcup_{\substack{x_\pm \in \widetilde{Fix}(\phi_{H_\pm}^1) \\ \bar{\mathcal{A}}^-(\tilde{x}_-) - \bar{\mathcal{A}}^+(\tilde{x}_+) + K' = d}}  \cM^{cont}(\tilde{x}_-, \tilde{x}_+). 
 	    \end{equation}

 As before, we have that, for sufficiently small $\epsilon >0$, the $PU(d+1)$ action on $\mathcal{T}^\epsilon_{\Lambda^{cont}_d}$ has finite stabilizers, and that the induced topology on $\cM^{cont}(d)$ agrees with the Gromov-Floer topology. The argument of Lemma \ref{lemma:surjective-perturbation-data-exist-2} adapts without change to show that one can choose compatible perturbation data $\Lambda^{cont}_d$ such that the extended linearized operators are always surjective on continuation solutions, which implies that $T_{\Lambda^{cont}_d}^\epsilon$ is a $PU(d+1)-\langle d-1\rangle$-manifold with projection $\Pi$ to $\mathcal{F}^{cont}(\CP^d, d)$ a fiberwise-smooth $C^1_{loc} PU(d+1)$ bundle. 
 
 In the boundary of $\cM^{cont}(\tilde{x}_-, \tilde{x}_+)$, we have (broken bubbled) Floer trajectories for $H_\pm$ appearing. We denote the moduli spaces of these trajectories via $\cM^\pm(\tilde{x}_\pm, \tilde{y}_\pm)$ for $\tilde{x}_\pm, \tilde{y}_\pm \in \widetilde{Fix}(H_\pm)$.
 
\subsubsection{Flow category for continuation map}
We now construct a proper, $E$-proper, $\Pi$-equivariant graded flow category $\CC(H^c, J)$, where $(\Pi, E, \mu)$ are a Novikov group and $E$ is associated to $\Omega$ as chosen in this section. 

The objects of $\CC(H^c, J)$ are $\widetilde{Fix}(H_-) \sqcup \widetilde{Fix}(H_+)$. The full subcategory on $\widetilde{Fix}(H_-)$ is $\mathcal{C}(H_-, J)$, and the full subcategory on $\widetilde{Fix}(H_+)$ is $\mathcal{C}(H_+, J)[-1]$, as needed to apply Lemma \ref{lemma:cone-lemma}. This fixes the gradings, the $E$-functions, and the $\Pi$-action. We set $\CC(H^c, J)(\tilde{x}_-, \tilde{x}_+) = \cM^{cont}(\tilde{x}_-, \tilde{x}_+)$ for $\tilde{x}_\pm \in \widetilde{Fix}(H_\pm)$. We set  $\mathcal{S}_{\CC(H^c, J)(\tilde{x}_-, \tilde{x}_+)} = [2(\mathcal{A}^-(\tilde{x_-}) - \mathcal{A}^+(\tilde{x_+}) + K')-2]$. To define composition, we glue bubbled broken Floer trajectories to bubbled broken continuation map trajectories in the standard way; the composition on the sets $\mathcal{S}_{\CC(H^c, J)(\tilde{x}, \tilde{y})}$ is defined exactly as in \eqref{eq:set-map-for-flow-category} but with $\CC_(H, J)$ replaced by $\CC(H^c, J)$.

The methods of Section \ref{sec:compatible-smoothing} adapt verbatim to this setting to show
\begin{proposition}
\label{prop:charts-for-continuation-category}
The flow category $\CC^{cont}$ admits a compatible system of derived $\langle k \rangle$-orbifold charts $\CC'_{cont}$. 
\end{proposition}
\begin{proof}
There is nothing new in the argument, except that we must make sure that the analog of Proposition \ref{prop:fiberwise-c1loc-on-kuranishi-chart} for $\mathcal{T}_{\Lambda^{cont}_d}$ still holds. However, as we use Morse trajectories to parameterize the continuation datum, exactly the same argument goes through. This is immaterial for a single continuation map, but becomes important in the subsequent Proposition \ref{prop:charts-for-continuation-homotopy}, where it becomes important that we use the gluing of Morse trajectories to parameterize gluing of continuation data, as does \cite[Appendix C]{PardonHam}.  
\end{proof}

Now, fixing compatible systems of perturbations $\sigma'_\pm$ for $\mathcal{C}(H_\pm, J)$, we use Proposition \ref{prop:transverse-perturbations-exist-inductively} to extend this to a compatible system of perturbations for $\CC(H^c, J)$, and then apply Lemma \ref{lemma:cone-lemma} to produce a map 
\begin{equation}
    c_{H_-, H_+}: |\mathcal{C}(H_+, J)(\sigma'_+)| \to |\mathcal{C}(H_-, J)(\sigma'_-)|
\end{equation} 
called the \emph{continuation map}. The construction of this map depended on the choice of $\Omega$ and $\bar{H}$; the choice of $H_s$; the choices of perturbation data, the stabilizations performed in the repeated application of Proposition \ref{prop:upgraded-smoothing};  and on the compatible systems of perturbations. 

We note that shifting $E|_{\widetilde{Fix}(H_-)} \to E|_{\widetilde{Fix}(H_-)}+C$ for a sufficiently large constant $C$ makes this category into an $E$-positive flow category. Write $c_{\Lambda^{univ}} = c \tensor_{\Lambda^{\Pi_E}_\Z} 1: |\mathcal{C}(H_+, J)(\sigma'_+)|_{\Lambda^{univ}_\Z} \to |\mathcal{C}(H_-, J)(\sigma'_-)|_{\Lambda^{univ}_\Z}$ for the corresponding map over the universal Novikov ring. By working through Lemma \ref{lemma:cone-lemma} and Lemma \ref{lemma:descend-to-lambda_0}, we see that this change of $E$ modifies $c_{\Lambda^{univ}}$  by multiplying it by it by $T^{C}$; in particular, $T^{C}c_{\Lambda^{univ}}$ has coefficients in $\Lambda_{\Z, 0}^{univ}$.  


\subsection{Compositions of continuation maps are homotopic}
Let $H^1, H^2, H^3, H^4$ be four $S^1$-families of Hamiltonians with nondegenerate fixed points, none of which are constant Hamiltonian trajectories. Supose we choose continuation data from $H^1$ to $H^2$, from $H^1$ to $H^3$, from $H^2$ to $H^4$, and from $H^3$ to $H^4$. 

A \emph{homotopy datum} $H^h$ is a pair of maps $\Delta^{2} \to \mathcal{H}(M)$ which glue together to a smooth map from a square, are locally constant near the vertices, and restrict to the continuation data as in the below diagram (the top left vertex corresponds to the $0$ vertex on both simplices, and the bottom right to the $2$ vertex):
\begin{equation}
\label{eq:commutative-square-of-Hamiltonians}
\begin{tikzcd}
H_4 \ar[r] \ar[d] & H_3 \ar[d] \\ H_2 \ar[r] & H_1
\end{tikzcd}
\end{equation}

Any Morse flow line $\tau$ on either of these simplices gives a family of Hamiltonians $H_{s, \tau}$ which could come from a continuation datum. 

Choose a \emph{single}  $(\Omega, N_1)$ such that, together with choices of perturbed action functionals $\bar{\mathcal{A}}^j$ for $j=0,1,2$, as well as choices of functions $\bar{H}_j$, $j=1 \ldots 4$, as well as a smooth family of functions $\bar{H}_{s, \tau}$, for each morse trajectory $\tau$ in the interior of the simplices, all defined via the construction above Lemma \ref{lemma:H-bar-continuation} such that these functions all satisfy the conclusion of Lemma \ref{lemma:H-bar-continuation} and such that, viewing $\bar{H}_{s, \tau}$ as a map from the interior of the moduli space of Morse trajectories on the square to $\mathcal{H}(M)$, on the boundary it restricts to the map specified by the $\bar{H}_j$ as in the diagram \eqref{eq:commutative-square-of-Hamiltonians}. For every (possibly) broken Morse trajectory $\tau$ in the space of Morse trajectories on the square, we thus have defined an appropriate two-form $\Omega_{u, \tau}^c$ which when integrated over a broken continuation trajectory gives the corresponding difference of integralized actions $\bar{\mathcal{A}}^j$.

With these choices, we have that for any solution of the continuation map equation for $H_{s, \tau}$ for any fixed $\tau$, or for $H_s^{01}$ or $H_s^{1,2}$ from $\tilde{x}_2 \in \widetilde{Fix}(H_2)$ to $\tilde{x}_0 \in \widetilde{Fix}(H_0)$, we have that 
\begin{equation}
    \int \Omega_{u, \tau}^c = \bar{A}^2(\tilde{x}_2) - \bar{A}^0(\tilde{x}_0) \in \{ \ell \in \Z | \ell \geq K' \} 
\end{equation}
for some fixed $K' \in \Z$. We then use \emph{this} $K'$ to define continuation maps described in the next paragraph. 

Write the flow category used to define the continuation map $c_{H^j, H^i}: |\mathcal{C}(H_j, J)(\sigma'_+)| \to |\mathcal{C}(H_i, J)(\sigma'_-)|$ as $\CC(H^{j}, H^i; J)$ with $(j, i)$ the sources and targets of arrows in \eqref{eq:commutative-square-of-Hamiltonians}, with associated systems of derived $\langle k \rangle$-orbifold charts $\CC'(H^{j}, H^i; J)$ and  choices of compatible systems of perturbations $\sigma_{H^j, H^i}'$ used to define these maps.

\begin{proposition}
The following maps are homotopic as maps of $\Lambda^{\Pi_E}_{\Z}$-modules:
\begin{equation}
    c_{H^4, H^3} c_{H^3, H^1} \sim c_{H^4, H^2} c_{H^2, H^1}. 
\end{equation}
\end{proposition}
\begin{proof}
We will construct a diagram of flow categories as in Lemma \ref{lemma:square-diagram}. We choose $\CC_j = \CC(H_j, J)$ for $j=0, \ldots, 3$. The full subcategories on pairs of these will be $\CC(H^{j}, H^i; J)$ in the natural way, with the appropriate grading shifts to fit the Lemma.  This specifies the gradings, the $E$-function, and the derived $\langle k \rangle$-orbifold charts on the subcategory. It remains to specify the morphisms from $\CC(H_4, J)$ to $\CC(H_1, J)$ and build a compatible system of derived $\langle k \rangle$-orbifold charts for the whole category.

The morphisms from $\CC(H^h; J)(\tilde{x}_, \tilde{x}_+)$ from $\tilde{x}_- \in \CC(H_4, J)$ to $\tilde{x}_+\in \CC(H_1, J)$ are defined to be the Gromov-Floer compactification of the space of pairs of a morse trajectory $\tau$ in the interior of the square that is the domain $H^h$ together with a solution of the continuation map equation for the continuation datum $H_{s,\tau}$, up to simultaneous $\R$-translation of both data. This space is naturally a $[2 (\mathcal{A}(\tilde{x}_-) - 2 \mathcal{A}(\tilde{x}_+ 2K')-3]$-space: it can break into at most two continuation map trajectories, and we use the extra element of the labeling set to describe where the map is broken into a pair of compactified continuation map trajectories. 
The moduli spaces $\CC(H^h; J)(\tilde{x}_, \tilde{x}_+)$ admit continuous maps $\alpha$ the moduli space of Morse trajectories in the square which specify which Morse trajectory defines the continuation datum that the corresponding solutions to the continuation map equation solve.

To equip these spaces with global Kuranishi charts we need to introduce another moduli space of domains akin to $\mathcal{F}^{cont}$; we will sketch the construction in outline. We let $\mathcal{F}^{cont, b}(\CP^\ell, d)$ be the space of elements of $\mathcal{F}^{cont}(\CP^\ell, d)$ with \emph{two} additional marked points $\hat{p}^1$, $\hat{p}^0$ on the geodesics, with $\hat{p}^0$ closer to $p_+$ than $\hat{p}^1$, and such that if $\hat{p}^0$ and $\hat{p}^1$ lie on the same component $S^2_{j_c}$, then their images under any biholomorphism $S^2_{j_c} \simeq Z$ lie at least $10$ apart. As before, these moduli spaces can be related via generalized inductions and restrictions, and have a natural notion of a compatible set of perturbation data. A pair of maps from $\mathcal{F}^{cont, b}(\CP^\ell, d)$ to the moduli space of Morse trajectories (a interval) in each of the two simplices comprising the homotopy datum $H^c$, which map to the broken Morse trajectory exactly when $\hat{p}^i$ are on different components, and which map to the Morse trajectory which bisects the square exactly when the $\hat{p}^i$ are on the same component and like a distance $10$ apart on a biholomorphism to the cylinder. These maps are required to satisfy the property that near the locus where $\hat{p}^0$ and $\hat{p}^1$ are on distinct components, if we choose biholomorphisms to the cylinder on these components for which these points get mapped to $s=0$, then as we glue these curves to a curve with a single component, the corresponding biholomorphisms to the cylidner map these points to those with $s$ values agreeing with those where the gluing map for the corresponding  Morse trajectories maps the mid-points of the pair of broken curves.  For each of these maps, one defines thickenings $T_{bord, d, \pm}^\epsilon(\tilde{x}_-, \tilde{x}_+)$ of the preimages of $\alpha$ in $\CC(H^h; J)(\tilde{x}_, \tilde{x}_+)$ of each half of the moduli space of Morse trajectories in the square, and the final thickening of $\CC(H^h; J)(\tilde{x}_, \tilde{x}_+)$ is simply a union of these two pieces. The thickenings consist of choices of a \emph{pre-continuation-homotopy trajectory} which as in Definition \ref{def:framed-contination-trajectory} are simply choices of an element $u_F \in \mathcal{F}^{cont, b}(\CP^\ell, d)$ together with a map $u$ from the universal curve over $u_F$ with the points $p_0, \ldots, p_{r_u}$ removed to $M$ with the usual Hamiltonian endpoint trajectories, such that the $\int \Omega^c_{u, \alpha(u_F)}$-degrees of the map agree with the degrees assigned to components of $u_F$, and such that if a component contains a pair of $\hat{p}^i$, we use the midpoint of the $\hat{p}^i$ along the geodesic to determine the biholomorphism to the cylinder and we solve the continuation map equation for $H_{s, \alpha(u_F)}$ on that component; and otherwise we solve the continuation map equations corresponding to each piece of the broken morse trajectory on the corresponding components.

Using the methods of Sections 3 and 4, one establishes as in Proposition \ref{prop:charts-for-continuation-category}:
\begin{proposition}
\label{prop:charts-for-continuation-homotopy}
There exist compatible choices of perturbation data such that the linearizations of the defining equations of $T_{bord, d, \pm}^\epsilon(\tilde{x}_-, \tilde{x}_+)$  are surjective Fredholm operators on all solutions. The thickenings $T_{bord, d, \pm}^\epsilon(\tilde{x}-, \tilde{x}_+)$ are topological $PU(d+1)-\langle \langle 2 (\mathcal{A}(\tilde{x}_-) - 2 \mathcal{A}(\tilde{x}_+ 2K')-3\rangle$ manifolds, and a stabilization of collars of these charts can be smoothed such that the resulting derived $\langle k \rangle$-orbifold charts give a compatible system of normally complex derived $\langle k \rangle$-orbifold charts for $\CC(H^h; J)$.
\end{proposition}

We now note that the choices used to define $c_{H^1_s}, c{H^0_s},$, and $c{H^{10, \epsilon}_s}$ give us compatible systems of perturbations $\sigma'_0$ for $\CC'_{cont, bord}|_0$ and $\sigma'_1$ $\CC'_{cont, bord}|_1$, which satisfy $|\CC'_{cont, bord}|_0(\sigma'_0)| = Cone(c_{H^1_s} c_{H^0_s})$ and $|\CC'_{cont, bord}|_1(\sigma'_1)| = Cone(c{H^{10, \epsilon}_s})$.  We thus have a homotopy equivalence of cones $Cone(c_{H^1_s} c_{H^0_s}) \sim Cone(c{H^{10, \epsilon}_s})$. Because the perturbation data restricted to $\CC'_{cont, bord}|_0(\widetilde{H}_i)$ agrees with that for $\CC'_{cont, bord}|_1(\widetilde{H}_i)$ for $i=0,2$ respecitvely, this homotopy equivalence of cones is in fact induced by a homotopy equivalence of maps. $c_{H^1_s} c_{H^0_s} \sim c{H^{10, \epsilon}_s}$.
\end{proof}

\paragraph{The identity continuation datum}
\begin{proposition}
\label{prop:constant-continuation-datum}
Choose $H_s = H$ independent of $S$, i.e. let the continuation datum be a constant map from $\Delta^1$. Then $c_{H_s}: |\mathcal{C}(H, J)(\sigma')| \to |\mathcal{C}(H, J)(\sigma')|$ is a quasi-isomorphism over $\Lambda_0$.
\end{proposition}
\begin{proof}
The constant continuation trajectory is regular. Thus, we can choose the system of perturbations used to define $c_{H_s}$ such that we preserve these constant perturbation data. For this perturbation datum, we moreover have that $c$ is a map of $\Lambda_0$ modules, since $\mathcal{C}_{cont}$ is $E$-proper already, because the last term in \eqref{} is zero. Analyzing the construction of $c_{H_s}$ from Lemma \ref{lemma:cone-lemma}, we see that we can write $c_{H_s}$ as $1 + O(T^{\epsilon})$ for some $\epsilon>0$. By the Nakayama lemma, this implies that $c_{H_s}$ is a quasi-isomorphism. 
\end{proof}

We have proven:
\begin{proposition}
The quasi-isomorphism type of $|\mathcal{C}(H, J)(\sigma')|$ over $\Lambda^{univ}$ is independent of $H$ and $J$, as well as all auxiliary data used to define it. 
\end{proposition}

\subsection{Moduli spaces for PSS map}
\label{sec:pss-moduli-spaces}
Finally, we show that if $H_1$ is a small time-dependent perturbation of a $C^2$-small  \emph{strictly positive} Morse function $H$ with distinct critical values, then $CF(H_1, J)$ is quasi-isomorphic to the Morse complex of $H$. To do this we will, following \cite{PSS}, define a map $PSS: CM(H) \to CF(H_1, J_1)$ which is an isomorphism after setting the Novikov parameter to zero, and is thus a quasi-isomorphism of complexes. 

For sufficiently small perturbations of $H$, we have a natural bijection $Crit(H) \simeq Fix(\phi_{H_1}^1)$. We denote this bijection by writing that $x' \in Crit(H)$ corresponds to $x \in Fix(\phi_{H_1}^1)$. Thinking of a critical point as a constant loop, we have a bijection between cappings 
\begin{equation}
\begin{gathered}
    \widetilde{Crit}(H) = \widetilde{Fix}(\phi^1_H) \simeq \widetilde{Fix}(\phi^1_{H_1}) \\
    \tilde{x}' \xleftrightarrow{} \tilde{x}. 
    \end{gathered}
\end{equation}

For convenience, we will also choose $H_1$ such that $\mathcal{A}_\omega(\tilde{x}') = \mathcal{A}_\omega(\tilde{x})$. 

To define the domain of the perturbation datum we must define yet another moduli space of McLean-stabilized domains. Luckily, in this case the combinatorics are essentially identical to the combinatorics of the differential.  We recall that $D_{k}(\CP^{\ell}, d) \subset \overline{\M}_{0, k}(\CP^\ell, d)$ consists of the stable maps which have linearly nondegenrate image. Let $\mathcal{F}^{PSS, R}(\CP^{\ell}, d) \subset D_{3}(\CP^{\ell}, d)$ consist those maps such that, writing $(p_-, p_+, p)$ for the three marked points on $u \in D_{3}(\CP^{\ell}, d)$, we have that $(p, p_-)$ lie on the same component. As before, we can write $S_1, \ldots, S_{r_u}$ for the components of the domain of $u$ where $S_1 \ni p_- = p_1$, $S_j$ is joined to $S_{j+1}$ along $p_j$, and $S_{r_u} \ni p_{r_u} = p_+$. We can think of an element $u \in \mathcal{F}^{PSS, R}(\CP^{\ell, d})$ as a stable map where $S^1$ is labeled by a geodesic from $p_-$ to $p_1$ through $p$, and is marked by $p$; thus, $S_1 \setminus \{p_-, p\}$ is canonically biholomorphic to $Z$.  

We define 
 \begin{equation}
 \mathcal{F}^{PSS}(\CP^\ell, d) \subset \bigcup_{d_1+d_2 = d} \mathcal{F}^{PSS_R}(\CP^\ell, d_1) \times_{\CP^\ell} \mathcal{D}_{2}(\CP^\ell, d_2)
 \end{equation}
 to consist of those pairs $(C, D)$ for which $p_+ \in C$ is joined to $p_- \in D$, and such that the image of the joined map is not contained in any $(d-1)$-dimensional linear subspace.  
 
 This space is a $PU(\ell+1)$-$\langle d-1 \rangle$-manifold in the same way that $\mathcal{F}(\CP^\ell, d)$ is; there are no extra corners introduced by the $\mathcal{D}_{2}(\CP^\ell, d_2)$ factors, which are manifolds. In particular, we have that  
 \begin{equation}
 	\label{eq:define-pss-stratification}
 	 \mathcal{F}^{PSS}(\CP^d, d)(S(d_1, \ldots, d_r))\subset 
 	\mathcal{F}(\CP^{d}, d_1)\times_{\CP^d} \ldots \times_{\CP^d} \mathcal{F}(\CP^{d}, d_{r-1})\times_{\CP^d} \mathcal{F}^{PSS}(\CP^{d}, d_r)
 \end{equation}
 is related to the spaces $\mathcal{F}(\CP^{d_1}, d_1), \ldots, \mathcal{F}(\CP^{d_{r-1}}, d_{r-1}), \mathcal{F}^{PSS}(\CP^{d_r}, d_r)$ by restrictions, products, equivalences, and stabilizations, as in Proposition \ref{prop:induction-relation-between-kuranishi-charts},whenever $d_1 + \ldots + d_r = d$. 
 
 A perturbation datum over a substratum of $\mathcal{F}^{PSS}(\CP^\ell, d)$ is defined exactly as in Definition \ref{def:perturbation-data}. Formula \ref{eq:associativity-of-induction} regarding compatibility of the inductions of perturbation data along decompositions \eqref{eq:define-pss-stratification} continues to hold, making ``compatibility of perturbation data'' adapt to this setting as in Definition \ref{def:compatible-choice-of-perturbation-data}, with $\mathcal{F}$ replaced by $\mathcal{F}^{PSS}$. Assume that now we have fixed a compatible choice of perturbation data $\Lambda^{PSS}_d$ over $\mathcal{F}^{PSS}(\CP^d, d)$ for all $d$, extending a compatible set of chosen perturbation data $\Lambda_d$ over $\mathcal{F}(\CP^d, d)$. 
 
 We choose a \emph{PSS datum}, namely a smooth map $H^{PSS}_s: \R_s \to \mathcal{H}(M)$ such that $H^{PSS}_s = H_1$ for $s \leq -1$, $\partial_s H^{PSS}(x) \leq 0$ for all $s, x$; and $H^{PSS}_s = 0$ for $s > 1$. This is possible since $H$ is a strictly positive Morse function. With this choice of $H^{PSS}_s$, the estimate \eqref{eq:energy-indentity-continuation-map} shows that for a map 
 \begin{equation}
 \label{eq:pss-trajectory}
 	u: Z \to M; \partial_s u + J (\partial_t u - X_{H^{PSS}_s}) = 0; \lim_{s \to - \infty}u(s,t) = x_-(t); \lim_{s \to + \infty} u(s,t) = u_{+\infty} \in M; \tilde{x} = [u]
 \end{equation}
(where the last condition is simply that $u$ is that $u$ viewed as a map from $S^2\setminus \{p_-\}$ is in the homology class of the capping $\tilde{x}_-$ of the Hamiltonian orbit $x_-$) we have that 
\begin{equation}
    \label{eq:action-positivity}
    \mathcal{A}_\Omega(\tilde{x_-}) \geq \epsilon > 0
\end{equation} for some fixed $\epsilon$ independent of $\tilde{x}_-$. As in Section , we use a process identical to that of Section \ref{sec:integral-forms-for-continuation-map} to choose an $(\tilde{\Omega}, N_1)$ and as well as $\epsilon_x$ for $x \in Fix(\phi_{H})$; and then we choose a $\bar{H}$ as in section \ref{sec:integral-two-form} and a $\bar{H}^{PSS}_s$ which has conditions as in \ref{eq:bar-h} for $s \to - \infty$, and still satisfies $\partial_s \bar{H}^{PSS}_s(x) \geq 0$ and $\bar{H}^{PSS}_s = 0$ for $s \geq \frac{1}{2}$. We then have the form $\Omega_u^{PSS}$ defined identically to \eqref{eq:floer-stabilizing-form} with $\bar{H}_s$ replaced by $\bar{PSS}_s$. Thus we have 
\begin{lemma}
    If $u$ is a solution to \eqref{eq:pss-trajectory} then $\bar{A}(\tilde{x}-) + K' > 0$ for some fixed $K'$; and the set of $u$ for which $\bar{A}(\tilde{x}_-) = \int_Z \Omega_u^{PSS} < C$ is compact for each $C$. 
\end{lemma}

 A \emph{framed bubbled pre-PSS trajectory} is defined as in Definition \ref{def:framed-contination-trajectory}, with the same notion of equivalence,  but with different markings: now we only have marked geodesic segments from $p_{j-1}$ to $p_j$ only up to $j \leq J$ for some $J \leq r_u$, with the geodesic segment from $p_{J-1}$ to $p_J$ carrying the extra marked point; and the map is required to limit to fixed point in $M$ at $p_{r_u}$.  As before we write $\psi_{u,J}$ for the canonical biholomorphism $S_J^2 \setminus \{p_{J-1}, p_J\} \simeq Z.$  
 
  We can use $\Omega_u$ chosen based on $H_1$ as in Section \ref{sec:integral-two-form} to make sense of this definition, and the arguments of that section continue to show that $\int_{\Sigma_\alpha} \Omega_u$ is integral on any component of a framed bubbled pre-PSS trajectory. 
 
  We can now define the thickenings $T_{\Lambda_d^{PSS}}$ to consist of pairs $((u, \Sigma, F), \eta)$ where $(u, \Sigma, F)$ is a framed bubbled pre-PSS trajectory with $u_F \in \mathcal{F}^{PSS}(\CP^d, d)$, and $\eta \in (V^{PSS}_d)_{u_F}$ which satisfies 
 \begin{equation}
 	\begin{gathered}
 	\bar{\partial}_Fu|_{S_j} + \lambda^{PSS}_d(\eta) \circ di_F = 0; \\
 	\partial_s \tilde{u} + J \partial_t(\tilde{u} - X_{H_s}) + \lambda^{PSS}_d(\eta) \circ di_F \circ d\psi_{u, J}= 0 , \text{ where } \tilde{u} =  u|_{S^j} \circ \psi_{u, J}; \\
 	\bar{\partial}u|_{S^2_\alpha} + \lambda^{PSS}_d(\eta)\circ di_F= 0 \text{ for all other components } S^2_\alpha \neq S^2_1, \ldots, S^2_J. 
 	\end{gathered}
 \end{equation}
 	We then define $\mathcal{K}_{\Lambda^{PSS}_d} =(PU(d+1), T_{\Lambda^{PSS}_d}, \mathcal{V}_{\Lambda^{PSS}_d} \oplus i \mathfrak{pu}_{d+1}, \sigma_{\Lambda^{PSS}_d} \oplus \sigma_H)$ defined exactly as in Definition \ref{def:thickening}. We write $\mathcal{T}^\epsilon_{\Lambda^{PSS}_d}$ to to be the sub-chart of $\mathcal{T}_{\Lambda^{PSS}_d}$ for which $\|\eta\| < \epsilon$. We have that 
 	    $\overline{\M}^{PSS}(d)$, the quotient of the zero section of the Kuranishi section of $\mathcal{K}_{\Lambda^{PSS}_d}$ by the $PU(d+1)$-action, is the set of bubbled PSS trajectories in $M$ (of $\Omega_u$-degree $d$). We write 
 	    \begin{equation}
 	        \cM^{PSS}(\tilde{x}_-) = \{ u \in \bigcup_d \cM^{PSS}(\overline{\mathcal{A}}(\tilde{x}_-)) | \lim_{z \to p_-} u(z) = x_-, u \text{ defines the capping } \tilde{x}_-\}.
 	    \end{equation}

 As before, we have that, for sufficiently small $\epsilon >0$, the $PU(d+1)$ action on $\mathcal{T}^\epsilon_{\Lambda}$ has finite stabilizers, and that the induced topology on $\cM^{PSS}(d)$ agrees with the Gromov-Floer topology. The argument of Lemma \ref{lemma:surjective-perturbation-data-exist-2} adapts to show that one can choose compatible perturbation data $\Lambda^{PSS}_d$ such that the extended linearized operators are always surjective on PSS solutions, which implies that $T_{\Lambda^{PSS}_d}^\epsilon$ is a $PU(d+1)-\langle d-1\rangle$-manifold with projection $\Pi$ to $\mathcal{F}^{PSS}(\CP^d, d)$ a fiberwise-smooth $C^1_{loc} PU(d+1)$ bundle.

 \paragraph{Morse moduli spaces}
 We now need to couple the moduli spaces defined above to Morse-theoretic moduli spaces to define the PSS map. Given $x'_\pm \in Crit(H)$, write $\mathcal{M}^{M}(x'_-, x'_+)$ for the set of Morse gradient trajectories
 \begin{equation}
     \mathcal{M}^M(x_-, x_+) = \{ \gamma: \R \to M | \dot{\gamma} = -\Grad f, \lim_{s \to \pm \infty} \gamma(s) = x_\pm\}/\R
 \end{equation}
 and write $\cM^M(x'_-, x'_+)$ for its compactification by broken morse trajectories. Furthermore, let
 $\cM^{M-out}(x_+)$ be the natural compactification of the moduli space of Morse trajectories starting from an arbitrary point of $M$ and flowing into $x_+$:
 \begin{equation}
 \begin{gathered}
     \mathcal{M}^{M-out}(x_+) = \{ \gamma: [0, \infty) \to M | \dot{\gamma} = -\Grad f, \lim_{s \to \infty} \gamma(s) = x_+\} \\
     \cM^{M-out}(x_+) = \bigcup_{\substack{k \geq 0 \\ x_0, \ldots, x_k = x_+ \in Crit(f) \\ x_i \neq x_j \text{ for } i \neq j}} \mathcal{M}^{M-out}(x_0) \times \mathcal{M}^M(x_0, x_1) \times \ldots \times \mathcal{M}^M(x_0, x_k)\\
     \end{gathered}
 \end{equation}
 We note that there is a bound on the number of breakings an element of $\cM^{M}(x_-, x_+)$ or of $\cM^{M-out}(x_+)$, namely $ind(x_-) - ind(x_+)-1$ for $\cM^{M}(x_-, x_+)$ and $2n - ind(x_+)$ for $\cM^{M-out}(x_+)$. In particular, for generic Morse-Smale $H$, we have that $\cM^M(x_-, x_+)$ is a $\langle ind(x_-) - ind(x_+)-1 \rangle$ manifold, while
 $\cM^{M}(x_-, x_+)$ is a $\langle 2n - ind(x_+)\rangle$ manifold. We have an evaluation map $ev: \cM^{M-out}(x_+) \to M$ sending $(\gamma, \ldots )$ to $\gamma(0)$. 
 	
We then define
\begin{equation}
    T^{\epsilon}_{\Lambda_d^{PSS}, M}(x'_+) = T^\epsilon_{\Lambda_d^{PSS}} \times_M \cM^{M-out}(x'_+)
\end{equation}
where the fiber product is over the evaluation map $ev \cM^{M-out}(x'_+) \to M$ and the evaluation map $u \mapsto u(p_+)$ for $((u, \Sigma, F), \eta) \in T^\epsilon_{\Lambda^{PSS}_d}$. We define $\mathcal{T}^\epsilon_{\Lambda_d^{PSS}, M}(x_+) = (PU(d+1), T^{\epsilon}_{\Lambda_d^{PSS}, M}(x'_+), \pi_1^*(\mathcal{V}_{\Lambda^{PSS}_d} \oplus i \mathfrak{pu}_{d+1}), (\sigma_{\Lambda^{PSS}_d} \oplus \sigma_H) \circ \pi_1)$ where $\pi_1$ is projection onto $T^\epsilon_{\Lambda_d^{PSS}}$ and the $PU(d+1)$ action is only on that factor. We have that $T^{\epsilon}_{\Lambda_d^{PSS}, M}(x_+)$ is a $\langle d-1 + (2n-x_+-1)\rangle$-manifold. 

Let us write the quotient of the zero section of the Kuranishi section by the group action on $T^{\epsilon}_{\Lambda_d^{PSS}, M}(x_+)$ as $\cM^{PSS, M}(x_+)$. Writing $\tilde{x}'_+$ for the constant cap on $x'_+$,  let $\cM^{PSS, M}(\tilde{x}_-, \tilde{x}'_+)$ be the subset of $\cM^{PSS, M}(x_+)$ for which the projection to the first factor lies in $\cM^{PSS}(\tilde{x}_-)$. We call elements of $\cM^{PSS, M}(\tilde{x}_-, \tilde{x}_+)$ \emph{PSS trajectories from $\tilde{x}_-$ to $\tilde{x}_+$.}

\subsection{Definition of the PSS map}

\paragraph{Morse Flow category} We first define a proper $E$-proper graded $\Pi$-equivariant flow category $\mathcal{C}^{Morse}$ corresponding to Morse cohomology. The objects are $Ob(M) = \Pi Crit(H)$. We have that $\mathcal{C}^{Morse}(\pi\tilde{x}'_-, \pi\tilde{x}'_+) = \cM^(x_-, x_+)$ for $\pi \in \Pi$ and $\tilde{x}'_\pm $ the constant caps at $x'_\pm \in Crit(H)$. Otherwise we set $\mathcal{C}^{Morse}(\tilde{x}'_-, \tilde{x}'_+) = \emptyset$. We use the usual $\mu$ and energy function $E$, thinking of the capped critical points as constant Hamiltonian trajectories. We set $\mathcal{S}_{\mathcal{C}^{Morse}(\pi\tilde{x}'_-, \pi\tilde{x}'_+)} = [ind(x'_-) - ind(x'_+)-1]$ and otherwise we set $\mathcal{S}_{\mathcal{C}^{Morse}(\tilde{x}'_-, \tilde{x}'_+)} = \bullet$. Composition is given by gluing Morse trajectories and maps of sets directly analogous to \eqref{eq:set-map-for-flow-category}. The recapping action makes $\mathcal{C}^{Morse}$ $\Pi$-equivariant. 

By choosing a compatible system of framings for the Morse moduli spaces \cite{CJS} we can think of $\cM(x_-, x_+)$ as normally complex derived $\langle k \rangle$-orbifolds in a trivial way, as they are smooth manifolds \cite{burghelea-haller} \cite{wehrheim-corners}; thus we have a compatible system of derived $\langle k \rangle$-orbifold charts $\mathcal{C}'^{Morse}$ which  $\mathcal{C}^{Morse}$. The process of choosing a compatible system of perturbations $\sigma'$ for $\mathcal{C}^{Morse}$ is trivial as such a choice is unique, and we have that $|\mathcal{C}^{Morse}(\sigma')|_{\Lambda^\Pi_\Z}$ is the Morse cochain complex over $\Lambda^\Pi_\Z$. 

\paragraph{PSS Flow Category}
We now proceed to define a filtered flow category $\mathcal{C}^{PSS}$ which will contain $\mathcal{C}_{Floer}(H_1, J)$ and $\mathcal{C}^{Morse}[-1]$ as full subcategories satisfying the conditions of Lemma \ref{lemma:cone-lemma}. This specifies the gradings, the functions $E$ (given by the usual symplectic actions), and the $\Pi$ actions.  It suffices to define $\mathcal{C}^{PSS}(\tilde{x}_-, \tilde{x}'_+)$ for all pairs $\tilde{x}_- \in \widetilde{Fix}(\phi_{H_1}^1)$, $\tilde{x}_+ \in \widetilde{Fix}(\phi_H^1)$. We set $\mathcal{C}^{PSS}(\pi \tilde{x}_-, \pi\tilde{x}'_+) = \cM(\tilde{x}_-, \tilde{x}'_+)$ for $\tilde{x}'_+$ the constant capping at $x'_+$ and $\pi \in \Pi$. We set $\mathcal{S}_{\mathcal{C}^{PSS}(\pi \tilde{x}_-, \pi\tilde{x}'_+)} = [\mathcal{A}(\tilde{x}_-)-1+(2n - ind x_+)-1]$. Otherwise we set $\mathcal{C}^{PSS}( \tilde{x}_-, \tilde{x}'_+) = \emptyset$. The composition operation glues Morse trajectories or Floer trajectories to PSS trajectories in the usual way, while the composition operation on sets is as in \eqref{eq:set-map-for-flow-category}, with $\mathcal{C}_{cont}$ replaced by $\mathcal{C}^{PSS}$. This makes $\CC^{PSS}$ into proper,$E$-proper $\Pi$-equivariant flow category.

The methods of Sections 3 and 4 adapt to this setting to show that 
\begin{proposition}
The flow category $\mathcal{C}^{PSS}$ admits a compatible system of derived $\langle k \rangle$-orbifold charts arising from the topological Kuranishi charts $\mathcal{K}_{\Lambda^{PSS}_d, M}$. 
\end{proposition}

Applying Lemma \ref{lemma:cone-lemma}, we produce a map 
\begin{equation}
	\label{eq:pss-chain-map}
	PSS: |\mathcal{C}^{Morse}(\sigma')|_{\Lambda^{\Pi_E}_\Z}  \to |\mathcal{C}(H_1, J)(\sigma')|_{\Lambda^{\Pi_E}_\Z}. 
\end{equation}

\paragraph{Completion of existence of PSS map}
We now argue almost identically to \ref{prop:constant-continuation-datum}. We note first that the flow category $\mathcal{C}^{PSS}$ is $E$-positive by \eqref{eq:action-positivity}. 
Thus we can tensor $PSS$ up to a map
\begin{equation}
    \label{eq:pss-chain-map-univ}
	PSS^{univ}: |\mathcal{C}^{Morse}(\sigma')|_{\Lambda^{univ}_\Z}  \to |\mathcal{C}(H_1, J)(\sigma')|_{\Lambda^{univ}_\Z}. 
\end{equation}

The $E$-positivity of $\CC^{PSS}$ implies that $PSS^{univ}$ is induced from a corresponding map $PSS^{univ}_0$ of $\Lambda^{univ}_{0,\Z}$ modules. Now, we claim that for sufficiently small perturbations $H_1$ of $H$ (which we choose to not change actions of corresponding fixed points), there is a unique nondegenerate PSS trajectory from $\tilde{x}$ to $\tilde{x}'$, where the caps are the constant caps at $x$ and $x'$. This holds because this would be true if $H = H_1$ and $H^{PSS}_s$ was constant then the \emph{constant} PSS trajectory at the constant orbit $x = x'$ would do this. But this solution is nondegenerate and so persists upon a sufficiently small perturbation.

We have that $E(\tilde{x}) = E(\tilde{x}')$, and thus (c.f. Lemma \ref{lemma:cone-lemma}) we can write $PSS^{univ}_0$ as $1 + T^{\epsilon}PSS'$ for some map $PSS'$ of $\Lambda^{\Pi_E}_\Z$ modules. By the Nakayama lemma this proves

\begin{proposition}
\label{prop:pss-is-qis}
	The map $PSS$ of \eqref{eq:pss-chain-map} is a quasi-isomorphism.
\end{proposition}

\subsection{Proof of Theorem \ref{thm:integral-arnold}}
We now combine the methods of this section to prove the main theorem.
\begin{proof}[Proof of Theorem \ref{thm:integral-arnold}]
First, we replace the main Hamiltonian $H$ of the theorem with a new Hamiltonian, which we will also denote $H$, so that the number of fixed points does not change but all of the corresponding orbits are nonconstant. We choose a $C^2$-small morse function $H_0$ and a small perturbation $H'_0$ of $H_0$ as in Section \ref{sec:pss-moduli-spaces} which would allow us to define the $PSS$ map. We fix a background almost complex structure $J$ We consider the following pair of diagrams
\begin{equation}
\begin{tikzcd}
H'_0 \ar[r] \ar[d] & H'_0 \ar[d] & H \ar[r] \ar[d] & H \ar[d] \\
H \ar[r] & H'_0 & H'_0 \ar[r] & H
\end{tikzcd}
\end{equation}
We choose continuation homotopy data filling out each of these pairs of diagrams, such that the continuation data associated to any arrow from a Hamiltonian to itself is the identity continuation datum; as well as a PSS datum to compare Morse theory of $H_0$ and Floer theory of $H'_0$. We can $(\tilde{\Omega}, N_1)$ so that we can simultaneously define all of these maps, and then choose $\bar{H}_{s,\tau}$ and $\bar{H}^{PSS}$ as in the corressponding sections. Applying the constructions of the previous sections defines maps on Floer complexes on the Hamiltonians as in the diagrams above, but with the arrows reversed, which are homotopy commutative. By Proposition \ref{prop:constant-continuation-datum} the nontrivial continuation maps when composed with quasi-isomorphisms become homotopic to quasi-isomorphisms, so they themselves are quasi-isomorphisms. Thus the Floer complex of $H$ is quasi-isomorphic to the Floer complex of $H'_0$, which is quasi-isomorphic to the Morse complex of $H_0$ by Proposition \ref{prop:pss-is-qis}. We concluse by Lemma \ref{lemma:qis-lemma}. 
\end{proof}

\section{Bifurcation analysis and bordisms of flow categories}
We note that in the above proof of Theorem \ref{thm:integral-arnold}, we have never established the independence of the Floer complex of the choice of systems of perturbations, nor of the choices of integralization data or the choices of $J$ or the perturbation data and stabilizing vector bundles. Showing independence of $J$ can be done as in the proof above, using continuation maps. While it is possible to show independence of the perturbation data or or the integralization data using continuation-like ideas, below we describe a different method to show independence of various data of Floer theoretic maps. Namely, we describe the idea of a \emph{bordism of flow categories} and a \emph{bordism of derived $\langle k \rangle$-orbifold charts}, and outline how this can show independence of floer complexes of the choices of perturbations or of the choices of integralization data. In essence, adapting the \emph{bifurcation method} for studying Floer theoretic operations to the setting of flow-categories can lead to more elegant proofs. We hope this will be helpful to symplectic topologists trying to explore the new invariants made possible by global charts in Hamiltonian floer homology.
\begin{definition}
For each of $\mathcal{D} = Top(Fin'), DerOrb', DerOrb'_{\C}$ there is a corresponding monoidal category 
$\mathcal{D}_{bord} = Top(Fin')_{bord}, DerOrb'_{bord}$, $DerOrb'_{\C, bord}$ defined as follows. 

The category $Top(Fin')_{bord}$ has objects given by $(S, X) \in Top(Fin')$ equipped with continuous $p: X \to [0,1]$ such that $p^{-1}(0) \cup p^{-1}(1) = X(S\setminus \bar{p})$ for some distinguished $\bar{p} \in S$. Morphisms from $(S, X_S)$ to $(T, X_T)$ are those in $Top(Fin')$ but with the map $\bar{f}: X_S \to X_T$ required to commute with $p$ and the map $f: S \to T$ required to send the distinguished element $\bar{p} \in S$ to the distinguished element $\bar{p} \in T$. The monoidal structure is given b $(S, X_S) \times(T, X_T) = (S \cup_{\bar{p}} T, X_S \times_p X_T)$ where $S \cup_{\bar{p}} T$ means we have identified the two copies of the distinguished element in the disjoint union, and the stratification on $X_S \times_p X_T$ is $X_S \times_p X_T(R) = X_S(R \cap S) \times X_T(R \cap T)$ for $R \subset S \cup_{\bar p} T$. 

The categories $DerOrb'_{bord}$ and $DerOrb'_{\C, bord}$ have objects pairs $(S, X)$ which are elements of $DerOrb'$ or $DerOrb'_\C$, respectively, with $\bar{p} \in S$ a distinguished element, and, writing $X = (T, V, \sigma)$, with an additional submersion $p$ from the $\langle S \rangle$-orbifold $T$ to the $\langle 1 \rangle$-orbifold $[0,1]$ for which $S_f = S \setminus \{\bar{p}\}$, and the map of sets $S \setminus S_f = \{\bar{p}\} \to [1]$ is the unique such map. In particular, this implies that $T(S \setminus \bar{p}) = p^{-1}(0) \cup p^{-1}(1)$. Maps are as in $DerOrb'$ or $DerOrb'_\C$, respectively, but the maps on sets must preserve distinguished elements and the maps on underlying $\langle k \rangle$ orbifolds must commute with the maps $p$. The monoidal structure is given by 
\begin{equation}
    (S_1, (T_{S_1}, V_{S_1}, \sigma_{S_1})) \times (S_2, (T_{S_2}, V_{S_2}, \sigma_{S_2})) = 
    (S_1 \cup_{\bar{p}} S_2, 
    (T_{S_1} \times_p T_{S_2}, 
    V_{S_1} \times V_{S_2}|_{T_{S_1} \times_p T_{S_2}}, 
    \sigma_{S_1} \times \sigma_{S_2}|_{T_{S_1} \times_p T_{S_2}})). 
\end{equation}

There are natural monoidal forgetful functors $DerOrb'_{\C, bord} \to DerOrb'_{bord} \to Top(Fin')_{bord}$. Moreover, for each of these categories $\mathcal{D}$, there are a pair of monoidal restriction functors $\mathcal{D}_{bord} \to \mathcal{D}$, denoted $X \mapsto X|_{0}$ and $X \mapsto X|_{1}$, which restict all the data to $p^{-1}(0)$ or to $p^{-1}(1)$, respectively. These restiction functors commute with the forgetful functors. 
\end{definition}

\begin{definition}
Suppose that $\CC$ is graded, proper, and $\Pi$-equivariant. 

A \emph{bordism of flow categories} is a category $\CC_{bord}$ enriched in $Top(Fin')_{bord}$. We define $|S_{\CC_{bord}(\tilde{x}, \tilde{y})}| = r(\tilde{x}, \tilde{y})$. The notions of $\Pi$-equivariance and properness, and gradings carry over. 

A bordism of derived $\langle k \rangle$-orbifold charts over a graded proper $\Pi$-equivariant bordism of flow categories $\CC_{bord}$ is a category $\CC'_{bord}$ enriched in $DerOrb'_{\C, bord}$ equipped with a free action of $\Pi$ such that applying the forget functor $DerOrb'_{\C, bord} \to Top(Fin')_{bord}$ recovers $\CC_{bord}$, and such that $vdim(\CC'_{bord}(\tilde{x}, \tilde{y})) = \mu(\tilde{x}) - \mu(\tilde{y})$ for each $\tilde{x}, \tilde{y} \in Ob(\CC_{bord})$ with $\CC_{bord}(\tilde{x}, \tilde{y}) \neq \emptyset$ and $\tilde{x} \neq \tilde{y}$. 

Note that $\CC_{bord}|_0$, $\CC_{bord}|_1$ are both flow categories, and $\CC'_{bord}|_0$, $\CC'_{bord}|_1$ are compatible systems of derived $\langle k \rangle$-orbifold charts for these flow categories, respectively. 

If $\Pi$ is a Novikov group, the notions of a $\Pi$-equivariant bordism of flow categories are modified as they are for flow categories, and the notions of $E$-positivity and $E$-properness extend to this case via the same definition.
\end{definition}

\begin{proposition}
\label{prop:novikov-bifurcation-continuation}
    Say $\Pi$ is a Novikov group. 
    
    Given a bordism of perturbation data $\CC'_{bord}$ over a $\Pi$-equivariant $E$-proper $E$-positive topological flow category $\CC$, and a pair of choices of compatible systems of perturbations $\sigma'_0$, $\sigma'_1$ for $\CC'_0 := \CC_{bord}|_0$ and $\CC'_1 := \CC_{bord}|_1$ respectively, there is a homotopy equivalence 
    \begin{equation}
    \label{eq:bordism-equivalence}
        \bar{f}: |\CC(\sigma'_0)|_{\Lambda^{\Pi_E}_\Z} \to |\CC(\sigma'_1)|_{\Lambda^{\Pi_E}_\Z} 
    \end{equation}
    and similarly for the complexes over $\Lambda^{\Pi^0_E}_\Z$. 
\end{proposition}

\begin{proof}
    This statement is directly analogous to the proof of the invariance of the Novikov complex computing the homology of a closed $1$-form on a manifold by analyzing a generic family of closed $1$-forms in the same cohomology class. We follow the strategy of \cite{hutchings-reidemeister-torsion}, which was implemented for several variants of Hamiltonian Floer homology in the semipositive setting by \cite{lee2005reidemeister-1, lee2005reidemeister-2}. Alternatively, \cite{FO3:book-vol12} uses a variant of this method to prove invariance of Floer homology under deformations of $J$ and of other auxiliary data chosen in that method, and \cite[Section~7.2.14]{FO3:book-vol12} contains several interesting comments about various implementations of this general idea. 
    
    Let us recall the finite-dimensional situation. Let $\alpha \in \Omega^1(M)$ be closed $1$-form and $V$ the vector field dual to $\alpha$ for some metric on $M$. Assume moreover that the zeros of $V$ are nondegenerate, as are its closed orbits, and the stable and unstable manifolds of all zeros of $V$ intersect transversely. Let $C_i$ be the set of zeros of $V$ of index $i$. Choosing a minimal abelian cover $pr: \tilde{M} \to M$ on which the pullback of $\alpha$ becomes exact, we have a Novikov group $(\Pi, E, \mu)$ where $\Pi$ is the covering group of $\tilde{M}$, $\mu$ is identically zero, and $E(\pi) = \int_\alpha [\gamma_\pi]$ where $\gamma_\pi: S^1 \to M$ is a loop representing a class in $H_1(M)$ which maps to $\pi$ under the surjection $H_1(M) \to \Pi$. We can then form the (cohomological) \emph{Novikov complex} $CN(\alpha)^\bullet$ of $\alpha$ \cite{hutchings-reidemeister-torsion} ; this is generated as a module over $\Lambda^{\Pi}_\Z$ 
     by the zeros of $V$, and the differential counts rigid flow lines of $V$ between its zeros. We write the set of zeros of $V$ as of index $i$ as $C_i$, and the set of lifts of these zeros to $\tilde{X}$ i.e. the set of index-$i$ zeros of $pr^*V$, as $\tilde{C}_i$. Given a one-parameter family of closed $1$-forms $\alpha_\tau$ with corresponding vector fields $V_\tau$, the following bifurcations can occur:
    \begin{enumerate}[(a)]
    	\item A degenerate flow line from $\tilde{p} \in \tilde{C}_i$ to $\tilde{q} \in \tilde{C}_{i-1}$.
    	\item A degenerate closed orbit;
    	\item A flow line from $\tilde{p} \in \tilde{C}_i$ to $\tilde{q}\in\tilde{C}_i$, $pr(\tilde{p}) \neq pr(\tilde{q})$;
    	\item A flow line from $\tilde{p}$ to $\pi \tilde{p}$ for some element $\pi$ in the Novikov group $\Pi$;
    	\item Birth or death of two critical points of two critical points at a degenerate critical point.
    \end{enumerate}
    
    The first two bifurcations do not change the Novikov complex. The last bifurcation adds a pair of canceling to the Novikov complex. The second and third bifurcations change the Novikov complex by a change of basis, in the sense that, writing $d^j, d'_j: CN^j \to CN^{j+1}$ for the differentials before and after the bifurcation, we have that $d^j = d'^j$ for $j \neq i, i-1$, and $d'_{i-1} = A d_{i-1}$, $d'_i = d_i A$, where $A: CN^i \to CN^i$ is some $\Lambda^\Pi_\Z$-linear isomorpism. 
    Concretely, (c) corresponds to changing the basis such that $\tilde{p}$ is replaced by $\tilde{p} \pm \tilde{q}$, and (d) corresponds changing the basis of the Novikov complex for which $\tilde{p}$ is multiplied by an invertible power series $(1 + h + \ldots)$. 
    
    We will see that there for bordisms of flow categories there are a corresponding set of bifurcations, and the analysis can be performed as in \cite{hutchings-reidemeister-torsion} because Kuranishi charts give finite-dimensional models of all phenomena involved.

    Note that $Ob(\CC'_{bord}) = Ob(\CC'_0) = Ob(\CC'_1)$, and that as $\Z/N_\Pi\Z$-graded modules the domain and the codomain of $\bar{f}$ agree. Since $\CC_{bord}$ is proper we can assume for the purposes of this proof that $\#(Ob(\CC)/\Pi) < \infty$. 
    
    As before, write $\CC'_{bord}(\tilde{x}, \tilde{y}) = (T_{bord}(\tilde{x}, \tilde{y}), V_{bord}(\tilde{x}, \tilde{y}), \sigma_{bord}(\tilde{x}, \tilde{y}))$.

    First, using an induction as in Proposition \ref{prop:transverse-perturbations-exist-inductively}, we modify the continuous functions $p: T_{bord}(\tilde{x}, \tilde{y})$ to functions $p': T_{bord}(\tilde{x}, \tilde{y})$ which are smooth, are submersions along $T_{bord}(\tilde{x}, \tilde{y})|_{i}$ for $i=0,1$, and still make $\CC'_{bord}$ into a bordism of derived $\langle k \rangle$-orbifold charts for $\CC_{bord}$. This is possible since we can take $p'$ to be linear functions of the distances to $T_{bord}(\tilde{x}, \tilde{y})_i$ for $i=0,1$ in open neighborhoods of these boundaries and inductively interpolate by smooth functions on the rest of the thickenings without introducing any additional $0$s or $1$s. We then inductively stabilize all charts by trivial complex vector bundles and extend $p'$ to smooth functions on the stabilized thickenings in order to arrange for the functions $p'$ to be submersions. 
    With this modification, we will have that for all $t \in [0,1]$, $\CC_{bord}|_t(\tilde{x}, \tilde{y})$ is a smooth derived $\langle k \rangle$-orbifold chart given by $p^{-1}(t) \subset T_{bord}(\tilde{x}, \tilde{y})$ (with the restrictions of the vector bundle and the Kuranishi section to this smooth sub-orbifold). These charts satisfy the needed compatibility conditions whenever they exist, and so for all   $t \in [0,1]$, we have that $\CC_{bord}|_t$ admits a system of derived $\langle k \rangle$-orbifold charts $\CC'_{bord}|_{t}$. 
    
    Now, we would like to apply a modification of  Proposition \ref{prop:transverse-perturbations-exist-inductively} to find a ``$\Pi$-equivariant compatible system of perturbations $\sigma'_{bord}$ for $\CC'_{bord}$'' such that $\sigma'_{bord}(\tilde{x}, \tilde{y})(i, x) = \sigma'_i(\tilde{x}, \tilde{y})(x)$ for $i = 0,1$ and $\tilde{x}, \tilde{y} \in Ob(\CC'_{bord})$ with $\CC'_{bord}(\tilde{x}, \tilde{y})$ (whenever $\CC'_i(\tilde{x}, \tilde{y}) \neq \emptyset$, which may or may not hold for any particular pair $\tilde{x}, \tilde{y}$). However, the method of Proposition \ref{prop:transverse-perturbations-exist-inductively} does not adapt to this setting because, while the product of a collection of FOP-transverse functions is transverse, the \emph{fiber product} of a collection of FOP-transverse functions can very easily not be transverse. However, if one of the FOP-transverse functions in the collection has no zeros, then the fiber product will have no zeros and so will be transverse. This allows for the induction of Proposition \ref{prop:transverse-perturbations-exist-inductively} to go through for all pairs $(\tilde{x}, \tilde{y})$ such that $\vdim T_{bord}(\tilde{x}, \tilde{y}) < 0$. 
    
    Write $Pairs_0$ for set of pars $(\tilde{x}, \tilde{y})$ with $vdim T_{bord}(\tilde{x}, \tilde{y}) = 1$ and $Pairs_{-1}$ for the set of such pairs with $vdim T_{bord}(\tilde{x}, \tilde{y}) = 0$. The first challenge is that for pairs $(\tilde{x}, \tilde{y}) \in Pairs_{-1}$, there can be codimension-1 boundary components of $T_{bord}(\tilde{x}, \tilde{y})$ which are fiber products $T_{bord}(\tilde{x}, \tilde{z}) \times_{[0,1]} T_{bord}(\tilde{x}, \tilde{z})$ where $(\tilde{x}, \tilde{z})$ and $(\tilde{z}, \tilde{y})$ are in $Pairs_{-1}$. As such, if $\sigma'_{bord}(\tilde{x}, \tilde{z})$ and $\sigma'_{bord}(\tilde{z}, \tilde{y})$ have (isolated, by inductive assumption) zeros with equal values of $p$, then $\sigma'_{bord}(\tilde{x}, \tilde{y})$ would be forced to have a zero in a boundary component, which is not acceptable by inductive assumption. Nonetheless, we can perform the induction for all pairs $(\tilde{x}, \tilde{y}) \in Pairs_{-1}$ while \emph{weakening} the transvesrality assumption to transversality of $\sigma'_{bord}$ on the interior of these moduli spaces, together with the condition that all bounady zeros arise as fiber products of transverse zeros on interiors of moduli spaces. Now, if we did not have the requirement that $\sigma'_{bord}$ is $\Pi$-equivariant, then by choosing $\sigma'_{bord}(\tilde{x}, \tilde{y})$ generically for $(\tilde{x}, \tilde{y}) \in Pairs_{-1}$, we would be able to arrange that no pair of interior zeros of these sections have equal values of $p$, and thus that the above fiber products are empty, and that the resulting sections are FOP transverse. However, when the sections $\sigma'_{bord}(\tilde{x}, \tilde{y})$ are required to be $\Pi$ equivariant, an interior zero of $\sigma'_{bord}(\tilde{x},\pi\tilde{x})$ or $\pi \in \Pi$ generates a sequence of boundary zeros for $\sigma'_{bord}(\tilde{x}, \pi^r \tilde{x})$ for all $r = 1, 2, \ldots.$, where we note that if $(x, \pi x) \in Pairs_{-1}$ then $(x, \pi^r x) \in Pairs_{-1}$. By requiring that all interior zeros of $\sigma'_{bord}(\tilde{x}, \tilde{y})$ have distinct values of $p$, we can arrange that these are the only violations of FOP-transversality of $\sigma'_{bord}(\tilde{x}, \tilde{y})$ for $(\tilde{x}, \tilde{y}) \in Pairs_{-1}$. 
    
    We then choose perturbations $\sigma'_{bord}(\tilde{x}, \tilde{y})$ for $(\tilde{x}, \tilde{y}) \in Pairs_{0}$, satisfying the induction conditions of Proposition \ref{prop:transverse-perturbations-exist-inductively}. If a boundary component of $T_{bord}(\tilde{x}, \tilde{y})$ for $(\tilde{x}, \tilde{y}) \in Pairs_0$ is a fiber product of thickenings none of which have negative virtual dimension, then they all must have virtual dimension zero except for one of them. By inductive assumption, we assume that the perturbation chosen for that virtual dimension $1$ component is FOP transverse, which is possible as that chart has no boundary components. We extend $\sigma'(\tilde{x}, \tilde{y})$ to a neighborhood of the boundary exactly as in Lemma \ref{lemma:fundamental-extension-lemma}, and then extend via an FOP-transverse section on the interior. As such, we have $\sigma'_{bord}(\tilde{x}, \tilde{y})^{-1}(0)$, for $(\tilde{x}, \tilde{y}) \in Pairs_0$, is a $1$-manifold with boundary, with boundary either lying on $T_{bord}(\tilde{x}, \tilde{y})_{i}$ for $i=0,1$, or arising from  a joining of an interior zero of $\sigma'_{bord}(\tilde{x}', \tilde{y}')$ for some $(\tilde{x}', \tilde{y}') \in Pairs_0$ with either 
    \begin{enumerate}[(A)]
    \item a sequence of self-glued trajectories arising from an interior zero of $\sigma'_{bord}(\tilde{z}, \pi\tilde{z})$ with $(\tilde{z}, \pi \tilde{z}) \in Pairs_{-1})$, or 
    \item with a single zero of $\sigma'_{bord}(\tilde{x}'', \tilde{y}'')$ for some pair $(\tilde{x}'', \tilde{y}'') \in Pairs_{-1}$
    \end{enumerate}

    We can consider the restrictions of these perturbations $\sigma'_{bord}|_t$ to $T(\tilde{x}, \tilde{y})|_t$. By perturbing $\sigma'_{bord}(\tilde{x}, \tilde{y})$ on the interior of $T(\tilde{x}, \tilde{y})$ for all $(\tilde{x}, \tilde{y})$, we can arrange for the restrictions $\sigma'_{bord}(\tilde{x}, \tilde{y})|_{t}$ to be FOP-transverse on $T(\tilde{x}, \tilde{y})$ for all $(\tilde{x}, \tilde{y}) \in Pairs_0$ for all $t \in [0,1]$ such that $t$ is not in some countable set $T_{bad} \subset [0,1]$. For any $t \notin T_{bad}$, we can then apply the induction of Proposition \ref{prop:transverse-perturbations-exist-inductively} to $\CC'_{bord}|_{t}$ to define $|\CC'_{bord}|_t(\sigma'_{bord}|_t)|$; note, however, that this complex is entirely determined by the choices of $\sigma'_{bord}(\tilde{x}, \tilde{y})$ for $(\tilde{x}, \tilde{y}) \in Pairs(0)$ made previously, and this step is only needed to verify that the putative differential squares to zero.
    
    The elements of $T_{bad}$ consist of those $t \in [0,1]$ such that there exists a $(\tilde{x}, \tilde{y}) \in Pairs_{-1}$ and a $u \in \sigma'_{bord}(\tilde{x}, \tilde{y})^{-1}(0)$ with $p(u) = t$; or those for which $\sigma'_{bord}(\tilde{x}, \tilde{y})|_{t}$ is not FOP-transverse at some interior point for some $(\tilde{x}, \tilde{y}) \in Pairs_0$; and only one of these alternatives occurs for any given $t \in T_{bad}$.  By $E$-properness of $\CC_{bord}$, to compute the differential up to $\Lambda^{E_0}$ it suffices to analyze finitely-many moduli spaces. Thus, for every $E_0$ we can compare terms in the differential associated to $\mathcal{C}_{bord}|_{t'}$ up to order $T^{E_0}$ by analyzing a finite number of moduli spaces; in particular, there is an $\epsilon$ such that the terms in the differential up to order $T^{E_0}$ all have the same value for $t' \in [t-\epsilon, t)$ and a different value for $t' \in (t, t+\epsilon]$. To understand this change in the chain complex corresponding to these $t \in T_{bad}$, it suffices to study the manifolds
    \begin{equation}
    R(\tilde{x}, \tilde{y}; t-\epsilon, t+\epsilon) = \{u \in \sigma'_{bord}(\tilde{x}, \tilde{y})^{-1}(0) | p(y) \in [t-\epsilon, t + \epsilon])\}
    \end{equation}
    as they control  whether the number of points of $\sigma'_{bord}(\tilde{x}, \tilde{y})_{t - \epsilon}^{-1}(0)$ agree or not  with $\sigma'_{bord}(\tilde{x}, \tilde{y})_{t + \epsilon}^{-1}(0)$ for sufficiently small $\epsilon$. But the second kind of point $t \in T_{bad}$ in the classification at the beginning of this paragraph does not change the number of points of $\sigma'_{bord}(\tilde{x}, \tilde{y})|_{t \pm \epsilon}$ since $R(\tilde{x}, \tilde{y}; t-\epsilon, t+\epsilon)$ gives a bordism between the points on the two sides. 
    
    The other kind of points $t \in T_{bad}$ make it possible for one or neither of the alternatives (A), (B) above to occur. If neither occurs then clearly $R(\tilde{x}, \tilde{y}; t-\epsilon, t+\epsilon)$ is still a bordism and the counts are unchanged. Otherwise, the cases (A), (B) corresponds respectively to the cases (d), (c) of the finite dimensional model, and one can argue as in  \cite{Hutchings05}. In case (B), we must either have $\tilde{x}' = \tilde{x}$, $\tilde{y}' = \tilde{x}''$ and $\tilde{y}'' = \tilde{y}$; or $\tilde{x}'' = \tilde{x}$, $\tilde{y}'' = \tilde{x}'$, and $\tilde{y}' = \tilde{y}$. Let us assume the first case without loss of generality; then as $t'$ increases through $t$, orientation computations as in \cite[(g)]{laudenbach} (see \cite[Lemma~3.4]{hutchings-reidemeister-torsion}) show that the number of zeros of $\sigma'_{bord}(\tilde{x}, \tilde{y})|_{t'}$ either increases or decreases by the number of zeros of $\sigma'_{bord}(\tilde{x}, \tilde{z})|_t$ in a way corresponding to the basis change $\tilde{z} \mapsto \tilde{z} \pm \tilde{y}$. Reversing the roles of $(\tilde{x}, \tilde{y})$ and $(\tilde{y}, \tilde{z})$ shows that for this bifurcation, for any fixed $E_0$ there is an $\epsilon$ such that terms in the diferential up to order $T^{E_0}$ change as in case (c) of the finite dimensional case as $t'$ passes from $t'-\epsilon$ to $t'+\epsilon$. 
	
    To deal with the subcase of (A) above corresponding to the finite-dimensional case (d), we note that the only obstruction to the above argument is that the $\Pi$-equivariance of $\sigma'_{bord}$, which makes the non-transverse self-gluings of (A) possible. However, as in \cite{hutchings-reidemeister-torsion} we break $\Pi$-equivariance of the perturbation $\sigma'_{bord}$ in a small region of $t' \in [t-\epsilon, t+\epsilon]$, instead (say) making the perturbation equivariant for $N\Pi \subset \Pi$ for some large $N$; for sufficiently large $N$ this reduces cases (A) to case (B) and one sees that this corresponds to performing a basis change like $\tilde{x} \mapsto (1 \pm \pi \pm O(\pi^2)) \tilde{x}$ in the sense that for any fixed $E_0$ there is an $\epsilon>0 $ such that terms in the differential up to order $T^{E_{0}}$ will undergo precisely this change as $t'$ passes from $t'-\epsilon$ to $t'+\epsilon$. 
    
    As the Novikov ring is complete, we can conclude that the homotopy-equivalence $|\CC(\sigma'_0)|_{\Lambda^{\Pi_E}_\Z} \to |\CC(\sigma'_1)|_{\Lambda^{\Pi_E}_\Z}$ can be constructed from an infinite convergent composition of basis changes corresponding to cases (A) and (B) above.

\end{proof}

    Using the above lemmata, we can prove
\begin{proposition}
\label{prop:qis-type-indep-of-perturbation}
	Suppose that $\CC(\tilde{x}, \tilde{y}) \neq \emptyset$ and $\tilde{x} \neq \tilde{y}$ implies that $E(\tilde{x}) - E(\tilde{y}) \leq 0$. Then the homotopy type of $|\CC(\sigma')|_{\Lambda^{\Pi_E}_\Z}$ only depends on the underlying compatible system of derived $\langle k \rangle$-orbifold charts. 
\end{proposition}

\begin{proof} 
Let $\CC_{bord}$ be the ``trivial bordism of flow categories on $\CC$''; this has $Ob(\CC_{bord}) = Ob(\CC)$, and $\CC_{bord}(\tilde{x}, \tilde{y}) = [0,1] \times \CC(\tilde{x}, \tilde{y})$ whenever $\tilde{x} \neq \tilde{y}$ and $CC(\tilde{x}, \tilde{y}) \neq \emptyset$, and the map $p: \CC_{bord}(\tilde{x}, \tilde{y}) \to [0,1]$ is simply the projection. We inherit the $\Pi$-action, grading, and energy function from $\CC$. We apply Proposition \ref{prop:novikov-bifurcation-continuation} $\CC'_{bord}$, the trivial bordism of derived $\langle k \rangle$-orbifold charts on $\CC_{bord}$, which has $\CC'_{bord}(\tilde{x}, \tilde{y}) = ([0,1] \times T(\tilde{x}, \tilde{y}), [0,1] \times V(\tilde{x}, \tilde{y}), \sigma_{bord}(\tilde{x}, \tilde{y})$ where $\tilde{x} \neq \tilde{y}$ and $\CC'(\tilde{x},\tilde{y}) = (T(\tilde{x}, \tilde{y}), V(\tilde{x}, \tilde{y}), \sigma(\tilde{x}, \tilde{y}))$ with $\sigma_{bord}(\tilde{x}, \tilde{y})(t, x) = \sigma(\tilde{x}, \tilde{y})(x)$. 
\end{proof}

\begin{proposition}
\label{prop:independence-of-integralization}
    The quasi-isomorphism type of $CF^{\bullet}$ does not depend on the choices of  $\tilde{\Omega}$, $N_1$, $h$ or the perturbation data and stabilizing vector bundles used to define it.
\end{proposition}  
\begin{proof} (Sketch).

This argument is a coherent, categorical, and slightly clarified version of \cite[Proposition~6.31]{AMS}.

There are spaces $\mathcal{F}(\CP^{\ell_1}, d_1; \CP^{\ell_2}, d_2)$ for $\ell_1 \geq d_1$ and $\ell_2 \geq d_2$ consisting of stable maps with two marked points $p_\pm$ on their domains, with target $\CP^{\ell_1} \times \CP^{\ell_2}$ such that the projected maps to each factor are also stable, and such that the domains are equipped with geodesic markings as in Section \ref{sec:moduli-of-decorated-stable-maps}.
There is a pair of maps $f_i: \mathcal{F}^{dbl}(\CP^{\ell_1}, d_1; \CP^{\ell_2}, d_2) \to \mathcal{F}(\CP^{\ell_i}, d_i)$ for $i=1,2$, which send a stable map to its two projections. Notice, however, that the product map $f_1 \times f_2$ is not an inclusion: given a pair of stable maps with isomorphic domain types, one must choose an \emph{isomorphism of the domains} to choose an element of the fiber of $f_1 \times f_2$. The universal curve over $\mathcal{F}(\CP^{d_1}, d_1; \CP^{d_2}, d_2)$ is denoted $\mathcal{C}(\CP^{d_1}, d_1; \CP^{d_2}, d_2)$.

Now, $\mathcal{F}(\CP^{\ell_1}, d_1; \CP^{\ell_2}, d_2)$ is a $2^{[d_1-1] \times [d_2-1]}$-stratified space with stratification as follows.  Given  $S \subset [d_1 -1] \times [d_2-1]$, we write $S_1 \subset [d_1 -1]$ and $S_2 \subset [d_2-1]$ for the projections. We say that $\mathcal{F}(\CP^{\ell_1}, d_1; \CP^{\ell_2}, d_2)(S) = \emptyset$ if, writing $S = \{(n_1, m_1), (n_2, m_2), \ldots, (n_r, m_r)\}$ with $n_1 < n_2 < \ldots < n_r$, we do not have that $m_1 < m_2 \ldots < m_r$. If this condition is satisfied we set $\mathcal{F}(\CP^{\ell_1}, d_1; \CP^{\ell_2}, d_2)(S) = f_1^{-1}(S_1) \cap f_2^{-1}(S_2)$. By a construction analogous to Lemma \ref{lemma:equip-mclean-moduli-spaces-with-k-manifold-structure}, these spaces are smooth $(PU(\ell_1+1) \times PU(\ell_2+1))-\langle [d_1-1]\times[d_2-1]\rangle$-manifolds. 
There are analogous spaces $\mathcal{F}(\CP^{\ell_1}, d_1; \CP^{\ell_2}, d_2; p, \pm)$ for which the image of $p_\pm$ must go through a given point $p \in \CP^{\ell_1} \times \CP^{\ell_2}$. These are restrictions of $\mathcal{F}(\CP^{\ell_1}, d_1; \CP^{\ell_2}, d_2)$ to a $(U(\ell_1) \times U(\ell_2))$-space along the inclusion $U(\ell_1) \times U(\ell_2) \subset PU(\ell_1+1) \times PU(\ell_2+1)$. 

Write $d_1 = d'_1 + d''_1$ and $d_2 = d'_2 + d''_2$, choose $p_1 \in \CP^{d'_1} \subset \CP^{d_1}$ and $p_2 \in \CP^{d'_1} \subset \CP^{d_2}$, and set $p = (p_1, p_2)$. Let $\CP^{d''_i} \subset \CP^{d_i}$ be the complementary subspaces to $\CP^{d'_i} \subset \CP^{d_i}$ in the sense of Section \ref{}. 
Given $S_1 = \{n_1 < \ldots < n_r\} \subset [d_1-1]$ and $S_2 = \{m_1 < \ldots < m_r\} \subset [d_2-1]$ we define $S_1 . S_2 = \{(n_1, m_1), \ldots, (n_r, m_r)\}$. We further define
\begin{equation}
    \mathcal{F}^{tr}_{S(d'_1, d''_1).S(d'_2, d''_2)}(\CP^{\ell_1}, d_1; \CP^{\ell_2}, d_2) \subset \mathcal{F}(\CP^{\ell_1}, d_1; \CP^{\ell_2}, d_2)
\end{equation} analogously to the definition of $\mathcal{F}^{tr}_{S(d_1, d_2)}(\CP^{d}, d)$ in \eqref{eq:split-mclean-domain-subset} as maps from a domain with two components where the map from the first component has degree $(d'_1, d'_2)$ and has image lying in $\CP^{d'_1} \times \CP^{d'_2}$, and the nodal point joining the components maps to $p$. We have have that $\mathcal{F}(\CP^{\ell_1}, d_1; \CP^{\ell_2}, d_2)(S(d'_1, d''_1).S(d'_2, d''_2))$ is an induction of $\mathcal{F}^{dbl, tr}_{S(d'_1, d''_1).S(d'_2, d''_2)}(\CP^{\ell_1}, d_1; \CP^{\ell_2}, d_2) $ along $(U(d'_1) \times U(d''_1)) \times (U(d'_2) \times U(d''_2)) \subset PU(\ell_1+1) \times PU(\ell_2+1)$, and that $\mathcal{F}^{tr}_{S(d'_1, d''_1).S(d'_2, d''_2)}(\CP^{d_1}, d_1; \CP^{d_2}, d_2)$ is a generalized induction of 

$\mathcal{F}(\CP^{d'_1}, d'_1; \CP^{d'_2}, d'_2; p, +) \times \mathcal{F}(\CP^{d''_1}, d''_1; \CP^{d''_2}, d'_2; p, -)$ along $(V^{tr}_{d, d_1})^2 \subset PGL(d_1, \C)^2$ after modifying the set labeling the stratification by applying the inclusion $[d'_1-1] \times [d'_2-1] \sqcup [d''_1-1] \times [d''_2-1] \to [d_1 -1] \times [d_2 - 2]$ given by including elements of $[d'_1-1] \times [d'_2-1]$ as is and adding $(d'_1, d'_2)$ to elements of $ [d''_1-1] \times [d''_2-1]$. 

Thus the analog of Lemma \ref{lemma:induction-relation-between-strata} with the the new spaces $\mathcal{F}(\CP^d, d; \CP^{d}, d)$ holds. 

Recall that we have two different compatible choices of peturbation data peturbation data $\Lambda^1_d = (V^1_d, \lambda^1_d)$  and $\Lambda^2_d= (V^2_d, \lambda^2_d)$. We also have a pair of choices $(\tilde{\Omega}^1, N^1_1, h^1)$ and $(\tilde{\Omega}^2, N^2_1, h^2)$. 

We define $\mathcal{T}^{dbl, bord}$ to consist of tuples $(t, (u, \Sigma, F_1, F_2), \eta_1, \eta_2)$ were $t \in [0,1]$, $\Sigma$ is a domain with two marked points $p_\pm$ and geodesic markings as before,  $u: \Sigma \to M$ is a map, $(u, \Sigma, F_i)$ are framed bubbled pre-Floer trajectories defined with respect to $\tilde{\Omega}^i, N^1_i, h^i$, (of corresponding degrees $d_1, d_2$) and $\eta_i \in (V^i_d)_{u_{F_i}}$, satisfying

\begin{equation}
    \bar{\partial}_F u + t\lambda_d^1(\eta_1) \circ di_{F_1} + (1-t)\lambda_d^2(\eta_2) \circ di_{F_2} = 0,
\end{equation}
up to equivalence. An equivalence of such data only exist when the data $t\in [0,1]$ agree, and in that case an equivalence is an equivalence o the \emph{doubly framed bubbled pre-Floer trajectories} $(u, \Sigma, F_1, F_2)$, namely isomorphism between domains intertwining the markings and the maps $u$, together with isomorphisms of the line bundles $L_{\Omega_u^1}$ and $L_{\Omega_u}^2$ over isomorphisms of domains which send the bases $(F_1, F_2)$ to the new bases. We define $\mathcal{T}^{dbl, bord}_{\Lambda_d, \epsilon}$ to be the subset with $\|\eta_1\|< \epsilon$ and $\|\eta_2\|< \epsilon$.

The set of equivalence classes of doubly framed bubbled pre-Floer trajectories can be viewed as a subset of the set of choices of pairs of integers $(d_1, d_2)$, an element $u_{F_1, F_2} \in \mathcal{F}(\CP^{d_1}, d_1; \CP^{d_2}, d_2)$, and a smooth map from the fiber of the universal curve over this space satisfying several conditions. As before we can equip this set with the topology induced from the Hausdorff topology on the closures of these graphs in $\mathcal{C}(\CP^{d_1}, d_1; \CP^{d_2}, d_2) \times M$. 

Over $\mathcal{T}^{dbl, bord}_{\Lambda_d}$ there is a vector bundle with fiber over $(t, (u, \Sigma, F_1, F_2), \eta_1, \eta_2)$ given by $(V^1_d)_{u_{F_1}} \oplus i \mathfrak{pu}_{d_1+1} \oplus (V^1_d)_{u_{F_2}} \oplus i \mathfrak{pu}_{d_2+1}$. This vector bundle has a section $\sigma = \sigma_{\Lambda^1_{d_1}} \oplus \sigma^1_H \oplus \sigma_{\Lambda^2_{d_2}} \oplus \sigma^2_H)$ where $\sigma^i_H$ is given by $((u, \Sigma, F_1, F_2), \eta_1, \eta_2) \mapsto H_i(u, \Sigma, u_{F_i})$, where $H_i$ is $H$ defined using $\Omega_i$ for $i=1,2$. 

The group $PU(d_1+1) \times PU(d_2+1)$ acts on everything in sight, and we have that $\M(d_1, d_2) = \sigma^{-1}(0)/PU(d_1+1) \times PU(d_2+1)$. For sufficiently small $\epsilon, \mathcal{T}^{dbl, bord}_{\Lambda_d}$ is a $\langle 1 + (d_1-1) \times(d_2-1) \rangle$-manifold because adding $\eta_1$ or $\eta_2$ individually to the moduli problem was already enough to make the linearized Cauchy-Riemann operators surjective. As in Lemma \ref{lemma:bounded-stabilizers}, the stailizer groups are finite, and the argument of Lemma \ref{lemma:induced-topology} adapt to this context essentially without change, and we thus have that $\mathcal{K}^{dbl, bord} = (T,V, \sigma)$ are topological Kuranishi charts on 
\begin{equation}
\cM(d_1, d_2) = \bigcup_{\substack{\tilde{x}_\pm\in \widetilde{Fix}(H) \\ \bar{\mathcal{A}}^i(\tilde{x}_-) - \bar{\mathcal{A}}^i(\tilde{x}_+) = d_i \text{ for } i=1, 2}} \cM(\tilde{x}_-, \tilde{x}_+) \times [0,1].
\end{equation}

Restricting these charts to $t=0$ or $t=1$ manifestly makes them into inductions of the Kuranishi charts constructed using $(\tilde{\Omega}^i, N^1_i, h^i)$ for $i=2,1$, respectively, as in 
\cite[Proposition~6.31]{AMS}. 
    
The methods of Section 3 and \ref{sec:compatible-smoothing} allow one to use these to define a compatible system of derived $\langle k \rangle$-orbifold charts $\mathcal{C}'(H, J)_{bord}$ on the trivial bordism of flow categories $\mathcal{C}(H, J)_{bord}$, in such a way that this gives a bordism of derived $\langle k \rangle$-orbifold charts between the previously chosen systems of derived $\langle k \rangle$-orbifold  charts $\CC'(H, J)_1$ and $\CC'(H, J)|_2$ associated to $(\tilde{\Omega}^i, N^1_i, h^i)$ for $i=1,2$, respectively. We conclude by applying Propositions  \ref{prop:novikov-bifurcation-continuation} and \ref{prop:qis-type-indep-of-perturbation}.

\end{proof}

\printbibliography

\end{document}